\documentclass[11pt,leqno,oneside]{article}

\usepackage[utf8]{inputenc} %
\usepackage[T1]{fontenc}    %
\usepackage{microtype} 
\usepackage[english]{babel} %
\usepackage[lcgreekalpha]{stix2}          %

\usepackage[hyphens]{url}
\usepackage[a4paper,margin=1in,marginpar=1in]{geometry} %
\usepackage{xargs}          %
\usepackage{csquotes}       %

\usepackage[noblocks]{authblk}  %

\usepackage[shortlabels]{enumitem}

\usepackage{amsmath}   %
\usepackage{amsthm}    %
\usepackage{mathtools} %
\def\mapsto{\DOTSB\mathchar"39AD }

\usepackage[
backend=biber,
style=alphabetic,
giveninits=true,
maxalphanames=5,
minalphanames=5,
maxbibnames=99
]{biblatex}

\usepackage{xcolor}                   %
\definecolor{green}{HTML}{2ECC71}
\definecolor{blue}{HTML}{3498DB}
\definecolor{red}{HTML}{E74C3C}
\definecolor{orange}{HTML}{FD6A02}
\usepackage[pdfpagelabels]{hyperref}  %
\hypersetup{%
  colorlinks,
  linktocpage,
  urlcolor  = blue,
  citecolor = green,
linkcolor = red}

\usepackage[capitalise,sort]{cleveref}

\AfterEndEnvironment{proof}{\noindent\ignorespaces} %
\makeatletter
\def\@endtheorem{\endtrivlist}
\makeatother
\Crefname{equation}{}{}
\Crefname{enumi}{}{}
\Crefname{conditioni}{Condition}{Conditions}
\Crefname{conditionalti}{Condition}{Conditions}
\newtheorem{theorem}{Theorem}[section]
\newtheorem*{theorem*}{Theorem}
\Crefname{theorem}{Theorem}{Theorems}
\newtheorem{lemma}[theorem]{Lemma}
\Crefname{lemma}{Lemma}{Lemmas}
\newtheorem{proposition}[theorem]{Proposition}
\Crefname{proposition}{Proposition}{Propositions}
\newtheorem{corollary}[theorem]{Corollary}
\Crefname{corollary}{Corollary}{Corollaries}

\Crefname{conjecture}{Conjecture}{Conjectures}

\Crefname{assumption}{Assumption}{Assumptions}
\theoremstyle{definition}
\newtheorem{definition}[theorem]{Definition}
\Crefname{definition}{Definition}{Definitions}

\Crefname{question}{Question}{Questions}
\theoremstyle{remark}
\newtheorem{remark}[theorem]{Remark}
\Crefname{remark}{Remark}{Remarks}

\Crefname{example}{Example}{Examples}
\newtheorem*{example*}{Example}

\numberwithin{equation}{section}

\usepackage{booktabs}

\usepackage{transparent}

\newcommand{\nat}{\mathbb{N}}
\newcommand{\real}{\mathbb{R}}
\newcommand{\rational}{\mathbb{Q}}

\DeclareMathOperator{\law}{law}

\DeclareMathOperator{\e}{e}
\newcommand{\dd}{\mathrm{d}}

\DeclarePairedDelimiter{\paren}{\lparen}{\rparen}
\DeclarePairedDelimiter{\bracket}{\lbrack}{\rbrack}
\DeclarePairedDelimiter{\set}{\lbrace}{\rbrace}
\DeclarePairedDelimiter{\abs}{\lvert}{\rvert}

\DeclarePairedDelimiter{\norm}{\lVert}{\rVert}
\DeclarePairedDelimiterX{\hsp}[2]{\langle}{\rangle}{#1, #2}
\DeclarePairedDelimiterX{\pairing}[2]{\langle}{\rangle}{#1 \mid #2}
\DeclarePairedDelimiterX{\hpairing}[2]{\lcurvyangle}{\rcurvyangle}{#1 \mid #2}

\DeclarePairedDelimiterXPP{\Exp}[1]{\exp}{\lparen}{\rparen}{}{#1}
\DeclarePairedDelimiterXPP{\Log}[1]{\log}{\lparen}{\rparen}{}{#1}
\DeclarePairedDelimiterXPP{\Inf}[1]{\inf}{\lbrace}{\rbrace}{}{#1}
\DeclarePairedDelimiterXPP{\Sup}[1]{\sup}{\lbrace}{\rbrace}{}{#1}

\newcommand{\given}[1][]{%
  \nonscript\:#1\vert
  \allowbreak
  \nonscript\:
\mathopen{}}

\DeclareMathOperator{\prob}{\mathbb{P}}
\DeclareMathOperator{\esp}{\mathbb{E}}

\DeclarePairedDelimiterXPP{\Prob}[1]{\prob}[]{}{\renewcommand\given{\nonscript\:\delimsize\vert\nonscript\:\mathopen{}} #1}
\DeclarePairedDelimiterXPP{\Esp}[1]{\esp}[]{}{\renewcommand\given{\nonscript\:\delimsize\vert\nonscript\:\mathopen{}} #1}
\DeclarePairedDelimiterXPP{\Var}[1]{\esp}[]{}{\renewcommand\given{\nonscript\:\delimsize\vert\nonscript\:\mathopen{}} #1}

\DeclareMathOperator{\ent}{\mathcal{H}}
\DeclareMathOperator{\fish}{\mathcal{I}}

\DeclarePairedDelimiterXPP{\Ent}[1]{\ent}(){}{\renewcommand\given{\nonscript\:\delimsize\vert\nonscript\:\mathopen{}} #1}
\DeclarePairedDelimiterXPP{\Fish}[1]{\fish}(){}{\renewcommand\given{\nonscript\:\delimsize\vert\nonscript\:\mathopen{}} #1}

\newcommand{\ou}{\mathsf{P}} %
\newcommand{\dou}{\ou^{\star}} %
\newcommand{\gen}{\mathsf{L}} %
\newcommand{\diff}{\mathsf{D}} %
\newcommand{\skoro}{\mathsf{D}^{\star}} %

\DeclareMathOperator{\dom}{\mathscr{D}om}
\newcommand{\smooth}{\mathscr{S}}
\newcommand{\config}{\varUpsilon}
\newcommand{\sobolev}{\mathscr{W}}

\DeclarePairedDelimiterXPP{\EspPi}[1]{\esp_{\pi}}[]{}{\renewcommand\given{\nonscript\:\delimsize\vert\nonscript\:\mathopen{}} #1}
\DeclarePairedDelimiterXPP{\EspPim}[1]{\esp_{\pi \otimes m}}[]{}{\renewcommand\given{\nonscript\:\delimsize\vert\nonscript\:\mathopen{}} #1}
\DeclarePairedDelimiterXPP{\EspPimt}[1]{\esp_{\pi \otimes m \otimes \dd t}}[]{}{\renewcommand\given{\nonscript\:\delimsize\vert\nonscript\:\mathopen{}} #1}

\newcommand{\emparg}{\,\cdot\,} %

\author{Lorenzo DELLO SCHIAVO}
\affil{%
  Institute of Science and Technology Austria
  \par
  Am Campus 1, 3400 Klosterneuburg, Austria%
  \par
  \normalfont\href{mailto:lorenzo.delloschiavo@ist.ac.at}{\texttt{lorenzo.delloschiavo@ist.ac.at}}%
}
\author{Ronan HERRY}
\affil{%
  IRMAR, Université de Rennes 1
  \par
  263 avenue du Général Leclerc, 35042 Rennes Cedex
  \par
  \normalfont\href{mailto:ronan.herry@univ-rennes1.fr}{\texttt{ronan.herry@univ-rennes1.fr}}%
}
\author{Kohei SUZUKI}
\affil{%
Department of Mathematical Science, Durham University,
\par
Science Laboratories, South Road, DH1 3LE, United Kingdom
\par
\normalfont\href{mailto:kohei.suzuki@durham.ac.uk}{\texttt{kohei.suzuki@durham.ac.uk}}%
}

\DeclareSourcemap{
  \maps[datatype=bibtex]{
    \map{
      \step[fieldsource=doi,final]
      \step[fieldset=url,null]
    }  
  }
}

\AtEveryBibitem{%
  \clearfield{note}%
  \clearlist{language}%
}

\begin{document}
\title{Wasserstein geometry and Ricci curvature bounds\\for Poisson spaces}

\maketitle

\begin{abstract} 
  Let $\config$ be the configuration space over a complete and separable metric base space, endowed with the Poisson measure $\pi$.
  We study the geometry of $\config$ from the point of view of optimal transport and Ricci-lower bounds.
  To do so, we define a formal Riemannian structure on $\mathscr{P}_{1}(\config)$, the space of probability measures over $\config$ with finite first moment, and we construct an extended distance $\mathcal{W}$ on $\mathscr{P}_{1}(\config)$.
  The distance $\mathcal{W}$ corresponds, in our setting, to the Benamou--Brenier variational formulation of the Wasserstein distance.
  Our main technical tool is a non-local continuity equation defined via the difference operator on the Poisson space.
  We show that the closure of the domain of the relative entropy is a complete geodesic space, when endowed with $\mathcal{W}$.
  We establish non-local infinite-dimensional analogues of results regarding the geometry of the Wasserstein space over a metric measure space with synthetic Ricci curvature bounded below.
  In particular, we obtain that:
  \begin{itemize}[wide]
    \item the Ornstein--Uhlenbeck semi-group is the gradient flow of the relative entropy;
    \item the Poisson space has a Ricci curvature, in the entropic sense, bounded below by $1$;
    \item the distance $\mathcal{W}$ satisfies an HWI inequality.
  \end{itemize}
\end{abstract}
\tableofcontents
\newpage

\section{Introduction}
The theory of optimal transportation, and in particular the \emph{Wasserstein geometry}, plays a prominent role in the study of the geometry of metric measure spaces and of functional inequalities on them.
For instance, the seminal contributions \cite{Sturm,LottVillani,AGSDuke} establish a synthetic theory of Ricci curvature lower bounds for metric measure spaces, subsuming and extending the classical theory on smooth Riemannian manifolds; see, for instance, \cite[Part III]{VillaniOldAndNew} for a broad introduction to this topic.

Later developments extend this approach to various settings, including finite spaces equipped with a discrete distance.
In this case, \cite{Maas,Mielke} provide a fundamental intuition regarding the generalization of the Benamou--Brenier dynamical formulation of the $W_{2}$ transport distance to discrete spaces, where there is no geodesic associated with $W_{2}$.

Following the above line of research, in this paper we develop a Wasserstein geometry on configuration spaces, which are prototypical infinite-dimensional non-local spaces.
In particular, our work establishes that the configuration sapce equipped with the Poisson measure has Ricci curvature bounded from below by $1$, in a synthetic sense.

\subsection{Main results}
The configuration space $\config$ over a metric space $X$ is the set of non-negative Borel measures on $X$ that are integer-valued on balls.
Provided $X$ is equipped with a $\sigma$-finite measure $m$, the Poisson measure $\pi$ with intensity $m$, e.g.~\cite[Ch.\ 3]{LastPenrose}, is a canonical reference probability measure on $\config$ .
In this paper, we construct a distance $\mathcal{W}$ on $\mathscr{P}_{1}(\config)$, the space of probability measures over $\config$ with finite first moment (see \cref{s:locally-integrable} for definitions).
The geometric properties of $(\mathscr{P}_{1}(\config), \mathcal{W})$ account for synthetic Ricci-curvature lower bounds associated with $(\config, \pi)$.
To state our result, we consider the \emph{Ornstein--Uhlenbeck} semi-group $\mathsf{P} = \set{ \mathsf{P}_{t} : t \geq 0 }$ which plays the role of the heat semi-group in our setting, as well as its dual semi-group $\mathsf{P}^{\star} = \set{ \mathsf{P}^{\star}_{t} : t \geq 0 }$ acting on measures (see \cref{s:ou} for definitions and details).
Let us also write $\Ent{\,\cdot \given \pi}$ for the relative entropy with respect to $\pi$, and $\dom \ent \coloneq \set{ \mu \in \mathscr{P}(\config) : \Ent{ \mu \given \pi } < \infty }$.
\begin{theorem*}
The distance $\mathcal{W}$ satisfies the following properties:
\begin{itemize}[wide]
\item (\cref{t:W})
  the space $(\mathscr{P}_{1}(\config),\mathcal{W})$ is a complete geodesic extended-metric space.
\item (\cref{t:talagrand}) $\mathcal{W}$ satisfies the \emph{Talagrand inequality}
\begin{equation*}
  \mathcal{W}^{2}(\mu, \pi) \leq \Ent{ \mu \given \pi}, \qquad \mu \in \mathscr{P}_{1}(\config).
\end{equation*}
\end{itemize}

Furthermore, the non-extended metric space $(\dom \ent, \mathcal{W})$ captures the Ricci-curvature lower bounds of $(\config, \pi)$ in the following sense:
\begin{itemize}[wide]
\item (\cref{t:dou-contractivity}) 
  The dual semi-group $\dou$ exponentially contracts $\mathcal{W}$ with rate $1$:
\begin{equation*}
\mathcal{W}(\dou_{t} \mu_{0}, \dou_{t} \mu_{1}) \leq \e^{-t} \mathcal{W}(\mu_{0}, \mu_{1}), \qquad  t \geq 0, \qquad \mu_0,\, \mu_1 \dom \ent.
\end{equation*}

\item (\cref{t:evi}) The Ornstein--Uhlenbeck semi-group satisfies an \emph{Evolution Variation Inequality}
  \begin{equation}\label{e:evi-intro}\tag{EVI}
    \Ent{\dou_{s} \mu \given \pi} + \frac{1}{2} \frac{\dd}{\dd s} \mathcal{W}^{2}(\dou_{s} \mu, \xi) + \frac{1}{2} \mathcal{W}^{2}(\dou_{s} \mu, \xi) \leq \Ent{\xi \given \pi}, \qquad s \geq 0, \qquad \mu,\,\xi \in \dom \ent.
  \end{equation}

\item (\cref{t:entropy-geodesically-convex})
  The relative entropy is $1$-geodesically convex on $\dom \ent$ with respect to $\mathcal{W}$.

\item  (\cref{t:hwi})
  The relative entropy $\mathcal{H}$, the distance $\mathcal{W}$, and the Fisher information $\mathcal{I}$ satisfy the HWI inequality 
\begin{equation*}
  \Ent{ \mu \given \pi } \leq \mathcal{W}(\mu, \pi) \sqrt{\Fish{ \mu \given \pi }} - \frac{1}{2} \mathcal{W}^{2}(\mu, \pi), \qquad \mu \in \dom \ent.
  \end{equation*}
\end{itemize}
\end{theorem*}

\begin{remark}
  On manifolds, the contraction of the heat semi-group with respect to the Wasserstein distance, the convexity of the relative entropy with respect to Wasserstein geodesic, and the EVI-gradient flow are all equivalent to have a Ricci curvature bounded from below.
  They do not coincide in our infinite-dimensional non-local setting.
\end{remark}

\subsection{Summary of our construction}
We construct the distance $\mathcal{W}$ on $\mathscr{P}_{1}(\config)$, the space of all probability measures on $\config$ with locally finite intensity (see \cref{d:P1} below).
The \emph{discrete difference operator} on functions $F\colon\config\to\mathbb{R}$ is
\begin{align*}
\diff F\colon \config\times X\ni (\eta,x) \mapsto \diff_x F(\eta) \coloneq F(\eta+\delta_x)-F(\eta),
\end{align*}
and we denote by $\skoro$ its formal adjoint, called \emph{Skorokhod divergence}.

On $\mathscr{P}_{1}(\config)$, we consider a formal Riemannian structure induced by $\skoro$ and by the Poisson measure~$\pi$, together with the corresponding intrinsic distance à la Benamou--Brenier.
Precisely, for a curve $\bar\mu= \set{ \mu_{t} : t \in [0,1]}$ of absolutely continuous measures with $\mu_t=\rho_t\pi\in\mathscr{P}_{1}(\config)$, $t \in [0,1]$, and a curve of \emph{tangent vectors} $\bar{w}= \set{ w_{t} : t \in [0,1] }$ with $w_t\in L^1(\pi\otimes m)$, we informally say that the pair $(\bar{\mu},\bar{w})$ is a solution to the \emph{continuity equation} if
\begin{equation}\label{e:intro-ce}
  \partial_{t} \rho_{t} + \skoro(w_{t} \hat{\rho}_{t}) = 0, \qquad t\in [0,1].
\end{equation}
Here $\hat{\rho}_t$ is a tangent vector built from $\rho_t$, accounting for the non-locality of $\diff$ (see below for precise definitions).
We endow $\mathscr{P}_{1}(\config)$ with the \emph{dynamical transport distance} $\mathcal{W}$ defined by
\begin{equation*}
  \mathcal{W}^{2}(\mu_{0}, \mu_{1}) = \inf \int_0^1 \norm{w_t}_{\mu_{t}}^2 \dd t, \qquad \mu_0,\mu_1\in\mathscr{P}_1(\config),
\end{equation*}
where the infimum runs over all solutions $(\bar{\mu},\bar{w})$ to~\cref{e:intro-ce} with $\bar{\mu}$ joining $\mu_{0}$ to $\mu_{1}$, and where we let
\begin{equation*}
\norm{w}_{\mu}^{2} \coloneq \int \abs{w(\eta, x)}^{2} \hat{\rho}(\eta, x) \pi(\dd \eta) m(\dd x) .
\end{equation*}

This distance $\mathcal{W}$ is \emph{extended}, meaning that it may take the value $+\infty$.
However, in view of the Talagrand inequality, it is finite on $\dom \ent$.
Restricting our attention to the $\mathcal W$-closure $\mathscr{P}_{1}^{*}(\config)$ of the domain of the relative entropy, we see that $(\mathscr{P}_{1}^{*}(\config), \mathcal{W})$ is a complete non-extended geodesic space.
We actually esyablish our functional inequalities on $\mathscr{P}_{1}^{*}(\config)$.

\subsection{Motivation}
Developing a theory of optimal transport in the setting of the Poisson space $(\config, \pi)$, and understanding the \emph{curvature} of this space  from the point of view of the theory of synthetic Ricci curvature bounds serve as our main guidelines.
Classically, the theory of synthetic Ricci curvature bounds comes in two flavours:
\begin{enumerate}[(i), wide]
  \item The Bakry--Émery theory \cite{BakryEmery,BGL}, also referred to as the \emph{Eulerian formalism}, is concerned with a Markov semi-group $\mathsf{P} = (\mathsf{P}_{t})_{t \geq 0 }$.
    This theory characterizes Ricci-curvature lower bounds by a convexity-type inequality of the relative entropy along the semi-group.
    For diffusion semigroups, this convexity property is a consequence of the celebrated sub-commutation inequality between the semi-group and the associated carré du champ operator.
    In the case of the Poisson space, the canonical Markov semi-group is the Ornstein--Uhlenbeck semi-group and it is known that it satisfies a Bakry--Émery \cite[Lem.\ 6]{LastAnaSto}.
    Namely, we have that $\diff \mathsf{P}_{t} = \e^{-t} \mathsf{P}_{t} \diff$.
    However, due to the non-diffusive nature of the Ornstein--Uhlenbeck semi-group on the Poisson space, it is rather difficult to draw consequences of this property in this case.
    Nevertheless, \cite{Chafai} uses the Bakry--Émery commutation in order to derive a modified logarithmic Sobolev inequality for the Poisson measure (first obtained by \cite{Wu} with different methods).

  \item The Lott--Sturm--Villani theory \cite{Sturm,LottVillani,AGSDuke}, also referred to as the \emph{Lagrangian formalism}, is concerned with a metric measure space.
    It characterizes Ricci-curvature lower bounds by a convexity-type inequality of the relative entropy along the geodesics of optimal transport.
    Since there is no canonical distance on the configuration space, this far-reaching theory simply does not apply.
    The absence of a canonical distance is a typical feature of infinite-dimensional spaces.
\end{enumerate}

Despite several works (see below) extending the Lagrangian side of the theory for non-diffusive semi-groups or discrete spaces, a generalization of those techniques to non-local infinite-dimensional spaces, such as the configuration space, have so far remained out of reach. 
Our work tackles this issue and provides foundational tools for the development of a Wasserstein geometry and theory of Ricci curvature bounds for point processes on general state spaces with no assigned geometry.

\subsection{Related works}
\subsubsection{Entropic Ricci curvature for Markov chains and jump processes}
\cite{Maas,Mielke} and the subsequent works \cite{ErbarMaas,FathiMaas} initiated the study of optimal transport and Ricci-curvature bounds for non-local operators.
More precisely, they construct a transport distance, based on a non-local continuity equation, and study related functional inequalities for finite Markov chains.
This approach is partially generalized to jump processes on $\mathbb{R}^{n}$ in~\cite{Erbar}.

In particular, the idea of using an analogue of the Benamou--Brenier formulation involving a discrete continuity equation goes back to \cite{Maas}, while our definition of the Lagrangian, and the formulation of the continuity equation through a couple $(\bar{\mu}, \bar{\nu})$ is an adaptation to the Poisson setting of the ones in \cite{DNS09} generalizing the Benamou--Brenier formula in a continuous setting, and in \cite{Erbar} for jump processes on $\mathbb{R}^{d}$.
In the case of finite Markov chains on some space $E$, \cite{Maas} shows that the interior of $\mathscr{P}(E)$ endowed with $\mathcal{W}$ is a Riemannian manifold.
In this spirit, \cref{t:p1-complete-geodesic} identifies a non-trivial component of $\mathscr{P}_{1}(\config)$ on which $\mathcal{W}$ is a complete geodesic space.
No such identification appears in \cite{Erbar}.
In particular, the work \cite{Erbar} does not exclude that the topology generated by~$\mathcal{W}$ for jump processes is trivial.
Let us further note that Poisson random measures naturally appear in the study of L\'evy processes through their jump measures. It would therefore be interesting to know whether the results of \cite{Erbar} can be recast in our setting via this identification.

The recent work \cite{PRST} generalizes this non-local Benamou--Brenier approach to rather general jump processes.
However, the jump kernel of the Poisson process does not satisfy \cite[Assumption (3.4)]{PRST}.

\subsubsection{Other transportation costs for the configuration space}
\cite{GHP} studies optimal transport, more specifically, transport-entropy inequalities on the Poisson space.
There, N.\ Gozlan, G.\ Peccati and the second author circumvent the lack of canonical cost by considering a non-linear generalization of the classical optimal transport problem.
This generalized optimal transport is fully theorized in \cite{GRST}, and is particularly well suited to study discrete spaces \cite{GRSTCurvature}.
One of their main result \cite[Thm.\ 1.2]{GHP} is very close in spirit to our Talagrand inequality for $\mathcal{W}$ (\cref{t:talagrand}): they also obtain an upper bound of their transport cost $\mathbb{M}^{2}$ by the relative entropy.
However, at the time of writing, no dynamical Benamou--Brenier formulation for the generalized optimal transport of \cite{GRST} exists, and a comparison of those results seems out of reach.
Whether the transport cost of \cite{GHP} satisfies a displacement convexity inequality is an interesting question outside of the scope of the current paper.

\subsubsection{Other geometries on the configuration space}
  The configuration space over a Riemannian manifold $X$ may be endowed with a differential geometry lifted from that of the base Riemannian manifold.
  This geometry, defined and studied in \cite{AKR}, arises from the \emph{continuous difference operator}
\begin{align*}
  \mathbf{\nabla} F\colon \config \times X \ni (\eta,x)\longmapsto \nabla^z\big\vert_{z=x} \diff_{z}F(\eta),
\end{align*}
and the associated \emph{Dirichlet form}
\begin{equation*}
  \mathcal{E}(F) \coloneq \int_{\config} \int_{X} \norm{ \mathbf{\nabla}_{\eta}F(x) }_{T_{x}X}^{2} \eta(\dd x) \pi(\dd \eta).
\end{equation*}
The corresponding dynamic is that of the second quantization of the heat semi-group to the Poisson space \cite{Surgailis}; while the Ornstein--Uhlenbeck semi-group studied in this paper corresponds to the second quantization of the semi-group $\mathsf{P}_{t} f = \e^{-t} f$ for all $f \in \mathscr{F}(X)$ and $t \ge 0$.
\cite{RoecknerSchied} proves that this geometry corresponds to that of the extended metric measure space $(\config, W_{2})$, where $W_{2}$ is the Wasserstein $2$ transport distance with respect to the Riemannian distance.
Following \cite{ErbarHuesmann}, this geometry on $\config$ inherits both Ricci-curvature and Alexandrov-curvature lower bounds from the base space.
Two of the authors \cite{LzDSSuz21,LzDSSuz22} have recently generalized these results to a large class of metric measure spaces; while the third author also has proved analogous curvature bounds~\cite{Suz22} in the setting of Dyson Brownian motion.

This geometry differs from the one we consider throughout the rest of the paper.
For instance, the process associated to this differential geometry is a diffusion process; while the Ornstein--Uhlenbeck semi-group defines a jump process.
Our analysis on the Poisson space also holds without any geometric assumptions on the base space; while \cite{AKR,ErbarHuesmann} require that the space is a manifold with some geometric assumptions.

\subsubsection{Curvature of the Wiener space}
Together with Gaussian measures, Poisson random measures are ubiquitous in probability theory.
Among other common properties, they share the existence of an orthogonal systems of \enquote{chaoses}.
Consequently, they admit a \enquote{differential calculus}, known as the \emph{Malliavin calculus}, completely characterised by their probabilistic properties.
In particular, we expect the geometric and functional analytic results one can deduce from this differential calculus to be independent of properties of the underlying space.
In this regard, \cite{FangShaoSturm} derives synthetic Ricci-curvature lower bounds for infinite-dimensional Wiener spaces, equipped with a Gaussian measure, that are as good as the finite-dimensional ones.
Our result parallels theirs on the configuration space, equipped with a Poisson measure.
Let us however highlight two fundamental differences:
\begin{itemize}
  \item The generator of the Ornstein--Uhlenbeck process on the Gaussian space is diffusive; while our operator is purely non-local.
  \item The Wiener space comes naturally equipped with an extended distance, the so-called \emph{Cameron--Martin distance}, while their is no canonical distance on the configuration space.
\end{itemize}

\subsection{Outline of the paper}
Throughout the paper, we let $X$ be a complete and separable metric space.
\cref{s:reminders} recalls the necessary definitions regarding the \emph{configuration space} $\config$ over $X$, and establishes some topological results regarding the topology of point processes.
Of particular importance, we define the space $\mathscr{P}_{1}(\config)$ of point processes with finite first moment and we endow it with a Polish topology (\cref{t:p1-polish}).
We show that mapping a point process in $\mathscr{P}_{1}(\config)$ to its \emph{reduced Campbell measure} is an homeomorphism (\cref{t:C-homeomorphism}).
In \cref{s:ou}, we recall definitions regarding the \emph{Ornstein--Uhlenbeck semi-group} $\mathsf{P}$ as well as the \emph{difference operator} $\diff$ and their interactions with the \emph{relative entropy $\ent$} and the \emph{Fisher information $\fish$}.

In \cref{s:continuity-equation}, we give a precise formulation to the \emph{continuity equation} \cref{e:intro-ce}.
We show (\cref{t:ou-solution}) that the Ornstein--Uhlenbeck evolution is a solution to the continuity equation, and that every solution has a continuous representative (\cref{t:ce-continuous-representative}).
We also obtain a closed formula for the entropy production along solutions to the continuity equations (\cref{t:ce-entropy-production}).

In \cref{s:variational-distance}, we define and study the \emph{Lagrangian} $\mathcal{L}$ and the \emph{action} $\mathcal{A}$ that are necessary to obtain our \emph{transport distance} $\mathcal{W}$.
We also study a \emph{entropic regularization} $\mathcal{J}_{\varepsilon}$ of $\mathcal{W}$, that is of independent interest.
We first state several properties of the Lagrangian (\cref{t:lsc-lagrangian,t:convex-lagrangian,t:bound-nu-lagrangian}) necessary to apply the direct method of the calculus of variations in order to prove existence of minimizing curves.
We also establish in \cref{t:contraction-lagrangian} that the action of the Ornstein--Uhlenbeck semi-group contracts the Lagrangian.
We then define the action $\mathcal{A}$ and verify the existence of minimizers in the infimum.
In that regard, we establish the compactness of sub-level sets in \cref{t:compact-action}.
After defining the extended distance $\mathcal{W}$, we summarize its main properties in \cref{t:W}.

In \cref{s:gradient-flow}, we show that $\mathcal{W}$ is finite on the domain of $\mathcal{H}$.
The main tool is the Talagrand inequality (\cref{t:talagrand}) comparing $\mathcal{W}$ and $\mathcal{H}$.
We then establish in \cref{t:evi} one of the main result of this work: on the domain of $\mathcal{H}$ the Ornstein--Uhlenbeck semi-group is an EVI-gradient flow for the entropy.
From this follows several important consequences such as the geodesic convexity of the relative entropy in \cref{t:entropy-geodesically-convex} and the HWI inequality \cref{t:hwi}.

\subsection{Acknowledgments}
The authors are grateful to Masha Gordina, Takashi Kumagai, Laurent Saloff-Coste, Karl-Theodor Sturm, and the Mathematisches Forschungsinstitut Oberwolfach (MFO) for organizing the workshop \emph{Heat Kernels, Stochastic Processes and Functional Inequalities} (2019), where the authors started discussing this work. 

L.D.S.~gratefully acknowledges funding by the Austrian Science Fund (FWF) grant F65, and by the European Research Council (ERC, grant No.~716117, awarded to Prof.\ Dr.~Jan Maas).
He acknowledges funding of his current position by the Austrian Science Fund (FWF) through grant ESPRIT~208.

R.H.~gratefully acknowledges funding from Centre Henri Lebesgue.
Most of this research was carried out while R.H.\ was postdoc for the DFG through the project \emph{Random Riemannian Geometry} (initiated by Prof.\ Dr.\ Karl-Theodor Sturm and Dr.\ Eva Kopfer) within the SPP 2265 \emph{Random Geometric Systems}.

K.S.~gratefully acknowledges funding by: the JSPS Overseas Research Fellowships, Grant Nr. 290142; World Premier International Research Center Initiative (WPI), MEXT, Japan; JSPS Grant-in-Aid for Scientific Research on Innovative Areas \emph{Discrete Geometric Analysis for Materials Design}, Grant Number 17H06465; and the Alexander von Humboldt Stiftung, Humboldt-Forschungsstipendium.

\section{Topological results for point processes}\label{s:reminders}
\subsection{Topological preliminaries for spaces of functions and measures}

Given a measure $\mu$ on some measurable space, we write $\abs{\mu}$ for its \emph{variation}; and for a non-negative measurable or $\mu$-integrable functions $f$, we write $\mu(f) = \int f \dd \mu$ for the integral of $f$ with respect to~$\mu$.
We say that a locally convex topological vector space is complete if it is complete with respect to each of the seminorms defining its locally convex topology.

\subsubsection{The weak topology}
Given a topological space $(E,\tau)$, we write $\mathfrak{B}(E)$ for the Borel sets of $E$, and $\mathfrak{K}(E)$ for the Borel compact sets.
We write $\mathscr{F}_{b}(E)$ for the set of $\real$-valued bounded Borel functions, and $\mathscr{C}_{b}(E)$ for those that are bounded and continuous.
We write $\mathscr{P}(E)$ for the set of all Borel probability, and $\mathscr{M}_{b}(E)$ for the set of Borel finite signed measures on $E$.
For $B \in \mathfrak{B}(E)$ we define the \emph{evaluation map} $\iota_{B}\colon \lambda \mapsto \lambda(B)$ for every Borel measure $\lambda$.
For $F \in \mathscr{F}(E)$, we also write $\iota_{F}$ whenever this is well-defined.
For an event $B \in \mathfrak{B}(E)$, we also write $\mathfrak{B}_{B}(E)$ for the $\sigma$-algebra of events depending only on $B$.
More precisely, $\mathfrak{B}_{B}(X)$ is the $\sigma$-algebra of all $B' \in \mathfrak{B}(X)$ such that either $B' \subset B$ or $X \setminus B \subset B'$.
The spaces $\mathscr{F}_{b}(E)$ and $\mathscr{C}_{b}(E)$ are endowed with the \emph{uniform norm} under which they are Banach spaces.
Likewise, $\mathscr{P}(E)$ and $\mathscr{M}_{b}(E)$ are always endowed with the \emph{weak topology}, that is the initial topology associated with $\iota_{F}$, $F \in \mathscr{C}_{b}(E)$.

We also use the superscript $+$ to indicate a subset of non-negative functions or measures.
For instance, we write $\mathscr{M}_{b}^{+}(E)$ for the cone of non-negative finite Borel measures, $\mathscr{F}^{+}(E)$ for the non-negative Borel functions.

\subsubsection{The vague topology}

When $(E,d)$ is a metric space, we write $\mathfrak{B}_{0}(E)$ for the bounded measurable sets.
We write $\mathscr{F}_{0}(E)$ for the space of bounded measurable that vanish outside of a bounded set, and $\mathscr{C}_{0}(E)$ for those that are also continuous.

  Given a closed and bounded $B \subset E$, we write $\mathscr{C}_{b,B}(E)$ for the subspace of functions $f \in \mathscr{C}_{0}(E)$ vanishing outside of $B$; this set $\mathscr{C}_{b,B}(E)$ is equipped with the uniform norm, under which it is a Banach space.
  The set $\mathscr{C}_{0}(E)$ can be endowed with the inductive limit topology associated to the inclusions $\mathscr{C}_{b,E_{n}}(E) \to \mathscr{C}_{0}(E)$, where $(E_{n})$ is any strictly increasing sequence of closed balls of $E$ whose union covers $E$.
Since, for all $n \in \mathbb{N}$, the topology induced on $\mathscr{C}_{b,E_{n}}(E)$ by $\mathscr{C}_{b,E_{n+1}}(E)$ coincides with that of $\mathscr{C}_{b,E_{n}}(E)$, the inductive limit is strict, and by \cite[Prop.\ 9 (iii), p.\ II.35]{BourbakiEVT}, $\mathscr{C}_{0}(E)$ is complete.
This topology is in general not metrizable.
A sequence $(f_{n}) \subset \mathscr{C}_{0}(E)$ converges to $f$ for the inductive topology we just defined, provided there exists a closed ball $B$ such that the supports of all the $f_{n}$'s are contained in $B$, and $(f_{n})$ converges to $f$ in $\mathscr{C}_{b,B}(E)$.
We endow the set $\mathscr{F}_{0}(E)$ with a similar inductive limit topology.
We also consider $\mathscr{M}_{0}(E)$ the space of signed Borel measures that are finite on bounded sets.
The set $\mathscr{M}_{0}(E)$ is endowed with the \emph{vague topology}, that is the initial topology associated with $\iota_{F}$, $F \in \mathscr{C}_{0}(E)$.
The importance of the inductive-limit topology on $\mathscr{C}_{0}(E)$ is highlighted by the fact that if $F_{n} \to F$ in $\mathscr{C}_{0}(E)$, then $\nu(F_{n}) \to \nu(F)$ for all $\nu \in \mathscr{M}_{0}(E)$.
\begin{remark}
  All the objects associated with a metric space $(E,d)$ as above depend on the metric structure of $d$ and not only on the topology generated by $d$.
  For instance, $d$ and $d \wedge 1$ generate the same topology.
  However, every set is bounded with respect to $d \wedge 1$.
\end{remark}

\subsubsection{Topological properties of the weak and vague topology}
Let us recall some fundamental results regarding the topology of the spaces of measures we consider.
\begin{theorem}\label{t:prohorov}
  Assume either that $(E, \tau)$ is a Polish space (for statements regarding the weak topology); or that $(E, d)$ is a complete and separable metric space (for statements regarding the vague topology).
  Then:
  \begin{enumerate}[$(i)$, wide]
  \item\label{i:prohorov-polish}
      The weak topology on $\mathscr{M}_{b}^{+}(E)$, resp.\ the vague topology on $\mathscr{M}_{0}^{+}(E)$, is induced by that of the simple convergence on a countable set of $\mathscr{C}_{b}(E)$, resp.\ $\mathscr{C}_{0}(E)$.
      Namely, there exists $(h_{k}) \subset \mathscr{C}_{b}(E)$, resp.\ $\mathscr{C}_{0}(E)$, such that the weak topology on $\mathscr{M}_{b}^{+}(E)$, resp.\ the  vague topology on $\mathscr{M}_{0}^{+}(E)$, is the locally convex topology generated by the seminorms
      \begin{equation*}
        \mu \mapsto \abs{\mu(h_{k})}, \qquad k \in \mathbb{N}.
      \end{equation*}
    Furthermore, the spaces $\mathscr{P}(E)$, $\mathscr{M}_{b}^{+}(E)$, and $\mathscr{M}^{+}_{0}(E)$ are Polish.
    \item\label{i:prohorov-compact}
      A set $\Delta \subset \mathscr{M}_{0}(E)$ is vaguely relatively sequentially compact if and only if both of the following conditions hold:
      \begin{subequations}\label{e:prohorov-compact}
          \begin{align}
    & \label{e:prohorov-compact-uniform-bound} \forall B \in \mathfrak{B}_{0}(E)\qquad \sup_{\mu \in \Delta} \abs{\mu}(B) < \infty; \\
    & \label{e:prohorov-compact-tight} \forall B \in \mathfrak{B}_{0}(E) \quad \forall \varepsilon > 0 \quad \exists K_{\varepsilon} \in \mathfrak{K}(E) : \sup_{\mu \in \Delta} \abs{\mu}(B \setminus K_{\varepsilon}) \leq \varepsilon.
          \end{align}
        \end{subequations}
      A set $\Delta \subset \mathscr{M}_{b}(E)$ is weakly relatively sequentially compact if and only if \cref{e:prohorov-compact-uniform-bound,e:prohorov-compact-tight} hold with $B = E$.
  \end{enumerate}
\end{theorem}

\begin{proof}
  \cref{i:prohorov-polish} \cite[Thms.\ 6.2, 6.5, \& 6.6]{Parthasarathy} for the case of $\mathscr{P}(E)$ with $(E,\tau)$ Polish.
  The case of $\mathscr{M}_{b}^{+}(E)$ is treated similarly.
  Now, assume that $(E,d)$ is complete and separable.
  Then it is also Polish, thus, by the previous case, we can find a countable family $(g_{k}) \subset \mathscr{C}_{b}(E)$ that induces the weak topology on $\mathscr{M}^{+}_{b}(E)$.
  We fix a point $o \in X$, and we consider a sequence $(f_{k}) \subset \mathscr{C}_{0}(E)$ such that $1_{B(o,k)} \leq f_{k} \leq 1_{B(o,k+1)}$.
  We take $(h_{k})$ an enumeration of $\set{ f_{j} g_{i} : j,\, i \in \mathbb{N} }$.
  Then
  \begin{equation*}
    \rho(\mu, \mu') \coloneq \sum_{k \in \mathbb{N}} 2^{-k} \abs*{ \int h_{k} \dd (\mu - \mu') },
  \end{equation*}
  is a distance metrizing the vague topology on $\mathscr{M}_{0}^{+}(E)$, and it is complete.

  \cref{i:prohorov-compact} \cite[Thm.\ 8.6.2]{Bogachev}.
\end{proof}

\subsection{Point processes, intensity measures, Campbell measures, Laplace transforms}
Let $(X, d)$ be a complete and separable metric space equipped with $m \in \mathscr{M}_{0}^{+}(X)$.
We write $\config = \config(X)$ for the space of \emph{configurations} over $X$, that is the $\mathbb{N}_{0} \cup \{\infty\}$-valued Borel measures on $X$ that are finite on every bounded set.
\begin{lemma}[{\cite[Lem.\ 2.1]{GHP}}]\label{t:config-locally-compact}
  The set $\config$ is closed in $\mathscr{M}_{0}^{+}(X)$. In particular, it is a Polish space.
\end{lemma}

A \emph{point process} $\mu$ is any element of $\mathscr{P}(\config)$.
Fix a point process $\mu$.
We write $I_{\mu}$ for the \emph{intensity measure} of $\mu$, that is
\begin{equation*}
  I_{\mu}(B) \coloneq \mu(\iota_{B}) = \int \eta(B) \mu(\dd \eta), \qquad B \in \mathfrak{B}(X).
\end{equation*}
The \emph{reduced Campbell measure} is
\begin{equation*}
  C_{\mu}(A \times B) \coloneq \iint 1_{B}(x) 1_{A}(\eta - \delta_{x}) \eta(\dd x) \mu(\dd \eta), \qquad A \in \mathfrak{B}(\config),\, B \in \mathfrak{B}(X).
\end{equation*}
It is a well-known fact \cite[Thm.\ 4.1]{LastPenrose} in the theory of point processes that $\mu$ is a Poisson point process (with intensity $I_{\mu}$) if and only if $C_{\mu} = \mu \otimes I_{\mu}$.
We refer to this relation as to the \emph{Mecke identity}.
When $\mu$ is a Poisson point process, for all probability densities $f \in L^{1}(\mu)$, we have that
\begin{equation}\label{e:campbell-density}
  \frac{\dd C_{f\mu}}{\dd (f\mu \otimes I_{f\mu})}(\eta, x) = \frac{f(\eta + \delta_{x})}{\int_{\config} f(\gamma + \delta_{x}) \mu(\dd \gamma)}, \qquad \eta \in \config,\, x \in X.
\end{equation}
However, for a generic point process $\mu$, the Campbell measure $C_{\mu}$ is \emph{not} absolutely continuous with respect to $\mu \otimes I_{\mu}$, see for instance, \cite{OsadaShirai} for an explicit counter-example.

Finally, the \emph{Laplace transform} of $\mu$ is the map
\begin{equation*}
  \Lambda_{\mu}(h) \coloneq \int \exp(- \eta(h)) \mu(\dd \eta), \qquad h \in \mathscr{C}^{+}_{0}(X).
\end{equation*}

  \subsection{The weak convergence \texorpdfstring{on $\mathscr{P}(\config)$}{of point processes}}
  We define $\mathscr{G}$ as the (algebraic) linear span of functions of the form $\e^{-\iota_{h}}$ for some $h \in \mathscr{F}^{+}_{0}(X)$.
  Set $\mathscr{S} \coloneq \mathscr{G} \cap \mathscr{C}_{b}(\config)$.
  Let us recall the following characterization of the weak convergence on $\mathscr{P}(\config)$.
  \begin{theorem}[{\cite[Thm.\ 4.11]{Kallenberg}}]\label{t:weak-convergence-P-upsilon}
    The space $\mathscr{P}(\config)$ is Polish.
   Moreover, for all $(\mu_{n}) \subset \mathscr{P}(\config)$ and $\mu \in \mathscr{P}(\config)$.
   Then,
\begin{equation}
\bracket[\Big]{ \mu_{n} \xrightarrow[n \to \infty]{\mathscr{P}(\config)} \mu } \Leftrightarrow \bracket[\Big]{ \mu_{n}(F) \to \mu(F), F \in \mathscr{S} } \Leftrightarrow \bracket[\Big]{ \Lambda_{\mu_{n}}(h) \to \Lambda_{\mu}(h), h \in \mathscr{C}_{0}^{+}(X) }.
\end{equation}
\end{theorem}
In general, there exists a no countable set $\mathscr{D} \subset \mathscr{S}$ convergence-determining for the weak topology on $\mathscr{P}(\config)$.
We now provide a partial ansatz to this result.
For $\lambda \in \mathscr{M}_{0}^{+}(X)$, the class $\mathfrak{B}_{0}^{\lambda}(X)$ of \emph{continuity sets} for $\lambda$ consists of the sets $B \in \mathfrak{B}_{0}(X)$ such that $\lambda(\partial B) = 0$.
We then define
\begin{equation}\label{e:control-continuity-sets}
  \mathscr{P}^{\lambda}(\config) \coloneq \set*{ \mu \in \mathscr{P}(\config) : \text{ continuity sets for $\lambda$ are also continuity sets for $I_{\mu}$} }.
\end{equation}
In particular, $\mu \in \mathscr{P}^{\lambda}(\config)$ whenever $I_{\mu} \ll \lambda$.
\begin{lemma}\label{t:weak-convergence-countable}
  Take $\lambda \in \mathscr{M}^{+}_{0}(X)$.
  There exists a countable set $\mathscr{G}^{\lambda} \subset \mathscr{G}$ such that the trace topology of $\mathscr{P}(\config)$ on $\mathscr{P}^{\lambda}(\config)$ is induced by the topology of simple convergence on $\mathscr{G}^{\lambda}$, namely it is induced by the seminorms
  \begin{equation}\label{e:weak-convergence-countable-seminorms}
    \mu \mapsto \abs{\mu(F)}, \qquad F \in \mathscr{G}^{\lambda}.
  \end{equation}
\end{lemma}

\begin{remark}
  We could also use \cref{t:prohorov} \cref{i:prohorov-polish} to find a countable subset of $\mathscr{C}_{b}(\config)$ to construct the seminorms.
  However, we cannot use $\mathscr{C}_{b}(\config)$ in the definition of the continuity equation \cref{e:ce} below.
\end{remark}

\begin{proof}
By \cite[Lem.\ 1.9 (v)]{Kallenberg}, $\mathfrak{B}_{0}^{\lambda}(X)$ is a \emph{dissecting ring} in the sense of \cite[p.\ 24]{Kallenberg}.
By \cite[Lem.\ 1.9 (i)]{Kallenberg}, there exists a countable dissecting ring $\mathfrak{I}^{\lambda} \subset \mathfrak{B}_{0}^{\lambda}(X)$.
Let $\mathcal{I}^{\lambda}$ be the set of simple, $\mathfrak{I}^{\lambda}$-measurable, $\mathbb{Q} \cap [0,1]$-valued functions on $X$.
In a more prosaic way, $\mathscr{I}^{\lambda}$ is the set of functions $h$ of the form
\begin{equation*}
  h = \sum_{i=1}^{l} q_{i} 1_{B_{i}}, \qquad l \in \mathbb{N},\, (q_{i}) \subset \mathbb{Q} \cap [0,1],\, (B_{i}) \subset \mathfrak{I}^{\lambda}.
\end{equation*}
Then $\mathscr{I}^{\lambda}$ is countable and we define:
\begin{equation*}
  \mathscr{G}^{\lambda} \coloneq \set{ \e^{-\iota_{h}} : h \in  \mathscr{I}^{\lambda} }.
\end{equation*}

Let us verify that $\mathscr{G}^{\lambda}$ is an appropriate choice for the claim.
Let $(\mu_{n}) \subset \mathscr{P}^{\lambda}(\config)$ and $\mu \in \mathscr{P}^{\lambda}(\config)$.
As a subset of the Polish space $\mathscr{P}(\config)$, the space $\mathscr{P}^{\lambda}(\config)$ is metrizable, and in particular, second-countable.
It is thus sufficient to verify that convergence of $(\mu_{n})$ with respect to the family of seminorms \cref{e:weak-convergence-countable-seminorms} is equivalent to weak convergence.
By construction, $\mathfrak{I}^{\lambda}$ is a dissecting ring consisting of continuity sets of $I_{\mu}$.
If $\mu_{n} \to \mu$ in $\mathscr{P}(\config)$, we get $\mu_{n}(F) \to \mu(F)$ for all $F \in \mathscr{G}^{\lambda}$, by \cite[Thm.~4.11 (iii)]{Kallenberg}.
Conversely, assume that $\mu_{n}(F) \to \mu(F)$ for all $F \in \mathscr{G}^{\lambda}$.
Then the same holds for all $F$ in the closure $\overline{\mathscr{G}^{\lambda}}$ of $\mathscr{G}^{\lambda}$ with respect to the uniform topology.
For $B \in \mathfrak{I}^{\lambda}$, $q \in \mathbb{Q} \cap [0,1]$, and $r \in [0,1]$, we have that
\begin{equation*}
\sup_{\eta \in \config} \abs[\big]{\e^{-q \eta(B)} - \e^{-r \eta(B)}} = \sup_{n \in \mathbb{N}} \abs[\big]{\e^{-qn} - \e^{-rn}} \xrightarrow[q \to r]{} 0.
\end{equation*}
Together with the triangle inequality, this shows that $\overline{\mathscr{G}^{\lambda}}$ contains functions of the form $\e^{-\iota_{h}}$ for $h$ a simple, $\mathfrak{I}^{\lambda}$-measurable, $[0,1]$-valued function on $\config$.
By \cref{t:weak-convergence-P-upsilon}, $\mu_{n} \to \mu$ in $\mathscr{P}(\config)$.
\end{proof}

\subsection{Locally integrable point processes}\label{s:locally-integrable}
Without further assumptions, $I_{\mu}$ is merely a non-negative measure on $X$, not necessarily finite on bounded sets.
This motivates the following definition.
We consider the set $\mathscr{C}_{b,0}(\config \times X)$ of continuous and bounded functions on $\config \times X$ that vanish outside of a set of the form $\config \times B$ for some $B \in \mathfrak{B}_{0}(X)$.
As for $\mathscr{C}_{0}(X)$ or $\mathscr{F}_{0}(X)$, the space $\mathscr{C}_{b,0}(\config \times X)$ can be endowed with an inductive limit topology.
More precisely, it is the strict inductive limit of the Banach spaces $\mathscr{C}_{b,\config \times B}(\config \times X)$ of continuous and bounded functions on $\config \times X$ vanishing outside of $\config \times B$ for some closed bounded set $B \subset X$.
Similarly to the vague topology, we consider the set $\mathscr{M}_{b,0}(\config \times X)$ of signed Borel measures $\nu$ on $\config \times X$ such that $\nu(\config \times B) < \infty$ for all $B \in \mathfrak{B}_{0}(X)$.
We equip it with the locally convex topology induced by the seminorms
\begin{equation*}
  \nu \mapsto \abs{\nu(F)}, \qquad F \in \mathscr{C}_{b,0}(\config \times X).
\end{equation*}
\cref{t:prohorov} also works for $\mathscr{M}_{b,0}(\config \times X)$, when we take for \enquote{bounded sets} the sets of the form $\config \times B$ for some $B \in \mathfrak{B}_{0}(X)$.
To see this we can consider the complete and separable metric space $E = \config \times X$ endowed with a distance of the form $d' \oplus d$ where $d'$ is any \emph{bounded} distance on $\config$ that is complete and induces the topology of $\config$.
Then, a set is bounded if $E$ if and only if it is contained in $\config \times B$ for some $B$ bounded in $X$, and the topology of $\mathscr{M}_{b,0}(\config \times X)$ we defined is the vague topology of $\mathscr{M}_{0}(E)$.

\begin{definition}
  We say that a point process $\mu$ is \emph{locally integrable} if $I_{\mu} \in \mathscr{M}_{0}(X)$.
  We write $\mathscr{P}_{1}(\config)$ for the set of all locally integrable point processes.
\end{definition}
We now equip $\mathscr{P}_{1}(\config)$ with a suitable topology.
We say that $F \in \mathscr{F}(\config)$ has \emph{sublinear growth}, provided there exists $c > 0$ and $h \in \mathscr{C}_{0}(X)$ such that:
\begin{equation*}
  \abs{F(\eta)} \leq c (1 + \eta(h)), \qquad \eta \in \config.
\end{equation*}
We write $\mathscr{C}_{1}(\config)$ for the set of continuous functions with sublinear growth.

\begin{remark}
  We always have $\mathscr{C}_{b}(\config) \subset \mathscr{C}_{1}(\config)$ with a strict inclusion, since for all $h \in \mathscr{C}_{0}(X) \setminus \{0\}$, $\iota_{h} \in \mathscr{C}_{1}(\config) \setminus \mathscr{C}_{b}(\config)$.
  This is true even when $X = \{*\}$ is the one-point space.
\end{remark}

\begin{definition}\label{d:P1}
  We equip $\mathscr{P}_{1}(\config)$ with the initial topology associated with the mappings $\iota_{F}$, $F \in \mathscr{C}_{1}(\config)$.
  In other words, it is the locally convex topology defined by the family of semi-norms
  \begin{equation*}
    \mu \mapsto \abs{\mu(F)}, \qquad F \in \mathscr{C}_{1}(\config).
  \end{equation*}
\end{definition}

We now establish that the space $\mathscr{P}_{1}(\config)$ with the above topology is Polish.
A central tool in proving so is the following property of the Campbell map.
\begin{theorem}\label{t:C-homeomorphism}
  The map $C \colon \mathscr{P}_{1}(\config) \to \mathscr{M}^{+}_{b,0}(\config \times X), \mu \mapsto C_{\mu}$ is a homeomorphism onto its image.
\end{theorem}

\begin{proof}
 We write $\mathscr{I}$ for the image of $C$.
  For all $\mu \in \mathscr{P}_{1}(\config)$ and $B \in \mathfrak{B}_{0}(X)$, $C_{\mu}(\config \times B) = I_{\mu}(B) < \infty$.
  Moreover, $C_{\mu}$ is always non-negative.
  Thus, $\mathscr{I} \subset \mathscr{M}^{+}_{b,0}(\config \times X)$, and the assertion is well-posed.
  In the rest of the proof, we write $\config^{*} \coloneq \config \setminus \set{\varnothing}$, where $\varnothing$ is the empty configuration.
  By \cite[IX, p.57, Prop. 1]{BourbakiTopo2}, the open set $\config^{*}$ is also Polish.
  \paragraph{$C$ is into.}
  Let $\mu$ and $\mu' \in \mathscr{P}_{1}(\config)$ such that $C_{\mu} = C_{\mu'}$.
  Let $A \in \mathfrak{B}(\config)$, $B \in \mathfrak{B}_{0}(X)$, and
  \begin{equation*}
    u(\eta, x) = \frac{1}{\eta(B) + 1}\, 1_{B}(x)\, 1_{A}(\eta + \delta_{x}) \qquad \eta \in \config,\, x \in X.
  \end{equation*}
  Then, we have that
  \begin{equation*}
    C_{\mu}(u) = \iint_{B} u(\eta - \delta_{x}, x) \eta(\dd x) \mu(\dd \eta) = \iint 1_{A}(\eta) \frac{1_{B}(x)}{\eta(B)} \eta(\dd x) \mu(\dd \eta) = \mu(A \cap \{\eta(B) > 0 \}).
  \end{equation*}
  Letting $B \nearrow X$ we get that $\mu(A) = \mu'(A)$ for all $A \in \mathfrak{B}(\config^{*})$ by monotone convergence.
  Thus~$\mu$ and~$\mu'$ coincide as measures on $\config^{*}$ but since they are probability measures on $\config$, we have that
  \begin{equation*}
    \mu(\varnothing) = 1 -\mu(\config^{*}) = \mu'(\varnothing).
  \end{equation*}
  Thus $\mu = \mu'$ on $\mathscr{P}_{1}(\config)$.
  \paragraph{$C$ is continuous.}
  Take $\mu_{o} \in \mathscr{P}_{1}(\config)$ and $l \in \mathbb{N}$.
  For $i=1,\dots,l$, let $u_{i} \in \mathscr{C}_{b,0}(\config \times X)$, and $\varepsilon_{i} > 0$.
  We set
  \begin{equation*}
  V \coloneq \bigcap_{i=1}^{l} \ \set[\Big]{ C_{\mu} : \mu \in \mathscr{P}_{1}(\config),\,  \abs{(C_{\mu} - C_{\mu_{o}})(u_{i})} \leq \varepsilon_{i}}.
  \end{equation*}
  The set $V$ is a neighbourhood in $\mathscr{I}$ of $C_{\mu_{o}}$ and the class of all sets $V$ of this form is a fundamental system of neighbourhoods of $C_{\mu_{o}}$ (for instance, \cite[II, pp. 2-4]{BourbakiIntegration}).
  Thus it suffices to show that $C^{-1}(V)$ is a neighbourhood of $\mu_{o}$.
  Now, since $C$ is into, we have that
  \begin{equation*}
  C^{-1}(V) = \bigcap_{i=1}^{l} \ \set[\Big]{ \mu \in \mathscr{P}_{1}(\config) : \abs{(C_{\mu} - C_{\mu_{o}})(u_{i})} \leq \varepsilon_{i} }.
  \end{equation*}
  Let $i = 1,\dots, l$.
  We set, for $\eta \in \config$, $F_{i}(\eta) \coloneq \int u_{i}(\eta - \delta_{x}, x) \eta(\dd x)$.
  By \cref{t:f-from-u-continuous} below, $F_{i} \in \mathscr{C}_{1}(\config)$.
  Moreover,
  \begin{equation*}
    (C_{\mu} - C_{\mu_{o}})(u_{i}) = \iint u_{i}(\eta-\delta_{x}, x) \eta(\dd x)(\mu - \mu_{o})(\dd \eta) = (\mu - \mu_{o})(F_{i}).
  \end{equation*}
  Thus, we get
  \begin{equation*}
  C^{-1}(V) = \bigcap_{i=1}^{l} \ \set[\Big]{ \mu \in \mathscr{P}_{1}(\config) : \abs{(\mu - \mu_{o})(F_{i})} \leq \varepsilon_{i}},
  \end{equation*}
  which, by definition of the topology on $\mathscr{P}_{1}(\config)$, is a neighbourhood of $\mu_{o}$ in $\mathscr{P}_{1}(\config)$.
  \paragraph{$C^{-1}$ is continuous.}
  Since $\mathscr{M}^{+}_{b,0}(\config \times X)$ is Polish, it is sufficient to show that $C^{-1}$ is sequentially continuous.
  Thus, let us consider $(\mu_{n}) \subset \mathscr{P}_{1}(\config)$ and $\mu \in \mathscr{P}_{1}(\config)$ such that $C_{\mu_{n}} \to C_{\mu}$ in $\mathscr{M}_{b,0}(\config \times X)$.
  Take $h \in \mathscr{C}_{0}(X)$ with $h \geq 0$.
  We have the following bound:
  \begin{equation*}
    \abs*{ \e^{-s\eta(h)} - \e^{-t\eta(h)} } \leq \abs{s-t} \abs{\eta(h)}, \qquad s,\,t \in (0,1),\, \eta \in \config.
  \end{equation*}
  Since $\mu \in \mathscr{P}_{1}(\config)$ and $h \in \mathscr{C}_{0}(X)$, we have that $\iota_{h} \in L^{1}(\mu)$.
  Since, $\mu \in \mathscr{P}_{1}(\config)$, by Lebesgue dominated convergence theorem, we find that the map $(0,1) \ni t \mapsto \Lambda_{\mu}(th)$ is differentiable with derivative given by:
  \begin{equation*}
    \frac{\dd}{\dd t} \Lambda_{\mu}(th) = C_{\mu}(u_{t}),
  \end{equation*}
  where
  \begin{equation*}
    u_{t}(\eta, x) \coloneq - h(x) \e^{-th(x)} \exp\left(-t \int h \dd \eta\right), \qquad \eta \in \config,\, x \in X,\, t \in (0,1).
  \end{equation*}
  Now, observe that,
  \begin{equation*}
    \abs{ C_{\mu_{n}}(u_{t}) } = \int \eta(h) \e^{-t\eta(h)} \mu_{n}(\dd \eta) \leq I_{\mu_{n}}(h) = C_{\mu_{n}}(1 \otimes h).
  \end{equation*}
  The right-hand side is uniformly bounded with respect to $n \in \mathbb{N}$ since, by assumption $(C_{\mu_{n}})$ converges in $\mathscr{M}_{b,0}(\config \times X)$ and $1 \otimes h \in \mathscr{C}_{b,0}(\config \times X)$.
  Thus, by dominated convergence:
  \begin{equation*}
    \Lambda_{\mu_{n}}(h) = 1 + \int_{0}^{1} C_{\mu_{n}}(u_{t}) \dd t \xrightarrow[n \to \infty]{} 1 + \int_{0}^{1} C_{\mu}(u_{t}) \dd t = \Lambda_{\mu}(h).
  \end{equation*}
  Thus $(\mu_{n})$ converges weakly to $\mu$ (\cref{t:weak-convergence-P-upsilon}).
  Take $F \in \mathscr{C}_{1}(\config)$.
  By definition, we can find $h \in \mathscr{C}_{0}(X)$ such that $\abs{F} \leq c (\iota_{h} + 1)$ for some $c > 0$.
  Now the convergence of the Campbell measures implies the convergence of the intensity measures in $\mathscr{M}_{0}(X)$.
  Thus, by \cite[Lem.\ 5.11]{KallenbergFMP}, we get that $\iota_{h}$ is uniformly integrable with respect to $(\mu_{n})$.
  So that $F$ is also uniformly integrable with respect to $(\mu_{n})$.
  By the continuous mapping theorem, $F_{\sharp} \mu_{n} \to F_{\sharp} \mu$ in distribution, together with uniform integrability, this gives by \cite[Lem.\ 5.11]{KallenbergFMP}, that $\mu_{n}(F) \to \mu(F)$.
  This shows that $\mu_{n} \to \mu$ in $\mathscr{P}_{1}(\config)$.
\end{proof}

Now we turn to the proof of Polishness.
\begin{theorem}\label{t:p1-polish}
  The space $\mathscr{P}_{1}(\config)$ is Polish.
\end{theorem}
\begin{proof}
  It is sufficient to show that $\mathscr{P}_{1}(\config)$ is homeomorphic to a Polish space (for instance, \cite[IX, p.\ 58, Cor.\ 2]{BourbakiTopo2}).
  Thus, in view of the previous theorem and \cite[IX, p.\ 57, Prop.\ 1]{BourbakiTopo2}, it suffices to show that the image $\mathscr{I}$ of $C$ in $\mathscr{M}^{+}_{b,0}(\config \times X)$ is closed.
  Take $(\mu_{n}) \subset \mathscr{P}_{1}(\config)$ such that $C_{\mu_{n}} \to \sigma \in \mathscr{M}^{+}_{b,0}(\config \times X)$.
  By continuity of the projection $\config \times X \to X$, the sequence $(I_{\mu_{n}})$ also converges to some measure in $\mathscr{M}^{+}_{0}(X)$.
  This yields that $\Delta \coloneq (I_{\mu_{n}})$ is relatively compact, so that by \cref{t:prohorov} it satisfies \cref{e:prohorov-compact}.
  Precisely, take $B \in \mathfrak{B}_{0}(X)$, by \cref{e:prohorov-compact-tight}, we have that
  \begin{equation*}
  \inf_{K \in \mathfrak{K}(X)} \sup_{n \in \mathbb{N}} \int \paren[\big]{\eta(B \setminus K) \wedge 1} \mu_{n}(\dd \eta) \leq \inf_{K \in \mathfrak{K}(X)} \sup_{n \in \mathbb{N}} I_{\mu_{n}}(B \setminus K) = 0;
  \end{equation*}
  while, by \cref{e:prohorov-compact-uniform-bound} and Markov's inequality, we get that
  \begin{equation*}
    \sup_{n \in \mathbb{N}} \mu_{n}(\eta(B) > r) \leq \frac{1}{r} \sup_{n \in \mathbb{N}} I_{\mu_{n}}(B) \leq \frac{c}{r}.
  \end{equation*}
  The two previous equations show that the conditions of \cite[Thm.\ 4.10]{Kallenberg} are satisfied and thus up to extraction we can find $\mu \in \mathscr{P}(\config)$ such that $\mu_{n} \to \mu$ in $\mathscr{P}(\config)$.

  Now, let $h \in \mathscr{C}_{0}(X)$.
  By \cite[Lem.\ 5.11]{KallenbergFMP}, we get $\mu \in \mathscr{P}_{1}(\config)$ with $I_{\mu}(h) \leq \liminf_{n} I_{\mu_{n}}(h) < \infty$.
  By definition, the map $\eta \mapsto \eta(h)$ is continuous.
  Thus, the set $\set{ \eta(h) + 1 \geq r }$ is a closed set, for every $r > 0$,
  and the map
  \begin{equation*}
    u_{r}(\eta, x) \coloneq h(x) 1_{\set{\eta(h) + 1 \geq r}}, \qquad \eta \in \config,\, x \in X,
  \end{equation*}
  is upper semi-continuous.
  By the Portmanteau Theorem,
  \begin{equation*}
    \limsup_{n} \int \eta(h)\, 1_{\set{\eta(h) \geq r}} \dd \mu_{n} \leq \limsup_{n} C_{\mu_{n}}(u_{r}) \leq \sigma(u_{r}).
  \end{equation*}
  By dominated convergence, the right-hand side converges to $0$ as $r \to \infty$.
  In particular, $\iota_{h}$ is uniformly integrable with respect to $(\mu_{n})$.
  By an argument similar to the one used in the previous proof, we conclude that $\mu_{n} \to \mu$ in $\mathscr{P}_{1}(\config)$.
  Since $C$ is continuous and $\mathscr{M}_{b,0}(\config \times X)$ is Hausdorff, this shows that $\sigma = C_{\mu}$.
\end{proof}

Actually in the previous proofs, we have established the two following results that we extract here for convenience.
\begin{proposition}\label{t:p1-convergence}
  Let $(\mu_{n})_{n \in \nat} \subset \mathscr{P}_{1}(\config)$ and $\mu \in \mathscr{P}(\config)$.
  Then, the following are equivalent:
  \begin{enumerate}[$(i)$]
    \item\label{i:convergence-p1} $\mu \in \mathscr{P}_{1}(\config)$ and $\mu_{n} \xrightarrow[n \to \infty]{\mathscr{P}_{1}(\config)} \mu$.
    \item\label{i:convergence-C} $C_{\mu_{n}} \xrightarrow[n \to \infty]{\mathscr{M}_{b,0}(\config \times X)} C_{\mu}$.
    \item\label{i:convergence-I} $\mu_{n} \xrightarrow[n \to \infty]{weakly} \mu$ and $I_{\mu_{n}} \xrightarrow[n \to \infty]{\mathscr{M}_{0}(X)} I_{\mu}$.
  \end{enumerate}
\end{proposition}

\begin{proof}
  The equivalence between \cref{i:convergence-p1,i:convergence-C} follows from \cref{t:C-homeomorphism} since $C$ is a homemorphism.
  We proved that \cref{i:convergence-C} implies \cref{i:convergence-I} implies \cref{i:convergence-p1} in the proof of the continuity of $C^{-1}$ in \cref{t:C-homeomorphism}.
\end{proof}

  \begin{remark}
    Take $d_{1}$ (resp.\ $d_{2}$) a distance that completely metrizes the topology of $\mathscr{P}(\config)$ (resp.\ that of $\mathscr{M}^{+}_{0}(X)$).
    In view of the above result the distance
    \begin{equation*}
      d(\mu, \mu) \coloneq d_{1}(\mu, \mu') + d_{2}(I_{\mu}, I_{\mu'}), \qquad \mu,\, \mu' \in \mathscr{P}_{1}(\config),
    \end{equation*}
    metrizes the topology of $\mathscr{P}_{1}(\config)$.
    However, this distance may in general be \emph{not} complete.
    
    Indeed, take $X = \set{ \ast }$, the one-point space.
    Then, $\config$ is identified with $\nat_0$, and $\mathscr{M}_{0}(X)$ is identified with $\mathbb{R}$.
    Let $\mu_{n}$ be the law of a random variable that takes the value $n$ with probability $1/n$, and $0$ with probability $1 -1/n$.
    Then $\mu_{n} \to \delta_{0}$ in $\mathscr{P}(\nat_0)$ so that $(\mu_{n})_n$ is Cauchy with respect to $d_{1}$.
    Moreover, we have that $I_{\mu_{n}} = 1$ for all $n \in \mathbb{N}$, so that $(I_{\mu_{n}})_n$ is Cauchy with respect to $d_{2}$.
 Thus, $(\mu_{n})_n$ is Cauchy with respect to $d$.
    However, it does not converge in $\mathscr{P}_{1}(\nat_0)$, since $I_{\delta_0}=0$.
  \end{remark}

\begin{proposition}\label{t:p1-compactness}
  Let $\mathscr{A} \subset \mathscr{P}_{1}(\config)$.
  Then, the following are equivalent:
  \begin{enumerate}[$(i)$]
    \item\label{i:compact-p1} $\mathscr{A}$ is relatively compact in $\mathscr{P}_{1}(\config)$;
    \item\label{i:ui} $\mathscr{A}$ is relatively compact in $\mathscr{P}(\config)$ and, for all $h \in \mathscr{C}_{0}(X)$, the map $\iota_{h}$ is uniformly integrable with respect to $\mathscr{A}$.
  \end{enumerate}
\end{proposition}

\begin{proof}
  Assume that $\mathscr{A}$ is relatively compact in $\mathscr{P}_{1}(\config)$.
  Take a sequence in $\mathscr{A}$.
  Up to extraction it converges in $\mathscr{P}_{1}(\config)$.
  Thus, by \cref{t:p1-convergence} \cref{i:convergence-I}, it converges weakly.
  This shows that $\mathscr{A}$ is weakly sequentially relatively compact, and thus relatively compact in $\mathscr{P}(\config)$.
  The uniform integrability follows from \cite[Lem.\ 5.11]{KallenbergFMP} together with \cref{t:p1-convergence} \cref{i:convergence-I}.

  Conversely, assume that $\mathscr{A}$ is weakly relatively compact and that we have the uniform integrability condition.
  Then up to extraction every sequence in $\mathscr{A}$ weakly converges, and by the uniform integrability and \cite[Lem.\ 5.11]{KallenbergFMP}, we find that the intensiy measures also converge.
  By \cref{t:p1-convergence}, we deduce that $\mathscr{A}$ is then sequentially relatively compact in $\mathscr{P}_{1}(\config)$, and thus it is relatively compact, since $\mathscr{P}_{1}(\config)$ is Polish by \cref{t:p1-polish}.
\end{proof}

Let us finish with the proof of the lemma used above.
\begin{lemma}\label{t:f-from-u-continuous}
  If $u \in \mathscr{C}_{b,0}(\config \times X)$, then
  \begin{equation*}
    F \colon \eta \longmapsto \int u(\eta - \delta_{x}, x) \eta(\dd x), \qquad \eta \in \config,
  \end{equation*}
  satisfies $F \in \mathscr{C}_{1}(\config)$.
\end{lemma}
\begin{proof}
  Without loss of generality, we assume that $u \geq 0$.
  Indeed, for $u \in \mathscr{C}_{b,0}(\config \times X)$, we also have that $u_{+}$ and $u_{-} \in \mathscr{C}_{b,0}(\config \times X)$.
  Since $u \in \mathscr{C}_{b,0}(\config \times X)$, there exists $c > 0$ and $h \in \mathscr{C}_{0}(X)$ such that $h \geq 0$ and
  \begin{equation*}
    u(\eta, x) \leq c h(x), \qquad \eta \in \config,\, x \in X.
  \end{equation*}
  We have that
  \begin{equation*}
    \abs{F(\eta_{n}) - F(\eta)} \leq c \abs*{ \int h(x) (\eta_{n} - \eta)(\dd x) } + \int \abs*{ u(\eta_{n} - \delta_{x}, x) - u(\eta - \delta_{x}, x) } \eta(\dd x).
  \end{equation*}
   The first term vanishes as $n \to \infty$, by definition of vague convergence.
  Since $u$ is continuous the integrand in the second term also vanishes and is dominated by $2c h \in L^{1}(\eta)$.
  By dominated convergence, the corresponding integral also vanishes.
  This shows that $F$ is continuous.
  We also have that $F(\eta) \leq c \eta(h)$ which shows that $F \in \mathscr{C}_{1}(\config)$.
\end{proof}

\section{Discrete operators and the Ornstein--Uhlenbeck dynamics}\label{s:ou}
\subsection{Mehler's formula, difference operator, divergence of a function}\label{s:mehler}
We refer the reader to \cite{LastAnaSto} for more details and proofs regarding objects introduced in this section.
We fix a Poisson point process $\pi$ with intensity $m \in \mathscr{M}_{0}^{+}(X)$.

  For all $\pi$-integrable $F \in \mathscr{F}(\config)$, we define the \emph{Ornstein--Uhlenbeck semi-group} by
\begin{equation*}
  \ou_{t} F(\eta) \coloneq \Esp*{ F(\eta^{(\e^{-t})} + \xi_{t}) }, \qquad \eta \in \config,\, t \ge 0,
\end{equation*}
where $\eta^{(s)}$ is the $s$-thinning of $\eta$ and $\xi_{t}$ is distributed as a Poisson point process with intensity $(1-\e^{-t})m$ and is independent of the thinning.

The family $\ou = (\ou_{t})_{t \ge 0}$ is a Markov semigroup on $L^{1}(\pi)$.
Moreover, it maps continuous functions to continuous functions.
For all point processes $\mu$, we define by duality
\begin{equation*}
  \dou_{t} \mu(A) \coloneq \mu(\ou_{t}1_{A}), \qquad A \in \mathfrak{B}(\config),\quad t \ge 0.
\end{equation*}
It is readily verified that this indeed defines a measure.
If $\mu = \rho \pi$, then $\dou_{t}\mu = (\ou_{t} \rho) \pi$ for all $t \ge 0$.
We also have that $\dou_{t}$ maps $\mathscr{P}_{1}(\config)$ to $\mathscr{P}_{1}(\config)$ for all $t > 0$, and that
\begin{equation}\label{e:continuity-dou-p1}
  \dou_{t}\mu \xrightarrow[t \to 0]{\mathscr{P}_{1}(\config)} \mu, \qquad \mu \in \mathscr{P}_{1}(\config).
\end{equation}
\begin{proof}[Proof of \cref{e:continuity-dou-p1}]
Let $\eta \sim \mu$ and $h \in \mathscr{C}_{0}(X)$, and set $\eta_{t} \coloneq \eta^{(\e^{-t})}$.
By \cite[Lem.\ 3.1]{Kallenberg}, we have that
  \begin{equation*}
    \Esp*{ \exp(-\eta_{t}(h)) }  = \Esp*{ \Exp*{ \int \Log*{1 - \e^{-t}(1 - \e^{-h})} \dd \eta } }.
  \end{equation*}
  Since $\Log*{ 1 - \e^{-t}(1 - \e^{-h}) } \leq (1-\e^{-h})$, by dominated convergence and \cite[Thm.\ 4.11]{Kallenberg}, we get that $\law(\eta_{t}) \to \mu$ in $\mathscr{P}(\config)$.
  Similar computations show that $\law(\xi_{t}) \to \delta_{\varnothing}$ in $\mathscr{P}(\config)$.
  Thus, by continuity of the sum and the continuous mapping theorem, we conclude that $\dou_{t}\mu = \law(\eta_{t} + \xi_{t}) \to \mu$ in $\mathscr{P}(\config)$.
  The Mehler formula also implies that
  \begin{equation}\label{e:intensity-Ornstein-Uhlenbeck}
    I_{\dou_{t}\mu}(h) = \e^{-t} I_{\mu}(h) + (1-\e^{-t}) m(h) \xrightarrow[t \to 0]{} I_{\mu}(h),
  \end{equation}
which concludes the proof in view of \cref{t:p1-convergence}.
\end{proof}

For $F \in \mathscr{F}(\config)$ we write
\begin{equation*}
  \diff_{x} F(\eta) \coloneq F(\eta + \delta_{x}) - F(\eta), \qquad \eta \in \config,\quad x \in X,
\end{equation*}
and regard $\diff F$ as the map
\begin{equation*}
  \diff F \colon \config \times X \ni (\eta, x) \mapsto \diff_{x} F(\eta).
\end{equation*}
The difference operator and the semi-group satisfy a commutation relation à la Bakry--Émery \cite[Lem.\ 6]{LastAnaSto}:
\begin{equation}\label{e:be}\tag{$\mathbf{BE}$}
  \diff \ou_{t} F = \e^{-t} \ou_{t} \diff F, \qquad F \in L^{2}(\pi).
\end{equation}
For all $u \in L^{1}(\pi \otimes m)$, we define a formal adjoint to $\diff$, namely the \emph{Skorokhod divergence}
\begin{equation*}
  \skoro u(\eta) \coloneq \int u(\eta - \delta_{x}, x) \eta(\dd x) - \int u(\eta, x) m(\dd x), \qquad \eta \in \config.
\end{equation*}
By the Mecke formula, $\skoro u \in L^{1}(\pi)$.

\subsection{Sobolev spaces}

Due to its discrete nature, $\diff$ does not give rise to a good notion of smooth functions.
As a partial substitute, let us define the Sobolev spaces associated with $\diff$.
For all $F \in \mathscr{F}(\config)$ and $k \in \nat$, we can define iteratively $\diff^{k} F \in \mathscr{F}(\config \times X^{k})$.
We can thus define the \emph{Sobolev spaces of order $k \in \nat$ and $p \in [1,\infty]$} as the set $\sobolev^{k,p} = \sobolev^{k,p}(\config)$ containing all $F \in L^{p}(\pi)$ such that $\diff^{k'} F \in L^{p}(\pi \otimes m^{\otimes k'})$ for all $k' \leq k$.
It is endowed with the norm
\begin{equation*}
  \norm{F}_{k,p} \coloneq \norm{F}_{L^{p}(\pi)} + \sum_{k'=1}^{k} \norm{D^{k'}F}_{L^{p}(\pi \otimes m^{\otimes k'})}.
\end{equation*}

\subsection{Generator of the Ornstein--Uhlenbeck semi-group}

The Markov generator of $\ou$ on $L^{2}(\pi)$ is the unbounded operator $\gen$ with $\dom \gen$ consisting of all the functions $F \in L^{2}(\pi)$ such that the following limit exists 
\begin{equation*}
  \gen F \coloneq L^{2}(\pi)\text{-}\lim_{t \to 0} \frac{\mathsf{P}_{t} - \mathsf{id}}{t} F.
\end{equation*}
By \cite[Props.\ 3 \& 4]{LastAnaSto}, we have that $\dom \mathsf{L} = \mathscr{W}^{2,2}$; for $F \in \dom{\gen} \cap \mathscr{W}^{1,1}$, we have the following representation
\begin{equation*}
  \gen F(\eta) = - \skoro \diff F(\eta) = \int \paren[\big]{F(\eta + \delta_{x}) - F(\eta)} m(\dd x) - \int \paren[\big]{F(\eta) - F(\eta - \delta_{x})} \eta(\dd x), \qquad \eta \in \config,
\end{equation*}
and the following integration by parts holds
\begin{align}
    & \label{e:ou-ipp-generator} \int F\, \mathsf{L} G\, \dd \pi = - \int \diff F\, \diff G\, \dd(\pi \otimes m), \qquad F \in \mathscr{W}^{1,2},\, G \in \dom \mathsf{L}.
\end{align}
In view of the general theory of Dirichlet forms \cite[Thm.\ 2.20]{MaRockner}, we have the following regularization property
\begin{equation}\label{e:ou-regularization-L2}
  \mathsf{P}_{t} L^{2}(\pi) \subset \dom \mathsf{L}, \qquad t > 0.
\end{equation}
The inverse of $\mathsf{L}$ is defined for all $F \in L^{2}(\pi)$ such that $\int F \dd \pi = 0$ via \cite[Thm.\ 7]{LastAnaSto}:
\begin{equation*}
  \mathsf{L}^{-1} F(\eta) \coloneq - \int_{0}^{\infty} \mathsf{P}_{s} F(\eta) \dd s.
\end{equation*}

\subsection{Relative entropy}
Let $\pi$ be the Poisson point process with intensity $m$, and $\mu$ be a point process.
The \emph{relative entropy of $\mu$ with respect to $\pi$} is
\begin{equation*}
\Ent{\mu \given \pi} \coloneq \int \rho \log \rho \dd \pi \qquad \text{if} \quad \mu\ll \pi , \quad \rho= \frac{\dd \mu}{\dd \pi},
\end{equation*}
and $\Ent{\mu \given \pi} = \infty$ otherwise.
We write $\dom \ent$ for the set of $\mu \in \mathscr{P}(\config)$ with $\Ent{\mu \given \pi} < \infty$.
The following result recasts well-known properties of the relative entropy with respect to the weak topology in the setting of $\mathscr{P}_{1}(\config)$.
\begin{lemma}\label{t:lsc-H}
  We have that $\dom \ent \subset \mathscr{P}_{1}(\config)$.
  Moreover, $\Ent{\emparg \given \pi}$ is lower semi-continuous with respect to the $\mathscr{P}_{1}(\config)$-topology and its sub-level sets are relatively compact in $\mathscr{P}_{1}(\config)$.
\end{lemma}
\begin{proof}
Set $\theta(s) \coloneq s \log s - s + 1$ for $s \geq 0$ and $\theta(s) \coloneq \infty$ otherwise. We denote its Legendre transform by
\begin{equation*}
\theta^{*}(t) \coloneq \sup_{s \in \real} \paren[\big]{s t - \theta(s) } = \e^{t} - 1, \qquad t \in \real.
\end{equation*}
The functions $\theta$ and $\theta^{*}$ are convex conjugate to each other, and a pair of Young functions.
  We define, the \emph{Orlicz norm}
  \begin{equation*}
    \norm{F}_{L^\theta} \coloneq \Sup*{ \int FG \dd \pi : \int \theta^{*}(|G|) \leq 1 }, \qquad F \in \mathscr{F}(\config).
  \end{equation*}
  Let $\mu=\rho\pi \in \dom \ent$ and note that $\int \theta(\rho) \dd \pi = \Ent{\mu \given \pi}$.
  For $r \in \real$ and $B \in \mathfrak{B}_{0}(X)$ set $F_{r}(\eta) \coloneq \eta(B) 1_{\{ \eta(B) > r\}}$.
  In view of \cite[Eq. (9.13), p. 73]{Orlicz}, we find that
  \begin{equation}\label{e:holder-ui}
    \abs*{ \int F_{r} \rho \dd \pi } \leq \norm{F_{r}}_{L^{\theta^{*}}} \paren[\big]{1 \wedge \Ent{\mu \given \pi}}.
\end{equation}
On the one hand, by Fenchel's inequality, we have that
\begin{equation}\label{eq:FenchelIotaB}
\norm{F_{r}}_{L^{\theta^{*}}}\leq \norm{F_{0}}_{L^{\theta^{*}}} \leq \int (\e^{\eta(B)} - 1) \dd \pi + 1.
\end{equation}
This quantity is finite in view of the exponential integrability of Poisson random variables.
The second inequality above shows that $\dom \ent \subset \mathscr{P}_{1}(\config)$.
On the other hand, by dominated convergence,
\begin{equation*}
  \int \theta^{*}(a F_{r}) \dd \pi = \int 1_{\{ a \eta(B) > r \}} (\e^{a \eta(B)} - 1)\, \pi(\dd \eta) \xrightarrow[r\to\infty]{} 0, \qquad  a > 0.
\end{equation*}
In view of the equivalence of the Orlicz and the Luxembourg norms \cite[Eq. (9.24), p. 80]{Orlicz}, the latter convergence implies that
\begin{equation*}
  \lim_{r \to \infty} \norm{F_{r}}_{L^{\theta^{*}}} = 0.
\end{equation*}
Together with \cref{e:holder-ui}, this shows that the uniform integrability condition in \cref{t:p1-compactness} is satisfied on sub-level sets of $\Ent{\emparg \given \pi}$.
By \cite[Lem.\ 6.2.12]{DemboZeitouni}, these sub-level sets are also weakly relatively compact, thus we conclude they are relatively compact in $\mathscr{P}_{1}(\config)$ by \cref{t:p1-compactness}.
Since $\Ent{ \emparg \given \pi}$ is weakly lower semi-continuous (e.g., \cite[Lem.\ 6.2.13]{DemboZeitouni}) and since the $\mathscr{P}_{1}(\config)$-topology is finer than the weak topology, we get the lower semi-continuity.
\end{proof}

\subsection{Fisher information}
The \emph{(modified) Fisher information} of $\mu \in \mathscr{P}(\config)$ is
\begin{equation*}
  \Fish{\mu \given \pi} \coloneq \int \diff \rho \ \diff \log \rho \dd (\pi \otimes m) \qquad \text{if} \quad \mu\ll \pi , \quad \rho= \frac{\dd \mu}{\dd \pi},
\end{equation*}
if $\mu = \rho \pi$, and $\Fish{\mu \given \pi} \coloneq \infty$ otherwise.
We write $\dom \fish$ for the set of $\mu$'s with $\Fish{\mu \given \pi} < \infty$.
By convexity of $\theta$, we have that $(a - b) (\log a - \log b) = (a-b) (\theta'(a) - \theta'(b)) \geq 0$ for all $a$ and $b \in \real$.
This shows that $\Fish{\mu \given \pi}$ is well-defined, although potentially $\infty$, for all $\mu \in \mathscr{P}(\config)$.
The relative entropy and the Fisher information are related through the \emph{modified logarithmic Sobolev inequality} \cite[Cor.\ 2.2]{Wu}:
\begin{equation}\label{e:lsi}
  \Ent{\mu \given \pi} \leq \Fish{\mu \given \pi}, \qquad \mu \in \mathscr{P}(\config).
\end{equation}

\begin{theorem}\label{t:lsc-I}
  The functional $\Fish{\emparg \given \pi}$ is lower semi-continuous on $\mathscr{P}_{1}(\config)$.
\end{theorem}

\begin{proof}
The lower semi-continuity of the Fisher information will follow from that of similar functionals defined at the level of functions, by a uniform integrability argument, as we now show.

\paragraph{Weak $L^1_{\mathrm{loc}}$ lower semi-continuity.}

Let $L^{1}_{\mathrm{loc}}(\pi \otimes m)$ be the space of (equivalence classes of) Borel functions $u \colon \config \times X \to \mathbb{R}$ such that
\begin{equation*}
  \norm{u}_{B} \coloneq \int_{\config \times B} \abs{u} \dd(\pi \otimes m) < \infty, \qquad B \in \mathfrak{B}_{0}(X).
\end{equation*}
Equipped with the locally convex topology induced by the family of semi-norms $\norm{\emparg}_{B}$ with $B \in \mathfrak{B}_{0}(X)$, the space $L^{1}_{\mathrm{loc}}(\pi \otimes m)$ is a Fréchet space, and every continuous linear functional on $L^{1}_{\mathrm{loc}}(\pi \otimes m)$ is represented by some $u \in L^{\infty}(\pi \otimes m)$ with $u$ vanishing outside of some $B \in \mathfrak{B}_{0}(X)$, see \cref{t:DualLLoc}.

Let us define
\begin{equation*}
  \psi(s,t) \coloneq
  \begin{cases}
  \paren[\big]{\log(s+t) - \log s} t, & \text{if } s > 0,\, t > - s;
  \\ +\infty, & \text{otherwise;}
\end{cases}
\end{equation*}
and
\begin{align*}
  & \mathcal{I}_{\pi}(\rho, u) \coloneq \int_{\config \times X} \psi(\rho,u) \dd (\pi \otimes m), \qquad \rho \in L^{1}(\pi),\, u \in L^{1}_{\mathrm{loc}}(\pi \otimes m); \\
  & \mathcal{I}_{\pi,B}(\rho, u) \coloneq \int_{\config \times B} \psi(\rho, u) \dd (\pi \otimes m), \qquad \rho \in L^{1}(\pi),\, u \in L_{\mathrm{loc}}^{1}(\pi \otimes m),\, B\in\mathfrak{B}_0(X).
\end{align*}
We fix $B \in \mathfrak{B}_{0}(X)$ and we write $m_{B}$ for the restriction of $m$ to $B$.
Since $\pi \otimes m_{B}$ is a finite non-atomic measure, and since $\psi \geq 0$, by \cite[Thm.\ 1]{Ioffe} we find that $\mathcal{I}_{\pi,B}$ is lower semi-continuous with respect to the weak topology of $L^{1}(\pi) \times L^{1}(\pi \otimes m_{B})$.
By \cite[II, p.\ 53, Prop.\ 8]{BourbakiEVT}, this weak topology is actually the product topology of the weak topologies on $L^1(\pi)$ and $L^1(\pi\otimes m_{B})$.

Let $(\rho_{\alpha}) \subset L^{1}(\pi)$ be a net weakly converging to $\rho \in L^{1}(\pi)$ and $(u_{\alpha}) \subset L^1_{\mathrm{loc}}(\pi\otimes m)$ be a net weakly converging to $u\in L^1_{\mathrm{loc}}(\pi\otimes m)$.
On the one hand, in view of \cref{t:DualLLoc}, we find that $(u_{\alpha} 1_{B})$ is a net converging weakly in $L^{1}(\pi \otimes m_{B})$ to $u 1_{B}$.
On the other hand,
\begin{equation*}
  \liminf_{\alpha} \mathcal{I}_{\pi,B}(\rho_{\alpha}, u_{\alpha}) = \liminf_{\alpha} \mathcal{I}_{\pi,B}(\rho_{\alpha}, u_{\alpha} 1_{B}) \geq \mathscr{I}_{\pi,B}(\rho, u 1_{B}) = \mathscr{I}_{\pi,B}(\rho, u).
\end{equation*}
Thus, by the lower semi-continuity established above, we find that $\mathcal{I}_{\pi,B}$ is actually lower semi-continuous with respect to the weak topology on $L^{1}(\pi) \times L^{1}_{\mathrm{loc}}(\pi \otimes m)$.

Since $\psi\geq 0$, the functional $B\mapsto \mathcal{I}_{\pi,B}$ is monotone increasing.
By monotone convergence
\begin{align}\label{e:restricted-fisher}
  \sup_{B\in\mathfrak{B}_0(X)} \mathcal{I}_{\pi,B} = \mathcal{I}_{\pi}.
\end{align}
Thus, as a supremum of lower semi-continuous functions, $\mathcal{I}_{\pi}$ is also lower semi-continuous with respect to the weak topology on $L^{1}(\pi) \times L^{1}_{\mathrm{loc}}(\pi \otimes m)$.

\paragraph{$\mathscr{P}_1(\config)$-lower semicontinuity.} 
Fix $b \geq 0$.
We show that
\begin{equation*}
A_b \coloneq \set*{ \mu =\rho\pi \in \mathscr{P}_{1}(\config) : \Fish{ \mu \given \pi } \leq b }
\end{equation*}
is closed in $\mathscr{P}_{1}(\config)$.
In view of \cref{t:p1-polish}, it suffices to show that it is sequentially closed.
Consider $(\mu_{n}) \subset A_b$, with $\mu_n= \rho_n\pi$, converging to some $\mu \in \mathscr{P}_{1}(\config)$.
By \cref{e:lsi} and a theorem of la Vallée-Poussin \cite[Thm.\ 22, p.\ 38]{DellacherieMeyer}, the set $A_b$ is uniformly integrable when regarded as a subset of $L^1(\pi)$.
Hence, by the Dunford--Pettis Theorem \cite[Thm.\ 25, p.\ 43]{DellacherieMeyer}, the family $(\rho_{n})$ is weakly relatively compact in $L^{1}(\pi)$.
Since $(\mu_{n})$ converges to $\mu$ in $\mathscr{P}_{1}(\config)$, we thus find that there exists $\rho \in L^{1}(\pi)$ with $\mu = \rho \pi$ and $(\rho_{n})$ converges to $\rho$ weakly in $L^{1}(\pi)$.

Let $v \in L^{\infty}(\pi \otimes m)$ such that there exists $B \in \mathfrak{B}_{0}(X)$ with $v = 0$ $\pi \otimes m$-almost everywhere outside of $\config \times B$.
Take $h \in \mathscr{C}_{0}(X)$ such that $0 \leq h \leq 1$ and $h = 1$ on $B$.
In view of \cref{t:p1-compactness}, we find that $(\iota_{h} \rho_{n})$ is uniformly integrable in $L^{1}(\pi)$.
By Dunford--Pettis theorem, the sequence $(\iota_{h} \rho_{n})$ converges weakly in $L^{1}(\pi)$ to $(\iota_{h}\rho)$.
By the Mecke formula, we get that
\begin{equation*}
  \begin{split}
    \int_{\config \times X} v\, \diff \rho_{n} \dd (\pi \otimes m) =&\ \int_{\config} \bracket*{ \int_{B} v(\eta - \delta_{x}, x) \eta(\dd x) - \int_{B} v(\eta, x) m(\dd x) } \rho_{n}(\eta) \pi(\dd \eta)
    \\
    =& \int_{\config} F \iota_{h} \rho_{n} \dd \pi - \int_{\config} G \rho_{n} \dd \pi,
  \end{split}
\end{equation*}
where
\begin{align*}
  & F(\eta) = \int_{B} \frac{v(\eta - \delta_{x}, x)}{\eta(h)}  \eta(\dd x) \leq \norm{v}_{L^{\infty}(\pi \otimes m)}, \\
  & G(\eta) = \int_{B} v(\eta, x) m(\dd x) \leq m(B) \norm{v}_{L^{\infty}(\pi \otimes m)}.
\end{align*}
Thus, both $F$ and $G \in L^{\infty}(\pi)$.
By the weak convergence of $(\rho_{n})_n$ and $(\iota_{h} \rho_{n})_n$, we thus find that $(\diff \rho_{n})_n$ converges weakly in $L^{1}_{\mathrm{loc}}(\pi \otimes m)$.

By weak lower semi-continuity of $\mathcal{I}_{\pi}$ on $L^{1}(\pi) \times L_{\mathrm{loc}}^{1}(\pi \otimes m)$ established in the first part of the proof,
\begin{equation*}
\Fish{\mu \given \pi} = \mathcal{I}_{\pi}(\rho, \diff \rho) \leq \liminf_n \mathcal{I}_{\pi}(\rho_{n}, \diff \rho_{n}) = \liminf_n \Fish{ \mu_{n} \given \pi } \leq b.
\end{equation*}
This shows that $\mu \in A_b$ and concludes the proof.
\end{proof}

By Jensen's inequality both $\Ent{\emparg \given \pi}$ and $\Fish{\emparg \given \pi}$ are decreasing along the dual Ornstein--Uhlenbeck semi-group.
In particular, both $\dom \ent$ and $\dom \fish$ are stable under the action of $\dou$.
For local Dirichlet forms, in a quite general setting, the semigroup maps $L^{2}$ densities to the domain of the Fisher information.
In our non-local setting, similar results are not available.
Thus, we carry out ad hoc computations owing to the explicit formula of the Dirichlet form in the Poisson setting.
\begin{theorem}\label{t:properties-ou}
  Let $\mu \in \dom \ent$ and $t > 0$.
  \begin{enumerate}[$(i)$, wide]
    \item\label{i:regularizing-ou} The Ornstein--Uhlenbeck semi-group is regularizing:
      \begin{equation}\label{e:regularizing-ou}
        \dou_{t}\mu \in \dom \fish.
      \end{equation}
    \item\label{i:production-entropy-ou} The Fisher information controls the entropy production along the Ornstein--Uhlenbeck semi-group:
      \begin{equation}\label{e:production-entropy-ou}
        \Ent{\dou_{t}\mu \given \pi} = \Ent{\mu \given \pi} - \int_{0}^{t} \Fish{\dou_{s} \mu \given \pi} \dd s.
      \end{equation}
    \item\label{i:exponential-convergence-entropy} The Ornstein--Uhlenbeck semi-group converges exponentially fast to equilibrium:
      \begin{equation}\label{e:exponential-convergence-entropy}
        \Ent{\dou_{t} \mu \given \pi} \leq \e^{-t} \Ent{\mu \given \pi}.
      \end{equation}
  \end{enumerate}
\end{theorem}

\begin{remark}
  \cref{i:regularizing-ou,i:production-entropy-ou} are the usual de Bruijn's identity.
  They are classical for diffusions.
  See, for instance, \cite[Prop.\ 5.2.2]{BGL} or \cite[Thm.\ 4.16]{AGSHeat}.
  We provide a proof for Poisson processes, for completeness.
\end{remark}

\begin{proof}
  \cref{i:regularizing-ou,i:production-entropy-ou}
  Assume first that $\rho \in L^{2}(\pi)$.
  As before, write $\theta(s) \coloneq s \log s -s + 1$, for $s \geq 0$ and $\theta(s) \coloneq \infty$ otherwise.
  For all $k \in \mathbb{N}$, set
  \begin{equation}\label{e:theta-k}
    \theta_{k}(s) \coloneq \int_{1}^{s} k \wedge \log r \vee (-k) \dd r, \qquad s \geq 0,
  \end{equation}
  and $\theta_{k}(s) \coloneq \infty$ otherwise.
  Then, $(\theta_{k})$ is an increasing sequence of Lipschitz functions converging to $\theta$.
  Let $t > 0$, by \cref{e:ou-regularization-L2}, $\mathsf{P}_{t}\rho \in \dom \mathsf{L}$.
  Since $\theta_{k}'$ is Lipschitz, we also find that $\theta_{k}'(\mathsf{P}_{t}\rho) \in \dom \mathcal{E}$ \cite[Prop.\ 3.3.1, p.\ 14]{BouleauHirsch}.
  Thus,
  \begin{equation*}
    \int \theta_{k}(\rho) \dd \pi - \int \theta_{k}(\mathsf{P}_{t} \rho) \dd \pi = - \int_{0}^{t} \int \theta_{k}'(\mathsf{P}_{t} \rho) \mathsf{L} \mathsf{P}_{t} \rho \dd \pi \dd t = \int_{0}^{t} \int \diff \theta_{k}'(\mathsf{P}_{t} \rho) \diff \mathsf{P}_{t} \rho \dd \pi \dd m.
  \end{equation*}
  As $k \to \infty$, by monotone convergence, the left-hand side converges to $\Ent{ \mathsf{P}_{t}^{\star} \mu \given \pi } - \Ent{ \mu \given \pi }$.
  Now, we also claim that the right-hand side is also monotone.
  First, by convexity of $\theta_{k}$, we find that the integrand on the right-hand side is non-negative.
  Differentiating twice yields that $\theta_{k+1} - \theta_{k}$ is convex.
  It thus follows, that
  \begin{equation}\label{e:theta-monotone-convex}
    (\theta_{k+1}'(s) - \theta_{k+1}'(r))(s-r) \geq (\theta_{k}'(s) - \theta_{k}'(r))(s-r), \qquad s,\, r \geq 0.
  \end{equation}
  The above formula is the monotonicity of the integrand.
  We obtain \cref{e:production-entropy-ou} by monotone convergence.
  This also gives \cref{e:regularizing-ou} for almost every $t$.
  We conclude it holds for every $t$ by continuity.

  Now we only assume that $\mu = \rho \pi \in \dom \ent$
  For $k \in \mathbb{N}$, let $\mu_{k} \coloneq (\rho \wedge k) \pi / Z_{k}$.
  We explicitly compute
  \begin{equation*}
    \Fish{ \mathsf{P}_{t}^{\star} \mu_{k} \given \pi } = \frac{1}{Z_{k}} \int \diff (\mathsf{P}_{t} \rho \wedge k) \diff \theta'_{k}(\mathsf{P}_{t} \rho) \dd \pi \dd m.
  \end{equation*}
  Similarly to \cref{e:theta-monotone-convex}, we have for $s$ and $r \geq 0$:
  \begin{equation*}
    (\log(s \wedge (k+1)) - \log(r \wedge (k+1)))(s \wedge (k+1) - r \wedge (k+1)) \geq (\log(s \wedge k) - \log(r \wedge k))(s \wedge k - r \wedge k).
  \end{equation*}
  By the previous argument for $L^{2}(\pi)$-densities, we get that
  \begin{equation*}
    \Ent{ \mathsf{P}^{\star}_{t} \mu_{k} \given \pi } - \Ent{ \mu_{k} \given \pi } = \int_{0}^{t} \Fish{ \mathsf{P}_{s}^{\star} \mu_{k} \given \pi } \dd s.
  \end{equation*}
  Since $Z_{k} \to 1$, we conclude by monotone convergence taking $k \to \infty$.

  \cref{i:exponential-convergence-entropy}
  Grönwall lemma together with \cref{i:production-entropy-ou,e:lsi}.
\end{proof}

\begin{remark}
  The statement above and its proof can be immediately extended to functions rather than probability measures.
  For $\rho \in L^{1}(\pi)$, write
  \begin{align*}
    & \ent_{\pi}(\rho) \coloneq \int \rho \log \rho \dd \pi - \int \rho \dd \pi \log \int \rho \dd \pi;
 \\ & \fish_{\pi}(\rho) \coloneq \int \diff \rho \diff \log \rho \dd \pi \dd m.
  \end{align*}
  If $\ent_{\pi}(\rho) < \infty$, then
  \begin{equation*}
    \frac{\dd}{\dd t} \ent_{\pi}(\mathsf{P}_{t}\rho) = - \fish_{\pi}(\mathsf{P}_{t} \rho), \qquad t > 0.
  \end{equation*}
\end{remark}

\section{Continuity equation}\label{s:continuity-equation}

In order to construct a \emph{Riemannian distance}, we first present a notion of \emph{infinitesimal variation} of a curve $\bar{\mu} = (\mu_{t}) \subset \mathscr{P}_{1}(\config)$.
Informally, the variation is obtained through a weak formulation of the \emph{discrete continuity equation} \cref{e:intro-ce}.
In order to give a more rigorous definition let us recall that we write $\mathscr{S}$ for the algebraic linear span of functions of the form $\e^{-\iota_{h}}$, $h \in \mathscr{C}_{0}^{+}(X)$.
For $T > 0$, we say that $\bar{\mu} = (\mu_{t}) \in \mathscr{F}([0,T],\mathscr{P}(\config))$ and $\bar{\nu} = (\nu_{t}) \in \mathscr{F}([0,T], \mathscr{M}_{b,0}(\config \times X))$ solve the \emph{continuity equation} on $[0,T]$ provided
\begin{equation}\label{e:ce}\tag{$\mathbf{CE}_{T}$}
  0 = \int_{0}^{T} \dot{\varphi}(t) \int G \dd \mu_{t} \dd t + \int_{0}^{T} \varphi(t) \int \diff G \dd \nu_{t} \dd t, \qquad  G \in \mathscr{S},\, \varphi \in \mathscr{C}_{c}^{\infty}((0,T)),
\end{equation}
and 
\begin{equation}\label{e:nu-locally-bounded}
  \int_{[0,T]} \abs{\nu_{t}}(\config \times B) \dd t < \infty, \qquad B \in \mathfrak{B}_{0}(X).
\end{equation}
Here, and in all the paper, $\dot{\varphi}$ indicates a time derivative, and we identify $\bar{\nu}$ with a measure on $\config \times X \times [0,1]$, by
\begin{equation*}
  \bar{\nu}(\dd \eta \dd x \dd t) = \int \nu_{t}(\dd \eta \dd x) \dd t.
\end{equation*}
With this identification, \cref{e:nu-locally-bounded} can be written $\bar{\nu} \in \mathscr{M}_{b,0}(\config \times X \times [0,T])$.
Informally, we can say that $\bar{\nu}$ is tangent to the curve $\bar{\mu}$.

\begin{remark}
Contrary to \cref{e:intro-ce}, the curve $\bar{\nu}$ does not depend explicitly on $\bar{\mu}$.
When constructing the distance in \cref{s:variational-distance}, the action functional will automatically select solutions of a particular form.
\end{remark}

Let us start with the following stability property for solutions to the continuity equation.
\begin{lemma}\label{t:ce-stability-limit}
  Let $(\bar{\mu}_{n}, \bar{\nu}_{n})$ be a sequence of solutions to the continuity equation, $\bar{\mu} \in \mathscr{F}([0,T], \mathscr{P}(\config))$, and $\bar{\nu} \in \mathscr{M}_{b,0}(\config \times X \times [0,T])$ such that
  \begin{align*}
    & \mu_{n,t} \xrightarrow[n \to \infty]{\mathscr{P}(\config)} \mu_{t}, \qquad a.e.\ t \in [0,T],
  \\&\bar{\nu}_{n} \xrightarrow[n \to \infty]{\mathscr{M}_{b,0}(\config \times X \times [0,T])} \bar{\nu}.
  \end{align*}
\end{lemma}
\begin{proof}
  The convergence of the first term in the right-hand side of \cref{e:ce} follows from the assumption on $(\bar{\mu}_{n})$ together with the dominated convergence theorem.
  The convergence of the second term in the right-hand side of \cref{e:ce,e:nu-locally-bounded} follow directly from the assumptions on $(\bar{\nu}_{n})$.
\end{proof}

\subsection{Examples of solutions to the continuity equation}

We start with an important example of solutions to the continuity equation built from the dual Ornstein--Uhlenbeck semi-group.
\begin{proposition}\label{t:ou-solution}
  Let $\mu_{0} \coloneq \rho_{0} \pi \in \mathscr{P}_{1}(\config)$.
  For all $t \geq 0$, set 
  \begin{equation*}
  \mu_{t} \coloneq \dou_{t} \mu_{0} = \ou_{t} \rho \pi, \qquad \nu_{t} \coloneq - \diff \ou_{t} \rho\, \dd (\pi \otimes m).
  \end{equation*}
  Then $(\bar{\mu}, \bar{\nu})$ is a solution to the continuity equation.
\end{proposition}
\begin{proof}
  Since $\mu \in \mathscr{P}_{1}(\config)$, by the Mecke identity we have that
  \begin{equation}\label{e:l1-loc}
    \int_{\config} \int_{B} \rho(\eta + \delta_{x}) m(\dd x) \pi(\dd \eta) = I_{\mu}(B) < \infty, \qquad B \in \mathfrak{B}_{0}(X).
  \end{equation}
  
  Thus,
  \begin{equation*}
  \int_{0}^{T} \abs{\nu_{t}}(\config \times B) \dd t = \int_{\config} \int_{B}  \abs{\diff \mathsf{P}_{t} \rho} \dd m \dd \pi \leq \int_{0}^{T} \paren[\big]{ I_{\mu_{t}}(B) + m(B) } \dd t.
  \end{equation*}
  The right-hand side is finite by \cref{e:intensity-Ornstein-Uhlenbeck}.
  This shows that $\bar{\nu}$ satisfies \cref{e:nu-locally-bounded}.
  
  Let $\varphi \in \mathscr{C}_{c}^{\infty}((0,T))$ and $G \in \mathscr{S}$.
  We compute
  \begin{equation*}
      \begin{split}
        \int_{0}^{T} \dot{\varphi}(t) \int G \dd \mu_{t} \dd t &= \int_{0}^{T} \dot{\varphi}(t) \int \mathsf{P}_{t} G \dd \mu_{0} \dd t \\
                                                          &= - \int_{0}^{T} \varphi(t) \int \frac{\dd}{\dd t} \mathsf{P}_{t} G \dd \mu_{0} \dd t \\
                                                          &= - \int_{0}^{T} \varphi(t) \int \mathsf{L} \mathsf{P}_{t} G \dd \mu_{0} \dd t \\
                                                          &= \int_{0}^{T} \varphi(t) \iint \diff G \diff \mathsf{P}_{t} \rho \dd \pi \dd m \dd t.
      \end{split}
  \end{equation*}
  We have used the symmetry of $\ou$ with respect to $\pi$, an integration by parts with respect to the $t$ variable, \cref{e:ou-regularization-L2,e:ou-ipp-generator}.
\end{proof}

The Ornstein--Uhlenbeck flow also preserves solutions of the continuity equation in the following sense.
\begin{proposition}\label{t:stability-ou-solution}
  Let $\varepsilon > 0$.
  Assume that $(\bar{\mu}, \bar{\nu})$ is a solution to the continuity equation.
  Consider the measures given for all $t \in [0,T]$ by
  \begin{equation*}
    \mu^{\varepsilon}_{t} \coloneq \mathsf{P}_{\varepsilon}^{\star} \mu_{t}, \qquad \nu^{\varepsilon}_{t} \coloneq \e^{-\varepsilon} \mathsf{P}_{\varepsilon}^{\star} \nu_{t}.
  \end{equation*}
  Then $(\bar{\mu}^{\varepsilon}, \bar{\nu}^{\varepsilon})$ is also a solution to the continuity equation.
\end{proposition}

\begin{proof}
  Let $\varphi \in \mathscr{C}_{c}^{\infty}$ and $G \in \mathscr{S}$.
  Then, $\ou_{\varepsilon} G \in \mathscr{S}$.
  By \cref{e:ce} for $(\bar{\mu}, \bar{\nu})$
  \begin{equation*}
    0 = \int_{0}^{T} \dot{\varphi}_{t} \int \ou_{\varepsilon} G \dd \mu_{t} \dd t + \int_{0}^{T} \varphi_{t} \int \diff \ou_{\varepsilon} G \dd \nu_{t} \dd t.
  \end{equation*}
  We conclude \cref{e:ce} for $(\bar{\mu}^{\varepsilon}, \bar{\nu}^{\varepsilon})$, since, by \cref{e:be}, $\diff \mathsf{P}_{\varepsilon} G = \e^{-\varepsilon} \mathsf{P}_{\varepsilon} \diff G$.
  
    Since $\dou$ acts on $\bar{\nu}$ only on the first coordinate, if $\bar{\nu}$ satisfies \cref{e:nu-locally-bounded} so does $\dou_{\varepsilon} \bar{\nu}$.
\end{proof}

Solutions to the continuity equation are also invariant under time reparametrization.
\begin{lemma}[{\cite[Lemma 8.1.3]{AGS}}]\label{t:ce-time-reparametrization}
  Consider a strictly increasing and absolutely continuous function $\lambda \colon [0,T'] \to [0,T]$, such that its inverse is also absolutely continuous.
  Then $(\bar{\mu}, \bar{\nu})$ solves \cref{e:ce} if and only if $(\bar{\mu} \circ \lambda, \lambda' \cdot \bar{\nu} \circ \lambda)$ solves the continuity equation on $(0,T')$.
\end{lemma}

  \subsection{Extending the notion of solutions}
  Let $\mathscr{H}$ be the space of all $G \in \mathscr{F}_{b}(\config)$ such that $\diff G \in \mathscr{F}_{b,0}(\config \times X)$.
  \cref{e:ce} makes sense for every $F \in \mathscr{H}$.
  In particular, it is possible to define another notion of solution to the continuity equation by replacing $\mathscr{S}$ by $\mathscr{H}$ in \cref{e:ce}.
  The goal of this section is to shows that it yields the same notion of solution.

  \subsubsection{The algebra of local sets \texorpdfstring{in $\config$}{on the space of configurations}}\label{s:local-algebra}
    Let $B \in \mathfrak{B}_{0}(X)$ be \emph{closed}.
    We write $\mathfrak{A}_{B}(\config)$ for the set of all $A \in \mathfrak{B}(\config)$ such that
    \begin{equation*}
      \forall x \in X \setminus B \qquad \eta \in A \Leftrightarrow \eta + \delta_{x} \in A .
    \end{equation*}
    It is easily verified that $\mathfrak{A}_{B}(\config)$ is a sub-$\sigma$-algebra of $\mathfrak{B}(\config)$ and that $F$ is $\mathfrak{A}_{B}(\config)$-measurable if and only if $\diff F$ vanishes outside of $\config \times B$.
    Let $\config_{B}$ be the set of configurations supported in $B$.
    Since $B$ is closed, $\config_{B}$ is closed subset of $\config$, by the Portmanteau theorem, and thus it is a Polish space.
    We shall need the following lemma.
    Let $\mathfrak{B}_{B}(\config) = \sigma\paren*{\iota_{B'} : B' \in \mathfrak{B}_{B}(X)} $, and write $\mathrm{pr}_{B} \colon \config \to \config_{B}$ for the canonical projection.
    \begin{lemma}\label{t:local-sigma-algebra}
      The following $\sigma$-algebras coincide
      \begin{equation*}
        \mathfrak{B}_{B}(\config) = \mathfrak{A}_{B}(\config) = \mathrm{pr}_{B}^{-1} \mathfrak{B}(\config_{B}).
      \end{equation*}
    \end{lemma}
    \begin{proof}
      Let $B' \in \mathfrak{B}_{B}(X)$.
      Since $\iota_{B'}$ is $\mathfrak{A}_{B}(\config)$-measurable, we find that $\mathfrak{B}_{B}(\config) \subset \mathfrak{A}_{B}(\config)$.
      On the other hand, $\mathfrak{B}(\config)$ is generated by all sets of the form 
      \begin{equation}\label{e:generating-borel-upsilon}
        \set{ \eta(C_{1}) = k_{1}, \dots, \eta(C_{l}) = k_{l} }, \qquad l \in \mathbb{N},\, (C_{i}) \in \mathfrak{B}_{0}(X)^{l},\, (k_{i}) \in \mathbb{N}^{l}.
      \end{equation}
      Since $\mathfrak{A}_{B}(\config)$ is a sub-$\sigma$-algebra of $\mathfrak{B}(\config)$, it is generated by those sets in \cref{e:generating-borel-upsilon} that are also in $\mathfrak{A}_{B}(\config)$.
      It is readily verified that every set $A$ as in \cref{e:generating-borel-upsilon} satisfies $A \in \mathfrak{A}_{B}(\config)$ if and only if $C_{i} \subset B$ for all $i$.
      Thus $\mathfrak{A}_{B}(\config) \subset \mathrm{pr}_{B}^{-1} \mathfrak{B}(\config)$.
      The fact that $\mathfrak{B}_{B}(\config)$ and $\mathrm{pr}_{B}^{-1} \mathfrak{B}(\config)$ coincide is standard.
    \end{proof}
    Finally, let us define the algebra
    \begin{equation*}
      \mathfrak{A}(\config) \coloneq \cup_{B} \mathfrak{A}_{B}(\config).
    \end{equation*}
    The reader can easily verify that $\mathfrak{A}(\config)$ is an algebra but in general not a $\sigma$-algebra.
    We have that $\mathscr{H}$ is the set of $F \in \mathscr{F}_{b}(\config)$ that are also $\mathfrak{A}(\config)$-measurable.

  \subsubsection{The topology of \texorpdfstring{$\mathscr{H}$}{H} and \texorpdfstring{$\mathscr{C}^{1}_{T}(\mathscr{H})$}{H-valued continuously differentiable functions}}
  For all closed $B \in \mathfrak{B}_{0}(X)$, we write $\mathscr{H}_{B}$ for the space of $G \in \mathscr{H}$, such that $\diff G = 0$ outside of $\config \times B$.
  Alternatively, $\mathscr{H}_{B}$ is the set of $F \in \mathscr{F}_{b}(\config)$ that are $\mathfrak{A}_{B}(\config)$-measurable.
  The space $\mathscr{H}_{B}$ is a Banach space for the norm
  \begin{equation*}
    \norm{G}_{\mathscr{H}_{B}} \coloneq \norm{G}_{\mathscr{F}_{b}(\config)} + \norm{\diff G}_{\mathscr{F}_{b}(\config \times X)}.
  \end{equation*}
  The topology on $\mathscr{H}$ is the strict inductive limit in $n \in \mathbb{N}$ of the Banach spaces $\mathscr{H}_{n} = \mathscr{H}_{B(o,n)}$, for any fixed $o \in X$.
  By \cite[Prop.\ 9 (iii), p.\ II.34]{BourbakiEVT}, $\mathscr{H}$ is complete.
  We consider the space $\mathscr{C}_{c}^{1}((0,T), \mathscr{H})$ of continuously differentiable and compactly supported functions $F \colon (0,T) \to \mathscr{H}$.
  In order to equip $\mathscr{C}_{c}^{1}((0,T), \mathscr{H})$ with a suitable topology let us introduce some notation.
  Given a locally convex linear space $E$, we write $\mathscr{C}^{1}_{T}(E) = \mathscr{C}_{c}^{1}((0,T), E)$, and, for $n \in \mathbb{N}$, $\mathscr{C}^{1}_{T,n}(E)$ for the space of those functions $F$ that are supported on $[1/n, T-1/n]$.
  We omit $E$ from the notation when $E = \mathbb{R}$.
  For all $k$ and $n \in \mathbb{N}$, the spaces $\mathscr{C}^{1}_{T,k}(\mathscr{H}_{n})$ are Banach spaces.
  We equip $\mathscr{C}^{1}_{T,k}(\mathscr{H})$ with the strict inductive limit topology in $n \in \mathbb{N}$ and $k$ fixed.
  Then, we equip $\mathscr{C}^{1}_{T}(\mathscr{H})$ with the strict inductive limit topology in $k \in \mathbb{N}$ of the $\mathscr{C}^{1}_{T,k}(\mathscr{H})$.
  This also coincides with the strict inductive limit in $n \in \mathbb{N}$ of $\mathscr{C}^{1}_{T,n}$.
  
  \begin{lemma}\label{t:density-tensor-product}
    The set $\mathscr{H} \otimes \mathscr{C}_{c}^{\infty}((0,T))$ is dense in $\mathscr{C}^{1}_{T}(\mathscr{H})$.
  \end{lemma}
  \begin{proof}
    Let $F \in \mathscr{C}^{1}_{T}(\mathscr{H})$ and $\varepsilon > 0$.
    There exists $n \in \mathbb{N}$ such that $F \in \mathscr{C}^{1}_{T,n}(\mathscr{H}_{n})$.
    Since $\mathscr{H}_{n} \otimes \mathscr{C}_{T,n}^{1}$ is dense in $\mathscr{C}_{T,n}^{1}(\mathscr{H}_{n})$, there exists $F_{\varepsilon} \in \mathscr{H}_{n} \otimes \mathscr{C}_{T,n}^{1} \subset \mathscr{H} \otimes \mathscr{C}_{T}^{1}$ such that
    \begin{equation*}
      \norm{F_{\varepsilon} - F}_{n} \leq \varepsilon.
    \end{equation*}
    Let $p$ be a continuous seminorm on $\mathscr{C}^{1}_{T}(\mathscr{H})$.
    By the universal property of inductive limits \cite[Prop.\ 5, p.\ II.29]{BourbakiEVT}, there exists $c > 0$ such that
    \begin{equation*}
      p(F_{\varepsilon} - F) \leq c \norm{F_{\varepsilon} - F}_{n} \leq c \varepsilon.
    \end{equation*}
    Thus $\mathscr{H} \otimes \mathscr{C}_{T}^{1}$ is dense in $\mathscr{C}_{T}^{1}(\mathscr{H})$.
    We obtain that $\mathscr{H} \otimes \mathscr{C}_{c}^{\infty}((0,T))$ is dense by mollification.
\end{proof}

\subsubsection{The continuity equation holds on \texorpdfstring{$\mathscr{C}^{1}_{T}(\mathscr{H})$}{H-valued continuously differentiable functions}}
  \begin{proposition}\label{t:ce-extended}
    Let $(\bar{\mu}, \bar{\nu})$ be a solution to the continuity equation.
    Then,
    \begin{equation}\label{e:ce-extended}
      \int_{0}^{T} \mu_{t}(\dot{F}_{t}) \dd t + \int_{0}^{T} \nu_{t}(\diff F_{t}) \dd t = 0, \qquad F \in \mathscr{C}^{1}_{T}(\mathscr{H}).
    \end{equation}
  \end{proposition}
  \begin{proof}
    Let $(\bar{\mu}, \bar{\nu})$ be a solution to the continuity equation.
    We split the proof in two parts.
    \paragraph{\cref{e:ce-extended} holds for $F = G \otimes \varphi \in \mathscr{H} \otimes \mathscr{C}_{c}^{\infty}((0,T))$.}
    Let $B \in \mathfrak{B}_{0}(X)$ closed.
    Write $\hat{\mathscr{H}}_{B}$ for the space of functions $G \in \mathscr{H}_{B}$ such that \cref{e:ce} holds for $G \otimes \varphi$, for all $\varphi \in \mathscr{C}_{c}^{\infty}((0,T))$.
    Since \cref{e:ce} is linear with respect to $G$, $\hat{\mathscr{H}}_{B}$ is a linear space containing constants.

    Take $(G_{n}) \subset \hat{\mathscr{H}}_{B}$ converging uniformly to some $G$.
    Firstly, since $\mathscr{H}_{B}$ is a Banach space for the uniform convergence, $G \in \mathscr{H}_{B}$.
    Secondly, we have that $G_{n} \to G$ uniformly on $\config$ and $\diff G_{n} \to \diff G$ uniformly on $\config \times B$.
    Thus, applying \cref{e:ce} to $G_{n} \otimes \varphi$, passing to the limit, and invoking Lebesgue dominated convergence theorem, we find that $G \otimes \varphi$ solves \cref{e:ce}.
    This shows that $G \in \hat{\mathscr{H}}_{B}$, and that $\hat{\mathscr{H}}_{B}$ is closed under uniform convergence.

    Take $(G_{n})\subset \hat{\mathscr{H}}_{B}$ an increasing and bounded sequence of non-negative functions.
    Write $G = \lim_{n} G_{n}$.
    By monotone convergence, we get that
    \begin{equation*}
      \int_{0}^{T} \dot{\varphi}(t) \mu_{t}(G_{n}) \dd t \xrightarrow[n \to \infty]{} \int_{0}^{T} \dot{\varphi}(t) \mu_{t}(G) \dd t.
    \end{equation*}
    By \cref{e:nu-locally-bounded} and definition of $\mathscr{H}_{B}$, $\abs{\diff G_{n} \otimes \varphi} \leq c 1_{\config \times B \times [0,T]} \in L^{1}(\bar{\nu})$.
    Thus, by dominated convergence,
    \begin{equation*}
      \bar{\nu}(\diff G_{n} \otimes \varphi) \xrightarrow[n \to \infty]{} \bar{\nu}(\diff G \otimes \varphi).
    \end{equation*}
    This shows that $G \in \hat{\mathscr{H}}_{B}$, and that $\hat{\mathscr{H}}_{B}$ is stable under uniformly bounded monotone convergence.

    Thus, $\hat{\mathscr{H}}_{B}$ satisfies the assumptions of the monotone class theorem \cite[Thm.\ 21, p.\ 20]{DellacherieMeyer}.
    Let $\mathscr{S}_{B}$ be the linear span of functions of the form $\e^{-\iota_{h}}$ for $h \in \mathscr{C}_{b,B}(X)$.
    By construction, $\mathscr{S}_{B} \subset \hat{\mathscr{H}}_{B}$ and $\mathscr{S}_{B}$ is stable by multiplication.
    Thus, $\hat{\mathscr{H}}_{B}$ contains all the bounded functions measurable with respect to the $\sigma$-algebra generated by $\mathscr{S}_{B}$.
    An argument similar to that of \cite[Lem.\ 2]{LastAnaSto} shows that this $\sigma$-algebra contains all the $\iota_{h}$ for $h \in \mathscr{C}_{b,B}(X)$ and $B' \subset B$.
    By \cref{t:local-sigma-algebra}, this $\sigma$-algebra is $\mathfrak{A}_{B}(\config)$.
    This shows that $\hat{\mathscr{H}}_{B} = \mathscr{H}_{B}$.

    Take $G \in \mathscr{H}$.
    By definition, there exists $B \in \mathfrak{B}_{0}(X)$ such that $G \in \mathscr{H}_{B}$.
    We conclude by the first part.

    \paragraph{\cref{e:ce-extended} holds for $F \in \mathscr{C}_{T}^{1}(\mathscr{H})$.}
    By \cref{t:density-tensor-product}, we can find $(F_{n}) \subset \mathscr{H} \otimes \mathscr{C}_{c}^{\infty}((0,T))$ converging to $F$ in $\mathscr{C}_{T}^{1}(\mathscr{H})$.
    By the previous part of the proof, we have that
    \begin{equation*}
      \bar{\mu}(\dot{F}_{n}) + \bar{\nu}(F_{n}) = 0.
    \end{equation*}
    By definition of the convergence on $\mathscr{C}_{T}^{1}(\mathscr{H})$, we can apply dominated convergence to conclude.
  \end{proof}

  \subsection{Properties of the continuity equation}
  In this section we obtain several results concerning the evolution of certain quantities along the continuity equation.
  All the results are a consequence of the following simple observation.
  \begin{lemma}\label{t:bound-difference-measure}
    Take $G \in \mathscr{H}$ and $B_{G} \in \mathfrak{B}_{0}(X)$ so that $\diff G = 0$ outside of $\config \times B_{G}$.
    Assume that $\varphi \otimes G$ satisfy \cref{e:ce} for all $\varphi \in \mathscr{C}_{c}^{\infty}((0,T))$.
    Then, there exists $L_{G} \in \mathfrak{B}((0,T))$ of full measure such that
    \begin{equation}\label{e:bound-difference-measure}
      \mu_{t}(G) - \mu_{s}(G) \leq 2 \norm{G}_{\infty} \int_{s}^{t} \abs{\nu_{r}} (\config \times B_G)\dd r, \qquad t,\, s \in L_{G}.
    \end{equation}
  \end{lemma}

  \begin{proof}
    The assumptions ensure that $t \mapsto \mu_{t}(G) \in W^{1,1}(0,T)$ with distributional derivative given by
  \begin{equation*}
    \frac{\dd}{\dd t} \mu_{t}(G) = \nu_{t}(\diff G), \qquad t \in [0,T].
  \end{equation*}
  For short, we write $N_{t}(B) = \abs{\nu_{t}}(\config \times B)$ for $t \in [0,T]$ and $B \in \mathfrak{B}_{0}(X)$.
  By assumption, there exists $B_{G} \in \mathfrak{B}_{0}(X)$ such that $\diff G = 0$ outside of $\config \times B_{G}$.
  We then have that
  \begin{equation*}
    \abs{\dot{\mu}_{t}(G)} \leq 2N_{t}(B_{G}) \norm{\diff G}_{\infty} \leq N_{t}(B_{G}) \norm{G}_{\infty}, \qquad t \in [0,T].
  \end{equation*}
  By Lebesgue differentiation theorem, there exists $L_{G} \subset \mathfrak{B}(0,T)$ of full measure such that 
  \begin{equation*}
    \dot{\mu}_{t}(G) = \lim_{\varepsilon \to 0} \frac{1}{2 \varepsilon} \int_{t-\varepsilon}^{t+\varepsilon} \mu_{s}(G) \dd s, \qquad t \in L_{G}.
  \end{equation*}
  This gives \cref{e:bound-difference-measure} and concludes the proof. 
\end{proof}
  \subsubsection{The intensity measure along the continuity equation}
  A first application of this result is the following control on the intensity measure.
\begin{theorem}\label{t:ce-intensity-measure}
    Let $(\bar{\mu}, \bar{\nu})$ be a solution to the continuity equation  with $\mu_{0} \in \mathscr{P}_{1}(\config)$.
    Then, for almost every $t \in [0,T]$, $\mu_{t} \in \mathscr{P}_{1}(\config)$ and
    \begin{equation*}
      I_{\mu_{t}}(B) = I_{\mu_{0}}(B) + \int_{0}^{t} \nu_{s}(\config \times B) \dd s, \qquad B \in \mathfrak{B}_{0}(X).
    \end{equation*}
  \end{theorem}

  \begin{proof}
    Let $h \in \mathscr{F}_{0}(X)$.
    By \cref{t:ce-extended}, we have that the continuity equation holds for $\iota_{h} \otimes \varphi$, $\varphi \in \mathscr{C}_{c}^{\infty}((0,T))$.
    Take $(h_{k}) \subset \mathscr{C}_{0}(X)$ as in \cref{t:prohorov} \cref{i:prohorov-polish}.
    For all $k \in \mathbb{N}$, take $B_{k} \in \mathfrak{B}_{0}(X)$ such that $h_{k} = 0$ outside of $B_{k}$.
    We set
    \begin{equation*}
      a_{k} \coloneq 2^{-k} \paren*{ 1 \wedge \abs{\bar{\nu}}\paren*{\config \times B_{k} \times [0,T]}^{-1} }.
    \end{equation*}
    By construction of the $h_{k}$'s, the distance
    \begin{equation*}
      \rho(\lambda, \sigma) \coloneq \sum_{k \in \mathbb{N}} a_{k} \abs{(\lambda - \sigma)(h_{k})}, \qquad \lambda,\, \sigma \in \mathscr{M}_{0}^{+}(X),
    \end{equation*}
    metrizes the vague topology on $\mathscr{M}_{0}^{+}(X)$.
   
    Now, we invoke \cref{t:bound-difference-measure}, with $G_{k} = \iota_{h_{k}}$ and $B_{G_{k}} = B_{k}$.
    This yields a set $L \coloneqq \cap_{k} L_{G_{k}}$ of full measure, such that
    \begin{equation*}
      \rho(I_{\mu_{t}}, I_{\mu_{s}}) \leq c \int_{s}^{t} \sum_{k \in \mathbb{N}} a_{k} \abs{\nu_{r}}(\config \times B_{r}) \leq c \abs{t -s}, \qquad s,\,t \in L.
    \end{equation*}
    This shows that $t \mapsto I_{\mu_{t}} \in \mathscr{M}_{0}^{+}(X)$ is uniformly continuous on the dense set $L \subset [0,T]$.
    By the theorem of continuation of uniformly continuous maps \cite[II, p.\ 20,Thm.\ 2]{BourbakiTopo}, we can extend it to a continuous map $\sigma \colon [0,T] \to \mathscr{M}_{0}^{+}(X)$.

    Since, $\sigma_{t} = I_{\mu_{t}}$ for almost every $t \in [0,T]$, we get that
    \begin{equation*}
      \int_{0}^{T} \dot{\varphi}(t) \sigma_{t}(h) \dd t + \int_{0}^{T} \varphi(t) \sigma_{t}(1 \otimes h) \dd t = 0, \qquad h \in \mathscr{F}_{0}(X),\ \varphi \in \mathscr{C}_{c}^{\infty}((0,T)).
    \end{equation*}
    Taking a sequence $(\varphi_{l}) \subset \mathscr{C}_{c}^{\infty}((0,1))$ such that, as $l \to \infty$, $\varphi_{l} \to 1_{[t_{0}, t_{1}]}$ and $\dot{\varphi}_{l} \to \delta_{t_{0}} - \delta_{t_{1}}$, we thus obtain that
    \begin{equation*}
      \sigma_{t_{1}}(h) = \sigma_{t_{0}}(h) + \int_{0}^{t} \nu_{s}(\config \otimes h), \qquad h \in \mathscr{F}_{0}(X).
    \end{equation*}
    The claim follows immediately.
  \end{proof}

  \subsubsection{Existence of continuous solutions}
\begin{theorem}\label{t:ce-continuous-representative}
    Every solution $(\bar{\mu}', \bar{\nu})$ to the continuity equation with $\mu_{0} \in \mathscr{P}_{1}(\config)$ admits a representative $(\bar{\mu}, \bar{\nu})$ such that $[0,T] \ni t \mapsto \mu_{t} \in \mathscr{P}_{1}(\config)$ is continuous.
    Moreover, for all $t_{0}$ and $t_{1} \in [0,T]$:
    \begin{equation}\label{e:ce-integral-representation}
      \mu_{t_{1}}(F_{t_{1}}) - \mu_{t_{0}}(F_{t_{0}}) = \int_{t_{0}}^{t_{1}} \mu_{t}(\dot{F}_{t}) + \nu_{t}(\diff F_{t}) \dd t, \qquad F \in \mathscr{C}^{1}([0,T], \mathscr{H}).
    \end{equation}
  \end{theorem}

  \begin{proof}
    We consider the non-negative measure
    \begin{equation*}
      \lambda(B) = I_{\mu_{0}}(B) + \int_{0}^{T} \abs{\nu_{s}}(\config \times B) \dd s, \qquad B \in \mathfrak{B}(X).
    \end{equation*}
    In view of \cref{t:weak-convergence-countable,t:ce-intensity-measure}, we find a countable set $(G_{k}) = \mathscr{G}^{\lambda} \subset \mathscr{G}$ such that on $\set{ \mu'_{t} : t \in [0,T]}$ the topology of $\mathscr{P}(\config)$ is induced by that of the simple convergence on $\mathscr{G}^{\lambda}$.
    For all $k \in \mathbb{N}$, write $B_{k}$ for a bounded set such that $G_{k} \in \mathscr{H}_{B_{k}}$, and set
    \begin{equation*}
      \begin{split}
        & b_{k} = 2^{-k} \paren*{ 1 \wedge \abs{\bar{\nu}}\paren*{\config \times B_{k} \times [0,T]}^{-1} }, \\
        & \delta(\mu'_{t}, \mu'_{s}) = \sum_{k \in \mathbb{N}} b_{k} \abs*{(\mu_{t} - \mu_{s})(G_{k})}, \qquad t,\,s \in [0,T].
      \end{split}
    \end{equation*}
    Then $\delta$ is a distance on $(\mu'_{t})$ metrizing the topology of $\mathscr{P}(\config)$.
    Invoking \cref{t:bound-difference-measure} and arguing as in the proof of \cref{t:ce-intensity-measure} shows that on the dense subset $L = \cap L_{G_{k}}$, the map $t \mapsto \mu'_{t} \in \mathscr{P}(\config)$ is uniformly continuous with respect to $\delta$.
    We can then extend it to a continuous map $t \mapsto \mu_{t} \in \mathscr{P}(\config)$.
    The fact that $\bar{\mu}$ actually takes its values in $\mathscr{P}_{1}(\config)$ and is continuous is a consequence of \cref{t:ce-intensity-measure,t:p1-convergence}.
    Formula \cref{e:ce-integral-representation} is obtained for functions $F \in \mathscr{C}_{T}^{1}(\mathscr{H})$ from \cref{e:ce-extended} and by considering a sequence of smooth functions on $(t_{0}, t_{1})$ and converging to $1_{(t_{0}, t_{1})}$ and whose derivatives converges to $\delta_{t_{0}} - \delta_{t_{1}}$ in the sense of distributions (see \cite[Lem.\ 3.1]{Erbar} for details).
    This extends to $F \in \mathscr{C}^{1}([0,T], \mathscr{H})$ by approximation.
\end{proof}

\begin{corollary}\label{t:ce-one-function}
  If $(\bar{\mu}, \bar{\nu})$ is a solution with $\mu_{0} \in \mathscr{P}_{1}(\config)$, we have that
  \begin{equation}\label{e:ce-integral-representation-G}
    \mu_{t}(F) = \mu_{0}(F) + \int_{0}^{t} \nu_{s}(\diff F) \dd s, \qquad F \in \mathscr{H}.
  \end{equation}
\end{corollary}
\begin{proof}
  Apply \cref{e:ce-integral-representation} with $F_{t} = F$ for all $t \in [0,T]$.
\end{proof}

\subsubsection{The relative entropy along the continuity equation}

In \cref{e:production-entropy-ou}, we have that the Fisher information controls the entropy production along $\mathsf{P}^{\star}$.
A similar result holds for the entropy along the continuity equation.
\begin{theorem}\label{t:ce-entropy-production}
  Let $(\bar{\mu}, \bar{\nu})$ be a solution to the continuity equation such that, for all $t \in [0,T]$, $\mu_{t} = \rho_{t} \pi \in \dom \ent$ and $\nu_{t} = w_{t} (\pi \otimes m)$,  and
  \begin{equation}\label{e:ce-bound-fish-action}
    \int_{0}^{T} \Fish{ \mu_{t} \given \pi } \dd t + \int_{0}^{T} \int \abs{w_{t}}^{2} \frac{\diff \log \rho_{t}}{\diff \rho_{t}} \dd t < \infty.
  \end{equation}
  Then, for all $t \in [0,T]$:
  \begin{equation}\label{e:ce-entropy-production}
    \Ent{\mu_{t} \given \pi} - \Ent{ \mu_{0} \given \pi } = \int_{0}^{t} \int \diff \log \rho_{s} \dd \nu_{s} \dd s.
  \end{equation}
\end{theorem}

\begin{remark}
  Let us comment on the assumption \cref{e:ce-bound-fish-action}.
  First of all by the Cauchy--Schwarz inequality this ensures that $\diff \log \rho \in L^{1}(\bar{\nu})$, so that the right-hand side of \cref{e:ce-entropy-production} is well-defined.
  Secondly, the condition on the Fisher information is not very restrictive.
  Indeed, if we start with a solution of the continuity equation in $\dom \ent$, then by \cref{t:ou-solution} we can always perturb it by the Ornstein--Uhlenbeck semi-group in order to have a solution satisfying the finiteness of the Fisher entropy by \cref{e:production-entropy-ou}.
  Lastly, the condition involving the second integral in \cref{e:ce-bound-fish-action} might seem more exotic.
  However, this quantity plays a natural role in the definition of the action and the variational distance in the next section.
\end{remark}

\begin{remark}
  Let us consider $(\bar{\mu}, \bar{\nu})$ a solution to the continuity equation given by the dual Ornstein--Uhlenbeck semi-group, as in \cref{t:ou-solution}.
  In this case,
  \begin{equation*}
    \int_{0}^{t} \int \diff \log \rho_{s} \dd \nu_{s} \dd s = - \int_{0}^{t} \Fish{ \mathsf{P}^{\star}_{s} \mu_{0} \given \pi } \dd s,
  \end{equation*}
  and formula \cref{e:production-entropy-ou} regarding the entropy production along the Ornstein--Uhlenbeck semi-group coincides with \cref{e:ce-entropy-production}.
\end{remark}

\begin{proof}
  For convenience, we first give a short heuristic proof of the statement that goes back at least to the seminal work of \cite{OttoVillani}.
  We thus assume that $\rho \in \mathscr{C}^{1}([0,T], \mathscr{H})$, with $\rho$ bounded away from $0$.
  Since
  \begin{equation*}
    \diff \log \rho  = \log (\rho + \diff \rho) - \log \rho,
  \end{equation*}
  we find that $\log \rho$ is also in $\mathscr{C}^{1}([0,T], \mathscr{H})$.
  Applying \cref{e:ce-integral-representation} to $F = \log \rho$ yields
  \begin{equation*}
    \Ent{ \mu_{t} \given \pi } - \Ent{ \mu_{0} \given \pi } = \int_{0}^{t} \int \dot{\rho}_{s} \dd \pi \dd s + \int_{0}^{t} \int \diff \log \rho_{s} \dd \nu_{s} \dd s.
  \end{equation*}
Since $\rho_{t}$ is a probability density for all $t$, $\int \dot{\rho}_{s} \dd \pi = 0$.
This shows the claim in this case.
The rest of the proof formalizes this idea for general densities.
We stress however that all the ideas are contained in this short argument.

Now, we only assume that $\diff \log \rho \in L^{1}(\bar{\nu})$.
We shall need two stability results for solutions to the continuity equation under regularization.
\paragraph{Stability of the continuity equation under time regularization.}
Let $\psi$ be smooth, compactly supported, non-negative, symmetric mollifier on $\mathbb{R}$, and $\varepsilon > 0$.
We define
\begin{align*}
  & \psi_{\varepsilon} \coloneq \frac{1}{\varepsilon} \psi\paren*{\frac{\cdot}{\varepsilon}} \\
  & \rho^{\varepsilon}_{t} \coloneq \int_{0}^{T} \rho_{\tau} \psi_{\varepsilon}(t-\tau) \dd \tau, \\
  & w^{\varepsilon}_{t} \coloneq \int_{0}^{T} w_{\tau} \psi_{\varepsilon}(t-\tau) \dd \tau.
\end{align*}
Then $\bar{\rho}^{\varepsilon} \in \mathscr{C}^{1}([0,T], L^{1}(\pi))$.
Setting $\mu^{\varepsilon}_{t} \coloneq \rho^{\varepsilon}_{t} \pi$ and $\nu^{\varepsilon}_{t} \coloneq w^{\varepsilon}_{t} (\pi \otimes m)$, we also have that $(\bar{\mu}^{\varepsilon}, \bar{\nu}^{\varepsilon})$ solves the continuity equation.
Indeed taking $F \in \mathscr{C}_{T}(\mathscr{H})$, and letting
\begin{equation*}
F_{\tau}^{\varepsilon} \coloneq \frac{1}{\varepsilon} \int_{0}^{T} F_{t} \psi\paren*{\frac{\tau - t}{\varepsilon}} \dd t,
\end{equation*}
we have that $F^{\varepsilon}_{\tau} \in \mathscr{C}_{T}(\mathscr{H})$ for all sufficiently small $\varepsilon > 0$, and
\begin{align*}
  \int_{0}^{T} \dot{F}_{t} \dd \mu^{\varepsilon}_{t} \dd t &= \frac{1}{\varepsilon} \int_{0}^{T} \int_{0}^{T} \int \dot{F}_{t} \psi\paren*{\frac{t-\tau}{\varepsilon}} \rho_{\tau} \dd \pi \dd t \dd \tau \\
                                                           &= \int_{0}^{T} \int \dot{F}_{\tau}^{\varepsilon} \rho_{\tau} \dd \pi \dd \tau \\
                                                           &= - \int_{0}^{T} \int \diff F^{\varepsilon}_{\tau} w_{\tau} \dd(\pi \otimes m) \dd \tau \\
                                                           &= - \int_{0}^{T} \diff F_{t} w^{\varepsilon}_{t} \dd(\pi \otimes m) \dd t.
\end{align*}
Since $F \in \mathscr{C}_{T}(\mathscr{H})$ is arbitrary, \cref{e:ce} holds for $(\bar{\mu}^{\varepsilon}, \bar{\nu}^{\varepsilon})$.
Moreover, by construction, $\bar{\nu}^{\varepsilon}$ satisfies \cref{e:nu-locally-bounded}.
This shows that $(\bar{\mu}^{\varepsilon}, \bar{\nu}^{\varepsilon})$ is a solution to the continuity equation.

\paragraph{Stability of the continuity equation under space regularization.}
Now, fix $B \in \mathfrak{B}_{0}(X)$, and define
\begin{align*}
  & \rho^{B}_{t} \coloneq \EspPi*{ \rho_{t} \given \mathfrak{A}_{B}(\config) }, \\
  & w^{B}_{t} \coloneq \EspPim*{ w_{t} 1_{\config \times B} \given \mathfrak{A}_{B}(\config) \otimes \mathfrak{B}_{B}(X) }.
\end{align*}
See, for instance, \cite[\S\S 39--43, pp. 36--43]{DellacherieMeyerMartingales} for reminders on conditional expectations and martingales with respect to $\sigma$-finite measures.
In a more prosaic way, we have that
\begin{equation}\label{e:w-conditional-x}
  w^{B}_{t}(\cdot, x) = 1_{B}(x)  \EspPi*{ w_{t}(\cdot, x) \given \mathfrak{A}_{B}(\config) }, \qquad x \in X.
\end{equation}
In view of the independence property of Poisson point processes, we have the explicit formula:
\begin{equation*}
  \rho^{B}_{t}(\eta) = \int \rho_{t}(\eta_{\restriction B} + \xi) \pi_{X \setminus B}(\dd \xi),
\end{equation*}
where, for $C \in \mathfrak{B}(X)$, $\pi_{C}$ is a Poisson point process with intensity $m_{\restriction C}$.
We let $\mu_{t}^{B} = \rho_{t}^{B} \pi$ and $\nu_{t}^{B} = w_{t}^{B} (\pi \otimes m)$, and we claim that $(\bar{\mu}^{B}, \bar{\nu}^{B})$ is a solution to the continuity equation.
By the tower property of conditional expectation
\begin{equation*}
  \abs{\bar{\nu}^{B}}(\config \times X \times [0,T]) = \int_{0}^{T} \EspPim*{ \abs{w_{t}^{B}} } \dd t \leq \int_{0}^{T} \EspPim*{ \abs{w_{t} 1_{\config \times B}} } \dd t =  \abs{\bar{\nu}}(\config \times B \times [0,T]) < \infty.
\end{equation*}
Thus, $\bar{\nu}^{B}$ satisfies \cref{e:nu-locally-bounded}.
Now, let $u$ be bounded and $\mathfrak{A}_{B}(\config) \otimes \mathfrak{B}_{B}(X)$-measurable.
In view, of the explicit formula
\begin{equation*}
  \skoro \paren*{ u 1_{\config \times B} }(\eta) = \int_{B} u(\eta - \delta_{x}, x) \eta(\dd x) - \int_{B} u(\eta, x) m(\dd x),
\end{equation*}
we find that, for $y \not\in B$,
\begin{equation*}
  \skoro( 1_{\config \times B} u)(\eta +\delta_{y}) = \int_{B}(\eta - \delta_{x} + \delta_{y}, x) (\eta + \delta_{y})(\dd x) - \int_{B} u(\eta+\delta_{y}, x) m(\dd x) = \skoro \paren*{ u 1_{\config \times B} }(\eta).
\end{equation*}
Thus, $\skoro(1_{\config \times B} u)$ is $\mathfrak{A}_{B}(\config)$-measurable.
For $F \in \mathscr{H}$, by the Mecke formula, we thus find that
\begin{align*}
  \int \diff F 1_{\config \times B} u \dd \pi \dd m &= \int F \skoro \paren*{ 1_{\config \times B} u } \dd \pi \\
                                                    &= \int \EspPi*{ F \given \mathfrak{A}_{B}(\config) } \skoro \paren*{ 1_{\config \times B} u } \dd \pi \\
                                                    &= \int \diff \EspPi*{ F \given \mathfrak{A}_{B}(\config) } u \dd \pi \dd m.
\end{align*}
Since $u$ was arbitrary,
\begin{equation*}
  \EspPim*{ \diff F \given \mathfrak{A}_{B}(\config) \otimes \mathfrak{B}_{B}(X) } 1_{\config \times B} = \diff \EspPi*{ F \given \mathfrak{A}_{B} }.
\end{equation*}
Thus, for $F \in \mathscr{C}_{T}(\mathscr{H})$,
\begin{align*}
  \int_{0}^{T} \int \dot{F}_{t} \dd \mu^{B}_{t} \dd t &= \int_{0}^{T} \int \EspPi*{ \dot{F}_{t} \given \mathfrak{A}_{B}(\config) } \dd \mu_{t} \dd t \\
                                                  &= - \int_{0}^{T} \int \diff \EspPi*{ F_{t} \given \mathfrak{A}_{B}(\config) } w_{t} \dd \pi \dd m \dd t \\
                                                  &= - \int_{0}^{T} \int \diff F_{t} \EspPim*{ w_{t} 1_{\config \times B} \given \mathfrak{A}_{B}(\config) \otimes \mathfrak{B}_{B}(X) }.
\end{align*}
This shows that $(\bar{\mu}^{B}, \bar{\nu}^{B})$ solves \cref{e:ce}.

\paragraph{Combining the two regularizations.}
Now we define
\begin{align*}
  & \rho^{\varepsilon,B}_{t} = \EspPi*{ \rho^{\varepsilon}_{t} \given \mathfrak{A}_{B}(\config) }, \\
  & w^{\varepsilon,B}_{t} = \EspPim*{ w^{\varepsilon}_{t} 1_{\config \times B} \given \mathfrak{A}_{B}(\config) \otimes \mathfrak{B}_{B}(X) }.
\end{align*}
We also consider the two associated measures $(\bar{\mu}^{\varepsilon, B}, \bar{\nu}^{\varepsilon,B})$.
Note that the two regularizations commute, that is we would get the same objects by first applying the regularization in space and then in time.
From what precedes, we have that $(\bar{\mu}^{\varepsilon, B}, \bar{\nu}^{\varepsilon,B})$ is a solution to the continuity equation.
Differentiating under the integral sign, we get that $\bar{\rho}^{\varepsilon, B} \in \mathscr{C}^{1}([0,T], L^{1}(\pi))$.
The two previous facts show that $\dot{\rho}^{\varepsilon,B}_{t} = \skoro w_{t}^{\varepsilon,B}$.
Fix $k \in \mathbb{N}$, recall $\theta_{k}$ defined in \cref{e:theta-k}.
We then find that
\begin{align*}
  \int \theta_{k}(\rho^{\varepsilon,B}_{t}) \dd \pi - \int \theta_{k}(\rho^{\varepsilon,B}_{0}) \dd \pi &= \int_{0}^{t} \int \theta_{k}'(\rho_{s}^{\varepsilon,B}) \skoro w_{s}^{\varepsilon, B} \dd \pi \dd s \\
                                                                                                                               &= \int_{0}^{t} \iint \diff \theta_{k}'(\rho_{s}^{\varepsilon,B}) w_{s}^{\varepsilon, B} \dd \pi \dd m \dd s.
\end{align*}
As $\varepsilon \to 0$, we have that $\bar{\rho}^{\varepsilon,B} \to \bar{\rho}^{B}$ in $\mathscr{C}^{0}([0,T], L^{1}(\pi))$, and $w_{s}^{\varepsilon,B} \to w_{s}^{B}$ in $L^{1}(\pi \otimes m)$ for all almost every $s \in [0,T]$.
Thus by dominated convergence, we get that
\begin{equation*}
  \int \theta_{k}(\rho^{B}_{t}) \dd \pi - \int \theta_{k}(\rho^{B}_{0}) \dd \pi = \int_{0}^{t} \iint \diff \theta_{k}'(\rho_{s}^{B}) w_{s}^{B} \dd \pi \dd m \dd s = \int_{0}^{t} \int \theta'_{k}(\rho^{B}_{s}) \skoro w^{B}_{s} \dd \pi \dd s.
\end{equation*}
By monotone convergence as $k \to \infty$, we find that
\begin{equation}\label{e:ce-entropy-production-B}
  \Ent{ \mu_{t}^{B} \given \pi } - \Ent{ \mu_{0}^{B} \given \pi } = \int_{0}^{t} \int \log \rho^{B}_{s} \skoro w^{B}_{s} \dd \pi \dd s = \int _{0}^{t} \iint \diff \log \rho^{B}_{s} w^{B}_{s} \dd \pi \dd m \dd s.
\end{equation}
By the theorem of almost sure convergence of martingales, we find that $\rho^{B}_{t} \to \rho_{t}$ almost surely as $B \to X$.
By \cite[Eq.\ 103.1, p.\ 186]{DellacherieMeyerMartingales}, we have that $\sup_{B} \rho^{B}_{t} \in L \log L(\pi)$.
Thus, the martingale also converges in $L \log L(\pi)$ by dominated convergence.
It follows that we can take the limit in the left-hand side of \cref{e:ce-entropy-production-B}.

We now show that we can also pass to the limit in the right-hand side.
First of all, by the theorem of almost sure convergence of martingales which also holds for $\sigma$-finite measures \cite[\S 41, p.\ 37]{DellacherieMeyerMartingales}, we have that $w^{B} \to w$ almost surely.
Thus in order to conclude it suffices to show that $(\diff \log \rho^{B} w^{B})$ is uniformly integrable in $L^{1}(\pi \otimes m \otimes \dd t)$.
Firstly, by the convexity of $(s, t) \mapsto (\log s - \log t)(s -t)$ and Jensen's inequality for conditional expectation
\begin{equation*}
  \diff_{x} \log \rho^{B}_{t} \diff_{x} \rho^{B}_{t} \leq \EspPi*{ \diff_{x} \log \rho_{t} \diff_{x} \rho_{t} \given \mathfrak{A}_{B}(\config)}.
\end{equation*}
Secondly, by the convexity of $(w, s, t) \mapsto w^{2} (\log s - \log t) / (s -t)$, Jensen's inequality for conditional expectation, and \cref{e:w-conditional-x}
\begin{equation*}
\paren*{ \abs*{w_{t}^{B}(\cdot, x)}^{2} \frac{\diff_{x} \log \rho^{B}_{t}}{\diff_{x} \rho^{B}_{t}} } \leq \EspPi*{ w_{t}^{2}(\cdot, x) \frac{\diff_{x} \log \rho_{t}}{\diff_{x} \rho_{t}} \given \mathfrak{A}_{B}(\config) }.
\end{equation*}
Finally, writing
\begin{equation*}
  \abs{ \diff_{x} \log \rho_{t}^{B} w^{B}_{t}(\cdot, x) }^{2} = \paren{ \diff_{x} \log \rho^{B}_{t} \diff_{x} \rho_{t}^{B} } \paren*{ \abs*{w_{t}^{B}(\cdot, x)}^{2} \frac{\diff_{x} \log \rho^{B}_{t}}{\diff_{x} \rho^{B}_{t}} },
\end{equation*}
and using the two previous inequalities together with $2ab \leq a^{2} + b^{2}$ yields
\begin{equation*}
  \abs{ \diff_{x} \log \rho^{B}_{t} w^{B}_{t} } \leq \frac{1}{2} \paren*{ \EspPi*{ \diff_{x} \log \rho_{t} \diff_{x} \rho_{t} \given \mathfrak{A}_{B}(\config) } + \EspPi*{ w_{t}^{2}(\cdot, x) \frac{\diff_{x} \log \rho_{t}}{\diff_{x} \rho_{t}} \given \mathfrak{A}_{B}(\config) } }.
\end{equation*}
Since this holds for all $x \in B$ and all $t \in [0,T]$, we actually have shown that
\begin{equation*}
    \abs{ \diff \log \rho^{B} w^{B} } \leq \EspPimt*{ \frac{1}{2} \paren*{ \diff \log \rho \diff \rho + w^{2} \frac{\diff \log \rho}{\diff \rho} } \given \mathfrak{A}_{B}(\config) \otimes \mathfrak{B}_{B}(X) \otimes \mathfrak{B}(0,T) }.
\end{equation*}
By \cite[Thm.\ 41.1, p.\ 38]{DellacherieMeyerMartingales} and \cref{e:ce-bound-fish-action}, the right-hand side is the sum of two uniformly integrable martingales and is thus uniformly integrable.
\end{proof}

\section{Synthetic Ricci curvature bounds on the Poisson space}

\subsection{A variational distance on the Poisson space}\label{s:variational-distance}
\subsubsection{The Lagrangian functional}
In view of what precedes, it is natural to consider vector fields to be elements of $\mathscr{M}_{b,0}(\config \times X)$.
Let us define the \emph{length} of the tangent vector $\nu$ at $\mu$.
We set
\begin{equation*}
  \theta(s,t) \coloneq \frac{s-t}{\log s - \log t}, \qquad s,t \in \real_{+},
\end{equation*}
and
\begin{equation*}
  \alpha(s,t,w) \coloneq \frac{\abs{w}^{2}}{\theta(s,t)}, \qquad w \in \real,\, s,t \in \real_{+},
\end{equation*}
where by convention $0/0 \coloneq 0$.
For convenience, for $F \in \mathscr{F}_{+}(\config)$ we also write
\begin{equation*}
\hat{F}(\eta, x) = \theta\paren[\big]{F(\eta), F(\eta + \delta_{x})} = \frac{\diff_{x} F(\eta)}{\diff_{x} \log F(\eta)}, \qquad  \eta \in \config,\, x \in X.
\end{equation*}
For all $\mu \in \mathscr{P}_{1}(\config)$ and $\nu \in \mathscr{M}_{b,0}(\config \times X)$, let us define
\begin{equation*}
  \mathcal{L}(\mu, \nu) = \int \alpha\paren*{\frac{\dd \mu \otimes m}{\dd \sigma}, \frac{\dd C_{\mu}}{\dd \sigma}, \frac{\dd \nu}{\dd \sigma}} \dd \sigma,
\end{equation*}
where $\sigma \in \mathscr{M}_{b,0}(\config \times X)$ is non-negative such that $\mu \otimes \pi$, $C_{\mu}$, and $\nu$ are absolutely continuous with respect to $\sigma$.
By homogeneity, the value of the action is independent of the choice of $\sigma$.
Provided $\mu = \rho \pi$ and $\nu = w (\pi \otimes m)$, in view of \cref{e:campbell-density}, we can choose $\sigma = \pi \otimes m$, and we find that:
\begin{equation*}
  \mathcal{L}(\mu, \nu) = \int \alpha\big(\rho(\eta), \rho(\eta + \delta_{x}),w(\eta, x)\big) \pi(\dd \eta) m(\dd x) = \int \frac{\abs{w}^{2}}{\hat{\rho}} \dd \pi \dd m.
\end{equation*}
We can then interpret $\mathcal{L}(\mu, \nu)$ as the norm of the \enquote{tangent vector} $\nu$ in the \enquote{tangent space} to $\mathscr{P}_1(\Upsilon)$ at~$\mu\in\mathscr{P}_1$.

In view of the convexity of $\alpha$ we immediately get the following lemma.
\begin{lemma}\label{t:convex-lagrangian}
  The Lagrangian $\mathcal{L}$ is jointly convex.
\end{lemma}

\begin{lemma}\label{t:lsc-lagrangian}
  The map $\mathcal{L} \colon \mathscr{P}_{1}(\config) \times \mathscr{M}_{b,0}(\config \times X) \to \real_{+}$ is lower semi-continuous.
\end{lemma}
\begin{proof}
  By \cref{t:p1-polish} and since $\mathscr{M}_{b,0}(\config \times X)$ is metrizable, it is enough to establish sequential lower semi-continuity.
  Let $(\mu_{n}) \subset \mathscr{P}_{1}(\config)$ converging to $\mu \in \mathscr{P}_{1}(\config)$ and $(\nu_{n}) \subset \mathscr{M}_{b,0}(\config \times X)$ converging to $\nu \in \mathscr{M}_{b,0}(\config \times X)$.
  Since $\alpha$ is lower semi-continuous and convex we can write
  \begin{equation*}
    \alpha(p) = \sup \left\{ p \cdot q - \alpha^{*}(q) : q \in \rational^{3} \right\},
  \end{equation*}
  where $\alpha^{*}$ is the Fenchel conjugate of $\alpha$.
  For $p$ and $q \in \rational^{3}$, we let $\alpha_{q}(p) = p \cdot q - \alpha^{*}(q)$.
  Then, by monotone convergence,
  \begin{align}
      \nonumber \mathcal{L}(\mu, \nu) &= \int \sup_{q \in \rational^{3}} \alpha_{q}\paren*{\frac{\dd \mu \otimes m}{\dd \sigma}, \frac{\dd C_{\mu}}{\dd \sigma}, \frac{\dd \nu}{\dd \sigma}} \dd \sigma \\
                            &= \sup_{q \in \rational^{3}} \int \alpha_{q}\paren*{\frac{\dd \mu \otimes m}{\dd \sigma}, \frac{\dd C_{\mu}}{\dd \sigma}, \frac{\dd \nu}{\dd \sigma}} \dd \sigma.\label{e:sup-integrals}
  \end{align}
  By \cref{i:convergence-C} in \cref{t:p1-convergence}, we find that, for $q$ fixed, the integral in the last line of \cref{e:sup-integrals} is continuous on $\mathscr{P}_{1}(\config) \times \mathscr{M}_{b,0}(\config \times X)$.
  As a supremum of continuous functions $\mathcal{L}$ is lower semi-continuous.
\end{proof}

Whenever $\mu$ in absolutely continuous with respect to $\pi$, the following result shows that we can restrict our study to $\nu \in \mathscr{M}_{b,0}(\config \times X)$ that are absolutely continuous with respect to $\pi \otimes m$.
The Lemma below is an adaptation to our setting of \cite[Lemma 2.3]{Erbar}. Since our notation is quite different from this reference, we give a complete proof.

\begin{lemma}\label{t:lagrangian-absolutely-continuous}
  Let $\mu = \rho \pi \in \mathscr{P}(\config)$ and $\nu \in \mathscr{M}_{b,0}(\config \times X)$ such that $\mathcal{L}(\mu, \nu) < \infty$.
  Then, $\nu$ is absolutely continuous with respect to $\pi \otimes m$.
\end{lemma}
\begin{proof}
   Take $A \in \mathfrak{B}(\config)$ and $B \in \mathfrak{B}_{0}(X)$ such that $\pi(A) m(B) = 0$, and $\sigma \in \mathscr{M}_{b,0}(\config \times X)$ non-negative and such that $\pi \otimes m \ll \sigma$ and $\nu \ll \sigma$.
  The homogeneity of $\theta$ yields:
  \begin{equation*}
  0 = \int_{A \times B} \theta(\rho(\eta), \rho(\eta + \delta_{x})) \pi(\dd \eta) m(\dd x) = \int_{A \times B} \theta\paren*{\frac{\dd (\mu \otimes m)}{\dd \sigma}, \frac{\dd C_{\mu}}{\dd \sigma}} \dd \sigma.
\end{equation*}
By positivity of $\theta$ and $\sigma$, the integrand vanishes $\sigma$-almost everywhere on $A \times B$.
By definition of $\mathcal{L}$:
\begin{equation*}
\mathcal{L}(\mu, \nu) = \int \frac{\abs[\Big]{\frac{\dd \nu}{\dd \sigma}}^{2}}{\theta\paren[\Big]{\frac{\dd (\mu \otimes m)}{\dd \sigma}, \frac{\dd C_{\mu}}{\dd \sigma}}} \dd \sigma.
  \end{equation*}
  The above quantity is finite by assumption.
  Since the denominator vanishes on $A \times B$ so does the numerator.
  Thus $\nu(A \times B) = 0$.
\end{proof}

\begin{lemma}\label{t:contraction-lagrangian}
  Let $\mu = \rho \pi \in \mathscr{P}(\config)$ and $\nu \in \mathscr{M}_{b,0}(\config \times X)$.
  Then,
  \begin{equation*}
    \mathcal{L}(\dou_{t}\mu, \dou_{t}\nu) \leq \mathcal{L}(\mu, \nu), \qquad t > 0.
  \end{equation*}
\end{lemma}

\begin{proof}
  We can assume that $\mathcal{L}(\mu, \nu) < \infty$ otherwise there is nothing to prove.
  By \cref{t:lagrangian-absolutely-continuous}, we have that $\nu = w (\pi \otimes m)$.
  By \cref{e:campbell-density}, we find that
  \begin{equation*}
    \mathcal{L}(\dou_{t}\mu, \dou_{t} \nu) = \int \alpha(\mathsf{P}_{t} \rho(\eta), \mathsf{P}_{t} \rho(\eta + \delta_{x}), \mathsf{P}_t w(\eta, x)) \dd (\pi \otimes m).
  \end{equation*}
  We conclude by convexity of $\alpha$, Jensen's inequality, and invariance of $\mathsf{P}$ with respect to $\pi$.
\end{proof}

We finish with a useful bound.
\begin{lemma}\label{t:bound-nu-lagrangian}
  Let $\mu \in \mathscr{P}_{1}(\config)$ and $\nu \in \mathscr{M}_{b,0}(\config \times X)$.
  Then:
  \begin{equation*}
    \abs{\nu}(A \times B) \leq \paren*{ \tfrac{1}{2} \paren*{m(B) + I_{\mu}(B)} \mathcal{L}(\mu, \nu) }^{1/2}, \qquad A \in \mathfrak{B}(\config),\, B \in \mathfrak{B}_{0}(X).
  \end{equation*}
\end{lemma}

\begin{proof}
  Take $\sigma = (\mu \otimes m) + C_{\mu} + \abs{\nu}$ so that we have, $\mu \otimes m = \rho^{1} \sigma$, $C_{\mu} = \rho^{2} \sigma$, and $\nu = w \sigma$.
  We assume that $\mathcal{L}(\mu, \nu)<\infty$, otherwise there is nothing to prove.
  We have that 
  \begin{align*}
    |\nu|(A \times B) 
&= \int_{A \times B} \abs{w} \dd \sigma
\\
&= \int_{A \times B} \sqrt{\theta(\rho^1, \rho^2)} \sqrt{\alpha(w, \rho^1, \rho^2)} \dd \sigma
\\
&\le \paren*{ \int_{A \times B} \theta(\rho^1, \rho^2)  \dd \sigma }^{1/2}  \paren*{ \int_{A \times B} \alpha(\rho^1, \rho^2,w) \dd \sigma }^{1/2}
\\
&= \paren*{ \int_{A \times B} \theta(\rho^1, \rho^2)  \dd \sigma }^{1/2} \sqrt{\mathcal{L}(\mu, \nu)}. 
  \end{align*}
Bounding from above the logarithmic mean with the arithmetic mean, we have
  \begin{align*}
    \int_{A \times B} 2\theta(\rho^1, \rho^2)  \dd \sigma
    \le \int_{\config \times B} (\rho^1+\rho^2) \dd \sigma
    = (\mu \otimes m)(\config \times B) + C_{\mu}(\config \times B) = m(B) + I_{\mu}(B)
    <\infty,
  \end{align*}
which completes the proof.
\end{proof}

\subsubsection{The action functional}
We now define the \emph{action} associated with a curve $\bar{\mu} \in \mathscr{F}([0,1], \mathscr{P}(\config))$.
We set
\begin{equation*}
  \mathcal{A}(\bar{\mu}) \coloneq \inf  \int_{0}^{1} \mathcal{L}(\mu_{t}, \nu_{t}) \dd t, 
\end{equation*}
where the infimum runs over all $\bar{\nu}$ such that $(\bar{\mu}, \bar{\nu})$ is a solution to the continuity equation on $[0,1]$.
Whenever there is no such $\bar{\nu}$, we set $\mathcal{A}(\bar{\mu}) = \infty$.

As a direct application of \cref{t:ou-solution,t:contraction-lagrangian}, we obtain the following contraction property for the action.
\begin{proposition}\label{t:contraction-action}
  For all $\varepsilon > 0$,
  \begin{equation*}
    \mathcal{A}(\dou_{\varepsilon} \bar{\mu}) \le \e^{-2\varepsilon} \mathcal{A}(\bar{\mu}).
  \end{equation*}
\end{proposition}

We now establish that $\mathcal{A}$ is a good functional for minimization problems.

\begin{lemma}\label{t:convex-action}
  The functional $\mathcal{A}$ is convex.
\end{lemma}

\begin{proof}
  Let $\bar{\mu}$ and $\bar{\mu}'$ with finite action.
  Thus, there exists $\bar{\nu}$ and $\bar{\nu}'$ such that $(\bar{\mu}, \bar{\nu})$ and $(\bar{\mu}', \bar{\nu}')$ solve \cref{e:ce}.
  Let $\tau \in [0,1]$.
  Then, with $\bar{\nu}_{\tau} = (1-\tau) \bar{\nu} + \tau \bar{\nu}'$ and $\bar{\mu}_{\tau} = (1-\tau) \bar{\mu} + \tau \bar{\mu}$, we have that $(\bar{\mu}_{\tau}, \bar{\nu}_{\tau}) \sim \cref{e:ce}$.
  Since $\mathcal{L}$ is convex by \cref{t:convex-lagrangian}, we get: 
  \begin{equation*}
    \mathcal{A}(\bar{\mu}_{\tau}) \leq \int_{0}^{1} \mathcal{L}\paren*{(1-\tau) \mu_{t} + \tau \mu'_{t}, (1-\tau) \nu_{t} + \tau \nu_{t}} \dd t \leq (1-\tau) \mathcal{A}(\bar{\mu}) + \tau \mathcal{A}(\bar{\mu}').
    \qedhere
  \end{equation*}
\end{proof}

Fix $\xi \in \mathscr{P}_{1}(\config)$.
We write
\begin{equation*}
\mathscr{C}_{\xi}\paren[\big]{[0,1], \mathscr{P}_{1}(\config)} \coloneq \set[\big]{ \bar{\mu} \in \mathscr{C}([0,1], \mathscr{P}_{1}(\config)) : \mu_{0} = \xi }.
\end{equation*}
As a consequence of \cref{t:p1-polish}, the space~$\mathscr{C}_{\xi}\paren[\big]{[0,1], \mathscr{P}_{1}(\config)}$ is Polish when endowed with the topology of uniform $\mathscr{P}_1(\Upsilon)$-convergence.

\begin{lemma}\label{t:lsc-action}
  The functional $\mathcal{A} \colon \mathscr{C}_{\xi}([0,1], \mathscr{P}_{1}(\config)) \to [0,\infty]$ is lower semi-continuous.
\end{lemma}
\begin{lemma}\label{t:compact-action}
  The action $\mathcal{A}$ has compact sub-level sets in $\mathscr{C}_{\xi}([0,1], \mathscr{P}_{1}(\config))$.
\end{lemma}

\begin{proof}[Proof of {\cref{t:lsc-action,t:compact-action}}]
  Let $r \in (0,\infty)$ and set $\Delta = \set{ \mathcal{A}(\bar{\mu}) \leq r } \cap \set{ \mu_{0} = \xi }$.
  Take $(\bar{\mu}_{n}) \subset \Delta$.
  Since $\mathcal{A}(\bar{\mu}_{n}) \leq r$, for all $n \in \nat$, there exists $\bar{\nu}_{n} \in \mathscr{M}_{b,0}(\config \times \bar{X})$ with $(\bar{\mu}_{n}, \bar{\nu}_{n})$ solving the continuity equation on $[0,1]$ and
  \begin{equation*}
    \mathcal{A}(\bar{\mu}_{n}) \le \int_{0}^{1} \mathcal{L}(\mu_{n,t}, \nu_{n,t}) \dd t \leq r + 1.
  \end{equation*}
  Let $A \in \mathfrak{B}(\config)$, $B \in \mathfrak{B}_{0}(X)$, and $I \in \mathfrak{B}([0,1])$.
In view of \cref{t:bound-nu-lagrangian}, \cref{t:ce-intensity-measure}, and Cauchy--Schwarz inequality for all $n \in \mathbb{N}$:
  \begin{equation}\label{e:bound-nu}
    \begin{split}
      \abs{\bar{\nu}_{n}}(A \times B \times I) & \leq \int_{I} \sqrt{m(B) + I_{\mu_{0}}(B) + \abs{\bar{\nu}_{n}}(A \times B \times [0,t])} \ {\mathcal{L}(\mu_{n,t}, \nu_{n,t})}^{1/2} \ \dd t \\
                                               & \leq \sqrt{m(B) + I_{\mu_{0}}(B) + \abs{\bar{\nu}_{n}}(A \times B \times [0,1])} \ \abs{I} (r+1).
    \end{split}
  \end{equation}
  Setting $I \coloneq [0,1]$ in \cref{e:bound-nu} yields
  \begin{equation*}
    \abs{\bar{\nu}_{n}}(A \times B \times [0,1]) \leq \sqrt{\abs{\bar{\nu}_{n}}(A \times B \times [0,1]) + m(B) + I_{\mu_{0}}(B)} \ (r+1).
  \end{equation*}
  Solving explicitly this equation yields
  \begin{equation}\label{e:bound-nu-total-mass}
    \abs{\bar{\nu}_{n}}(A \times B \times [0,1]) \leq a_{r} + b_{r} \sqrt{m(B) + I_{\mu_{0}}(B)},
  \end{equation}
  with $a_{r}$ and $b_{r} > 0$ depending only on $r$.
  This shows that \cref{e:prohorov-compact-uniform-bound} in \cref{t:prohorov} is satisfied.

  Let us now show that \cref{e:prohorov-compact-tight} is satisfied.
  Let $\varepsilon > 0$.
  By \cref{t:prohorov}, we can find $\Gamma \in \mathfrak{K}(\config)$, $K \in \mathfrak{K}(X)$, and $J \in \mathfrak{K}([0,1])$ such that
  \begin{equation*}
    \abs{[0,1] \setminus J} +  m(B \setminus K) + I_{\mu_{0}}(B \setminus K) \leq \varepsilon.
  \end{equation*}
  Now let $\Delta \coloneq \Gamma \times K \times J \in \mathfrak{K}(\config \times X \times [0,1])$.
  Then by \cref{e:bound-nu,e:bound-nu-total-mass}, we get that
  \begin{equation*}
    \begin{split}
    \abs{\bar{\nu}_{n}}\paren[\big]{(\config \times B \times [0,1]) \setminus \Delta} & \leq \abs{\bar{\nu}_{n}}\paren*{\paren[\big]{\config \setminus \Gamma} \times (B \setminus K) \times ([0,1] \setminus J)} \\
                                                                            & \leq \varepsilon (r+1) \sqrt{\varepsilon + a_{r} + b_{r} \sqrt{\varepsilon}}.
    \end{split}
  \end{equation*}
  Thus, \cref{t:prohorov} \cref{i:prohorov-compact} applies, and, up to passing to a subsequence, we can find $\bar{\nu} \in \mathscr{M}_{b,0}(\config \times \bar{X})$ such that $\bar{\nu}_{n} \to \bar{\nu}$, as $n \to \infty$.
  
    Recall that $\mathfrak{A}(\config)$ is the algebra defined in \cref{s:local-algebra}.
    Define,
    \begin{equation*}
      \mu_{t}(A) \coloneq \mu_{0}(A) + \int_{0}^{t} \nu_{r}(\diff 1_{A}) \dd r, \qquad A \in \mathfrak{A}(\config).
    \end{equation*}
    By \cref{t:ce-one-function}, we find that for all $F \in \mathscr{H}$, $\mu_{t}(F)$ is the limit of $\mu^{n}_{t}(F)$.
    Thus $\mu_{t}$ is a non-negative set function on the algebra $\mathfrak{A}(\config)$ with total mass $1$.
    By Hahn's extension theorem \cite[Thm.\ III.5.8, p.\ 136]{DunfordSchwartz}, it can be uniquely be extended to a probability measure $\mu_{t}$ on $\sigma(\mathfrak{A}(\config)) = \mathfrak{B}(\config)$.
    Moreover, since $\mathscr{G} \subset \mathscr{H}$, we find that $\mu^{n}_{t} \to \mu_{t}$ in $\mathscr{P}(\config)$.
    A similar argument at the level of intensity measures shows that actually $\mu^{n}_{t} \to \mu_{t}$ in $\mathscr{P}_{1}(\config)$.
    By \cref{t:ce-stability-limit}, the find that $(\bar{\mu}, \bar{\nu})$ is a solution to the continuity equation.
    Thus, by lower semi-continuity of $\mathcal{L}$ (\cref{t:lsc-lagrangian}), we find that $\bar{\mu} \in \Delta$.
    This shows that $\Delta$ is compact and this establishes  the two lemmas.
\end{proof}

As a consequence of the properties of $\mathcal{A}$ established above, we obtain the following result.
\begin{theorem}\label{t:action-minizers}
  Let $\bar{\mu} \in \mathscr{C}\paren[\big]{[0,1], \mathscr{P}_{1}(\config)}$ such that $\mathcal{A}(\bar{\mu}) < \infty$, then there exists $\bar{\nu} \in \mathscr{M}_{b,0}(\config \times \bar{X})$ such that $(\bar{\mu}, \bar{\nu})$ solves the continuity equation on $[0,1]$ and
  \begin{equation*}
    \mathcal{A}(\bar{\mu}) = \int_{0}^{1} \mathcal{L}(\mu_{t}, \nu_{t}) \dd t.
  \end{equation*}
\end{theorem}

\subsubsection{The variational distance end the entropic costs}

We now define our distance $\mathcal{W}$.
Actually, we derive our entropic curvature for $\mathcal{W}$ through properties of a regularized version of it.
\begin{definition}
For $\varepsilon \geq 0$, we define the \emph{entropic cost} by
\begin{equation*}
  \mathcal{J}_{\varepsilon}(\xi_{0}, \xi_{1}) \coloneq \Inf*{ \mathcal{A}(\bar{\mu}) + \varepsilon \int_{0}^{1} \Fish{ \mu_{t} \given \pi } \dd t : \mu_{0} = \xi_{0},\, \mu_{1} = \xi_{1} }.
\end{equation*}
We also set $\mathcal{W} \coloneq \mathcal{J}_{0}^{1/2}$.
\end{definition}
We call the quantity $\mathcal{J}_{\varepsilon}$ the entropic cost in analogy with the continuous setting (see \cite{GigliTamanini} and the references therein).
It can be thought of as an entropic regularization of $\mathcal{W}$.
Properties specific to $\mathcal{W}$ are studied below.

\begin{theorem}\label{t:minimizing-curve}
  Let $\varepsilon \geq 0$ and $\xi_{0}$ and $\xi_{1}$ such that $\mathcal{J}_{\varepsilon}(\xi_{0}, \xi_{1}) < \infty$.
  Then, there exists $(\bar{\mu}^{\varepsilon}, \bar{\nu}^{\varepsilon})$ solving the continuity equation such that
  \begin{equation*}
    \mathcal{J}_{\varepsilon}(\xi_{0}, \xi_{1}) = \int_{0}^{1} \mathcal{L}(\mu^{\varepsilon}_{t}, \nu^{\varepsilon}_{t}) \dd t + \varepsilon \int_{0}^{1} \Fish{ \mu^{\varepsilon}_{t} \given \pi } \dd t.
  \end{equation*}
\end{theorem}

\begin{proof}
  Since $\varepsilon$ is fixed, in this proof we drop the dependence on $\varepsilon$ whenever no confusion may arise.
The relative Fisher information is lower semi-continuous, by \cref{t:lsc-I}, and convex, by Jensen's inequality.
  Thus in view of \cref{t:convex-action,t:lsc-action}, we get the lower semi-continuity and convexity of 
  \begin{equation}\label{e:action-epsilon}
    \mathcal{A}^{\varepsilon}(\bar{\mu}) \coloneq \mathcal{A}(\bar{\mu}) + \varepsilon \int_{0}^{1} \Fish{ \mu_{t} \given \pi } \dd t.
  \end{equation}
  Thus the set $A_{\varepsilon} \coloneq \set{ \mathcal{A}^{\varepsilon} \leq r }$ is closed for all $r \in (0,\infty)$.
  Clearly, we have that $A_{\varepsilon} \subset \set{\mathcal{A} \leq r}$.
  Thus, $A_{\varepsilon}$ is relatively compact by \cref{t:compact-action}.
  The result follows from standard optimization arguments.
\end{proof}

\begin{theorem}\label{t:gamma-convergence}
  Let $\xi_{0}$ and $\xi_{1}$ such that $\mathcal{J}_{\varepsilon_{o}}(\xi_{0}, \xi_{1}) < \infty$ for some $\varepsilon_{o} > 0$.
  For $\varepsilon \in (0,\varepsilon_{o})$, write $(\bar{\mu}^{\varepsilon}, \bar{\nu}^{\varepsilon})$ for a minimizer of $\mathcal{J}_{\varepsilon}(\xi_{0}, \xi_{1})$.
  Then, we have that
  \begin{equation*}
    \mathcal{J}_{\varepsilon}(\xi_{0}, \xi_{1}) \xrightarrow[\varepsilon \to 0^{+}]{} \mathcal{W}^{2}(\xi_{0}, \xi_{1}).
  \end{equation*}
  Moreover, up to passing to a subsequence
  \begin{equation*}
    (\bar{\mu}^{\varepsilon}, \bar{\nu}^{\varepsilon}) \xrightarrow[\varepsilon \to 0^{+}]{} (\bar{\mu}, \bar{\nu}),
  \end{equation*}
  for a minimizer $(\bar{\mu}, \bar{\nu})$ for $\mathcal{W}(\xi_{0}, \xi_{1})$.
\end{theorem}

\begin{proof}
  Let us write $r \coloneq \mathcal{J}_{\varepsilon_{o}}(\xi_{0}, \xi_{1}) + 1 < \infty$, and $A \coloneq \set{ \mathcal{A}^{\varepsilon_{o}} \leq r }$.
  Since the family $( \mathcal{A}^{\varepsilon} )$ is decreasing in $\varepsilon$ when regarded as functionals on $A$, we have
  \begin{equation*}
    \mathcal{J}_{\varepsilon}(\xi_{0}, \xi_{1}) = \Inf*{ \mathcal{A}^{\varepsilon}(\bar{\mu}) : \mu_{0} = \xi_{0},\, \mu_{1} = \xi_{1},\, \bar{\mu} \in A }.
  \end{equation*}
   On $A$, we have that $\mathcal{A}^{\varepsilon} \searrow \mathcal{A}$ pointwise, and that $\mathcal{A}$ is lower semi-continuous.
  Thus, by \cite[Prop. 5.7]{DalMaso}, $\mathcal{A}^{\varepsilon}$ $\Gamma$-converges to $\mathcal{A}$ on $A$.
  Now, since $\mathcal{A}^{\varepsilon} \geq \mathcal{A}$ and since $\mathcal{A}$ has compact-sublevel sets, the first part of the claim follows from \cite[Prop.\ 7.7 \& Thm.\ 7.8]{DalMaso}.
  The second part of the claim follows from \cite[Cor.\ 7.20]{DalMaso} provided we can show that $\set{ (\bar{\mu}^{\varepsilon}, \bar{\nu}^{\varepsilon}) : \varepsilon \in (0, \varepsilon_{o}) }$ is compact.
  We argue as in \cref{t:compact-action}.
  Indeed, by construction $\bar{\mu}^{\varepsilon} \in \set{ \mathcal{A}^{\varepsilon} \leq r } \subset \set{ \mathcal{A} \leq r }$.
  Thus, \cref{e:bound-nu} holds with $\varepsilon$ in place of of $n$ and the rest of the argument is the same.
\end{proof}

We now study the properties of $\mathcal{W}$.
We start with a classical argument.
\begin{lemma}\label{t:wcurl-time-reparametrization}
  For all $T > 0$, and $\xi_{0}$ and $\xi_{1} \in \mathscr{P}_{1}(\config)$:
  \begin{equation*}
    \mathcal{W}(\xi_{0}, \xi_{1}) = \Inf*{ \int_{0}^{T} \mathcal{L}^{\frac{1}{2}}(\mu_{t}, \nu_{t}) \dd t  : (\bar{\mu}, \bar{\nu}) \sim \cref{e:ce},\, \mu_{0} = \xi_{0},\, \mu_{T} = \xi_{1} }.
  \end{equation*}
\end{lemma}
\begin{proof}
  Follows from a standard reparametrization argument, for instance \cite[Thm.\ 5.4]{DNS09} with \cref{t:ce-time-reparametrization}.
\end{proof}

We now summarize the main property of $\mathcal{W}$.
\begin{theorem}\label{t:W}
  \begin{enumerate}[$(i)$, wide]
    \item\label{i:W-extended-distance} The map $\mathcal{W}$ defines an extended distance on $\mathscr{P}_{1}(\config)$.
    \item\label{i:W-topology} The topology induced by $\mathcal{W}$ on $\mathscr{P}_{1}(\config)$ is stronger than that of $\mathscr{P}_{1}(\config)$.
    \item\label{i:W-lsc} The map $\mathcal{W}$ is lower semi-continuous on $\mathscr{P}_{1}(\config) \times \mathscr{P}_{1}(\config)$.
    \item\label{i:W-bounded} Bounded sets with respect to $\mathcal{W}$ are $\mathscr{P}_{1}(\config)$-relatively compact.
    \item\label{i:W-complete-geodesic} For every $\eta \in \mathscr{P}_{1}(\config)$ the accessible component $\set*{ \mathcal{W}(\eta, \cdot) < \infty }$ is a complete geodesic space when equipped with $\mathcal{W}$.
  \end{enumerate}
\end{theorem}

\begin{proof}
  \cref{i:W-extended-distance} The symmetry is immediate.
  We obtain the triangle inequality by concatenation and using \cref{t:wcurl-time-reparametrization}.
  Now take $\xi_{0}$ and $\xi_{1} \in \mathscr{P}_{1}(\config)$ with $\mathcal{W}(\xi_{0}, \xi_{1}) = 0$.
  By \cref{t:minimizing-curve}, take $\bar{\mu}$ realizing $\mathcal{W}(\xi_{0}, \xi_{1})$.
  Then $\mathcal{A}(\bar{\mu}) = 0$, thus $\bar{\nu} =0$ and $\xi_{0} = \xi_{1}$.
  This shows that $\mathcal{W}$ is an extended distance.

  \cref{i:W-topology}
    Let $(\xi_{n}) \subset \mathscr{P}_{1}(\config)$ and $\xi \in \mathscr{P}_{1}(\config)$ be such that $\mathcal{W}(\xi_{n}, \xi) \to 0$.
    For all $n \in \mathbb{N}$, take $(\bar{\mu}_{n}, \bar{\nu}_{n})$ realizing the infimum in $\mathcal{W}(\xi_{n}, \xi)$.
    Let $(h_{k}) \subset \mathscr{C}_{0}(X)$ be as in \cref{t:weak-convergence-countable}.
    For all $k \in \mathbb{N}$, set $G_{k} \coloneq \e^{-\iota_{h_{k}}}$, $B_{k} \in \mathfrak{B}_{0}(X)$ such that $\diff G_{k} = 0$ outside of $B_{k}$.
    Arguing as in the proof of \cref{t:ce-continuous-representative}, and then using \cref{t:bound-nu-lagrangian}, we find that
    \begin{equation*}
      \begin{split}
        \abs*{ \int G_{k} \dd (\xi_{n} - \xi) } & \leq \abs{\bar{\nu}_{n}}(\config \times B_{k} \times [0,1]) \\
                            & \leq \ \int_{0}^{1} \paren*{m(B_{k}) + I_{\mu_{n,t}}(B_{k})}^{\frac{1}{2}} \mathcal{L}\paren{\mu_{n,t}, \nu_{n,t}}^{\frac{1}{2}} \dd t \\
                            & \leq C_{k} \mathcal{W}(\xi_{n}, \xi).
      \end{split}
    \end{equation*}
    Thus, by \cref{t:weak-convergence-countable}, we find that $\xi_{n} \to \xi$ with respect to the $\mathscr{P}(\config)$-topology.
    Take $h \in \mathscr{C}_{0}(X)$.
    By \cref{t:ce-intensity-measure,t:bound-nu-lagrangian}, we find that
    \begin{equation*}
      \begin{split}
        \abs*{ I_{\xi_{n}}(h) - I_{\xi}(h) } & \leq \abs{ \bar{\nu}_{n} }(1 \otimes h \otimes 1_{[0,1]}) \\
                                             & \leq \int_{0}^{1} \paren*{m(h) + I_{\mu_{n,t}}(h)}^{\frac{1}{2}} \mathcal{L}\paren{\mu_{n,t}, \nu_{n,t}}^{\frac{1}{2}} \dd t \\
                                             & \leq C \mathcal{W}(\xi_{n}, \xi)
      \end{split}
    \end{equation*}
    for some constant~$C>0$ depending on~$h$.
  This shows that $I_{\xi_{n}} \to I_{\xi}$ in $\mathscr{M}_{0}(X)$.
  By \cref{t:p1-convergence}, we find that $\xi_{n} \to \xi$ in $\mathscr{P}_{1}(\config)$.
  
\cref{i:W-lsc}
Fix $r \geq 0$, we want to show closedness of the set
\begin{equation*}
  A \coloneq \set*{ (\xi, \chi) \in \mathscr{P}_{1}(\config) \times \mathscr{P}_{1}(\config) : \mathcal{W}(\xi, \chi) \leq r }.
\end{equation*}
Let $(\xi_{n})$ and $(\chi_{n}) \subset A$ converging respectively to $\xi$ and $\chi \in \mathscr{P}_{1}(\config)$.
By \cref{t:minimizing-curve}, for all $n \in \mathbb{N}$, there exists a solution to the continuity equation $(\bar{\mu}_{n}, \bar{\nu}_{n})$ realizing $\mathcal{W}(\xi_{n}, \chi_{n})$.
Since $\xi_{n} \to \xi$ and $\chi_{n} \to \chi$ arguing as in the proof of \cref{t:compact-action}, we can find $(\bar{\mu}, \bar{\nu})$ solving the continuity equation and joining $\xi$ to $\chi$.
Thus, by \cref{t:lsc-action}, we find that 
\begin{equation*}
  \mathcal{W}(\chi, \xi) \leq \mathcal{A}(\bar{\mu}) \leq \liminf_{n \to \infty} \mathcal{A}(\bar{\mu}_{n}) = \liminf_{n \to \infty} \mathcal{W}(\chi_{n}, \xi_{n}).
\end{equation*}

\cref{i:W-bounded}
Follows from \cref{t:compact-action}.

\cref{i:W-complete-geodesic}
The geodesic property follows from \cref{t:minimizing-curve}, the geodesic being given by the minimizing curve $\bar{\mu}$.
The completeness follows from \cref{i:W-bounded,i:W-lsc}.
\end{proof}

The quantity $\mathcal{J}_{\varepsilon}^{\frac{1}{2}}$ is not a distance for $\varepsilon > 0$ (the reparametrization argument given in \cref{t:wcurl-time-reparametrization} does not work here).
However, we have the following quasi-triangle inequality.

\begin{proposition}\label{t:entropic-triangle-inequality}
  Let $\xi_{0}, \xi_{1}, \xi_{2} \in \mathscr{P}_{1}(\config)$ and $\varepsilon > 0$.
  Then,
  \begin{equation*}
    \mathcal{J}_{\varepsilon}(\xi_{0}, \xi_{2}) \leq 2 \mathcal{J}_{\varepsilon}(\xi_{0}, \xi_{1}) + 2 \mathcal{J}_{\varepsilon}(\xi_{1}, \xi_{2}).
  \end{equation*}
\end{proposition}

\begin{proof}
  We assume that the right hand side is finite.
  Let $(\bar{\mu}^{1}, \bar{\nu}^{1})$ and $(\bar{\mu}^{2}, \bar{\nu}^{2})$ realizing the two infima.
  By concatenation, using \cref{t:ce-time-reparametrization} and that the Lagrangian is quadratic in $\nu$, we find that
  \begin{equation*}
    \mathcal{J}_{\varepsilon} \leq 4 \int_{0}^{1/2} \mathcal{L}(\mu^{1}_{2t}, \nu^{1}_{2t}) \dd t + \varepsilon \int_{0}^{1/2} \Fish{ \mu_{2t} \given \pi } \dd t + 4 \int_{1/2}^{1} \mathcal{L}(\mu^{2}_{2t-1}, \nu^{2}_{2t-1}) \dd t + \varepsilon \int_{1/2}^{1} \Fish{ \mu^{2}_{2t-1} \given \pi } \dd t.
  \end{equation*}
  This gives the claim by an immediate change of variable and since $\varepsilon / 4 \leq \varepsilon$.
\end{proof}

  \subsection{The geometry of \texorpdfstring{$(\dom \ent, \mathcal{W})$}{(dom H, W)}}\label{s:gradient-flow}

  \subsubsection{The metric space \texorpdfstring{$(\dom \ent, \mathcal{W})$}{(dom H, W)}}

We first show that \cref{t:W} is non-trivial by showing that $\dom \ent$ yields an example of an accessible component for $\mathcal{W}$.
The central tool is the following \emph{Talagrand inequality}.
\begin{theorem}\label{t:talagrand}
  For all $\mu \in \mathscr{P}_{1}(\config)$,
  \begin{equation}\label{e:talagrand}
    \mathcal{W}^{2}(\mu, \pi) \leq \Ent{ \mu \given \pi}.
  \end{equation}
  Moreover, for all $\mu \in \dom \ent$ and all $\varepsilon \geq 0$,
  \begin{equation*}
    \mathcal{J}_{\varepsilon}(\mu, \pi) < \infty.
  \end{equation*}
\end{theorem}

\begin{remark}
  Classically, the Talagrand inequality is a consequence of the convexity of the entropy (\cref{t:entropy-geodesically-convex}).
Since $\mathcal{W}$ can be infinite, we derive the Talagrand inequality \emph{a priori} by other means.
\end{remark}

\begin{proof}
  We show \cref{e:talagrand} first.
  We can assume that $\mu \in \dom \ent$ otherwise the claim is empty.
  Let $T > 0$.
  By \cref{t:ou-solution,e:lsi,e:production-entropy-ou}, we find that
  \begin{equation*}
    \begin{split}
      \mathcal{W}(\mu, \mathsf{P}_{T}^{\star} \mu) &\leq \int_{0}^{T} \Fish{ \mathsf{P}_{t}^{\star} \mu \given \pi }^{1/2} \dd t \\
                                                     & \leq \int_{0}^{T} \frac{\Fish{ \mathsf{P}^{\star}_{t} \mu \given \pi }}{\Ent{\mathsf{P}^{\star}_{t} \mu \given \pi }^{1/2}} \dd t \\
                                                     & = - \int_{0}^{T} \frac{\dd}{\dd t} \Ent{ \mathsf{P}_{t}^{\star} \mu \given \pi }^{1/2} \dd t \\
                                                     &= \Ent{ \mu \given \pi }^{1/2} - \Ent{ \mathsf{P}_{T}^{\star} \mu \given \pi }^{1/2}.
    \end{split}
  \end{equation*}
  We conclude by letting $T \to \infty$, and by lower semi-continuity of $\mathcal{W}$ (\cref{t:W} \cref{i:W-lsc}).
 
 Now let us prove the second part of the claim.
 On the one hand, since $\mu \in \dom \ent$, by \cref{e:regularizing-ou,t:ou-solution}, we see that $\mathscr{J}_{\varepsilon}(\mu, \mathsf{P}_{\delta}^{\star} \mu) < \infty$ for all $\delta > 0$.
 In view of \cref{t:entropic-triangle-inequality}, it thus sufficient to show that $\mathscr{J}_{\varepsilon}(\mathsf{P}_{\delta}^{\star}, \pi) < \infty$.
 Since $\mathcal{W}(\mu, \pi) < \infty$ by the first part, we can consider a solution $(\bar{\mu}, \bar{\nu})$ to the continuity equation minimal for $\mathcal{W}(\mu, \pi)$.
 Applying \cref{t:stability-ou-solution} to this solution, and using that $\mathsf{P}_{\delta}^{\star} \pi = \pi$ yields that $(\mathsf{P}_{\delta}^{\star}\bar{\mu}, \e^{-\delta}\mathsf{P}_{\delta}^{\star} \bar{\nu})$ is an admissible candidate for the minimization problem of $\mathscr{J}_{\varepsilon}(\mathsf{P}_{\delta}^{\star}\mu, \pi)$.
 Furthermore, by \cref{e:regularizing-ou}, we find that it has finite $\varepsilon$-energy.
 The proof is complete.
\end{proof}

The following definition is thus very natural.
\begin{definition}
  We write $\mathscr{P}_{1}^{*}(\config)$ for the $\mathcal{W}$-closure of $\dom \ent$.
\end{definition}

The following is a consequence of \cref{t:W,t:talagrand}.
\begin{corollary}\label{t:p1-complete-geodesic}
  The space $(\mathscr{P}_{1}^{*}(\config), \mathcal{W})$ is a complete geodesic space.
\end{corollary}

\begin{remark}
  We have
  \begin{equation*}
    \dom \ent \subset \mathscr{P}_{1}^{*}(\config) \subset \mathscr{P}_{1}(\config).
  \end{equation*}
  A priori each inclusion could be strict.
\end{remark}

\begin{proposition}\label{t:wcurl-derivative-fisher}
  Fix $\mu$ and $\xi \in \mathscr{P}^{*}_{1}(\config)$ then
  \begin{equation*}
    \frac{\dd^{+}}{\dd t} \mathcal{W}(\dou_{t}\mu, \xi) \leq \sqrt{\Fish{ \dou_{t} \mu \given \pi }}, \qquad t > 0.
  \end{equation*}
\end{proposition}
\begin{proof}
  Assume that $\mu = \rho \pi \in \dom \fish$, otherwise there is nothing to prove.
  Write, for $t > 0$, $\mu_{t} \coloneq \dou_{t} \mu$ and $\nu_{t} \coloneq \diff \rho \dd(\pi \otimes m)$.
  By \cref{t:ou-solution}, $(\bar{\mu}, \bar{\nu})$ is a solution to the continuity equation, and
  \begin{equation*}
    \mathcal{L}(\mu_{t}, \nu_{t}) = \Fish{ \mu_{t} \given \pi }.
  \end{equation*}
  Thus by \cref{t:wcurl-absolutely-continuous}, we get:
  \begin{equation*}
    \mathcal{W}(\mu_{t+s}, \xi) - \mathcal{W}(\mu_{t}, \xi) \leq \mathcal{W}(\mu_{t+s}, \mu_{t}) \leq \int_{t}^{t+s} \abs{\dot{\mu}_{u}} \dd u \leq \int_{t}^{t+s} \sqrt{\Fish{ \mu_{u} \given \pi }} \dd u.
  \end{equation*}
  The claim immediately follows.
\end{proof}

Recall that a curve $\bar{\mu} \in \mathscr{F}([0,T], \mathscr{P}_{1}(\config))$ is \emph{absolutely continuous with respect to $\mathcal{W}$} provided there exists $g \in L^{1}(0,T)$ such that:
\begin{equation*}
  \mathcal{W}(\mu_{s}, \mu_{t}) \leq \int_{s}^{t} g(r) \dd r, \qquad 0 \leq s \leq t \leq T.
\end{equation*}
By definition the \emph{metric derivative} of $\bar{\mu}$ is the minimal $g$ in the above inequality denoted by $t \mapsto \abs{\dot{\mu}_{t}}$.
Recall from \cite[Thm.\ 1.1.2]{AGS}, that, for almost every $t \in (0,T)$,
\begin{equation*}
  \abs{\dot{\mu}_{t}} = \lim_{\varepsilon \to 0} \frac{\mathcal{W}(\mu_{t + \varepsilon}, \mu_{t})}{\varepsilon}.
\end{equation*}

\begin{proposition}\label{t:wcurl-absolutely-continuous}
The curve $\bar{\mu} \in \mathscr{C}\paren[\big]{[0,T], \mathscr{P}^{*}_{1}(\config)}$ is absolutely continuous with respect to $\mathcal{W}$ if and only if there exists $\bar{\nu} \in \mathscr{M}_{b,0}(\config \times \bar{X})$ such that $(\bar{\mu}, \bar{\nu}) \sim \cref{e:ce}$ and
  \begin{equation*}
    \int_{0}^{T} \sqrt{\mathcal{L}(\mu_{t}, \nu_{t})} \dd t < \infty.
  \end{equation*}
  In this case, $\abs{\dot{\mu}_{t}}^{2} \leq \mathcal{L}(\mu_{t}, \nu_{t})$ for almost every $t \in [0,T]$.
  Moreover, there exists a unique $\bar{\nu}' \in \mathscr{M}_{b,0}(\config \times \bar{X})$ such that $(\bar{\mu}, \bar{\nu}') \sim \cref{e:ce}$ and
  \begin{equation}\label{e:unique-tangent}
    \abs{\dot{\mu}_{t}}^{2} = \mathcal{L}(\mu_{t}, \nu'_{t}), \qquad \text{for a.e.}\ t \in [0,T].
  \end{equation}
\end{proposition}
\begin{proof}
 See \cite[Thm.\ 5.17]{DNS09}: the precompactness result in \cite[Cor.\ 4.10]{DNS09} corresponds to \cref{t:compact-action,t:lsc-action}.
\end{proof}
In the previous section, we have informally chosen $\mathscr{M}_{b,0}(\config \times X)$ to be the tangent space of $\mathscr{P}_{1}(\config)$.
However, it would be natural to consider only vector fields that have minimal Lagrangian.
In order to do so, observe that if $(\bar{\mu}, \bar{\nu})$ and $(\bar{\mu}, \bar{\nu}')$ solve the continuity equation, then for all $t \in [0,1]$, $\nu_{t} - \nu'_{t}$ is \emph{divergence-free}, in the sense that
\begin{equation*}
  (\nu_{t} - \nu'_{t})(\diff F) = 0, \qquad F \in \mathscr{H}.
\end{equation*}
This leads to the following definition of the tangent space, for $\mu \in \mathscr{P}_{1}^{*}(\config)$,
\begin{equation*}
T_{\mu} \mathscr{P}^{*}_{1}(\config) \coloneq \set[\big]{ \nu \in \mathscr{M}_{b,0}(\config \times X) : \mathcal{L}(\mu, \nu) \leq \mathcal{L}(\mu, \nu + \nu') < \infty,\quad \nu' \ \text{divergence-free} }.
\end{equation*}
From \cref{t:wcurl-absolutely-continuous} and this definition, we get the following result.
\begin{corollary}\label{t:nu-tangent}
  Take $(\bar{\mu}, \bar{\nu})$ a solution to the continuity equation such that $\bar{\mu}$ is absolutely continuous with respect to $\mathcal{W}$, and $\mu_{t} \in \mathscr{P}_{1}^{*}(\config)$, for all $t \in [0,1]$.
  Then, $\bar{\nu}$ is the unique solution to \cref{e:unique-tangent} if and only if $\nu_{t} \in T_{\mu_{t}} \mathscr{P}_{1}(\config)$.
\end{corollary}

  As in the Euclidean case \cite[Section 8.1]{AGS}, we obtain an explicit representation of the tangent as a closure of gradient fields.
  \begin{proposition}
    Assume that $\mu = \rho \pi \in \mathscr{P}^{*}_{1}(\config)$.
    Then, $T_{\mu} \mathscr{P}^{*}_{1}(\config)$ is the set of measures $\nu = w (\pi \otimes m)$ such that $w$ is in the $L^{2}(\pi \otimes m)$-closure of $\set{ \diff F : F \in \mathscr{H} }$.
  \end{proposition}

  \begin{proof}
    In view of \cref{t:lagrangian-absolutely-continuous}, the claim follows by observing that $\nu' = w' (\pi \otimes m)$ is divergence-free if and only if $\int \diff F w' \dd \pi \dd m = 0$ for all $F \in \mathscr{H}$ and that the space of such densities is the orthogonal space to the space of gradient fields.
  \end{proof}

  \subsubsection{Evolution variation inequality and entropic curvature bounds}

  We now establish the main results of the paper, namely we show that of the Ornstein--Uhlenbeck semi-group is the gradient flow of $\Ent{\cdot \given \pi}$ on $(\mathscr{P}_{1}^{*}(\config), \mathcal{W})$.
Despite $\mathcal{W}$ being an extended distance on $\mathscr{P}_{1}(\config)$, the space $(\mathscr{P}_{1}^{*}(\config), \mathcal{W})$ is a metric space in the usual sense (that is, not extended).

Following \cref{t:contraction-action,t:talagrand}, the following contraction estimates hold.
\begin{theorem}\label{t:dou-contractivity}
  For $\mu_{0}$ and $\mu_{1} \in \mathscr{P}_{1}(\config)$, and $t \geq 0$:
  \begin{align}
    & \mathcal{W}(\dou_{t} \mu_{0}, \dou_{t} \mu_{1}) \leq \e^{-t} \mathcal{W}(\mu_{0}, \mu_{1}); \\
    & \mathcal{W}(\dou_{t} \mu_{0}, \pi) \leq \e^{-t} \Ent{ \mu_{0} \given \pi }.
  \end{align}
\end{theorem}
We now establish a much stronger relationship between $\mathcal{W}$ and $\ent$ by showing that $\dou$ is the gradient flow of the entropy with respect to $\mathcal{W}$.
\begin{theorem}\label{t:evi}
  The space $\dom \ent$ is geodesically convex with respect to $\mathcal{W}$.
  Furthermore, the following Evolution Variation Inequality holds: for all $\mu$ and $\xi \in \dom \ent$,
  \begin{equation}\label{e:evi}\tag{EVI}
    \Ent{\dou_{s} \mu \given \pi} + \frac{1}{2} \frac{\dd}{\dd s} \mathcal{W}^{2}(\dou_{s} \mu, \xi) + \frac{1}{2} \mathcal{W}^{2}(\dou_{s} \mu, \xi) \leq \Ent{\xi \given \pi}, \qquad s \geq 0.
  \end{equation}
\end{theorem}

\begin{proof}
  By the semigroup property of $\mathsf{P}^{\star}$ it suffices to show the claim at $s=0$.
  Our strategy consists in starting from a minimizing curve $(\bar{\mu}, \bar{\nu})$ for $\mathcal{W}(\mu, \xi)$ and $\delta > 0$ construct a deformation $(\bar{\mu}^{\delta}, \bar{\nu}^{\delta})$ that is admissible for $\mathcal{W}(\dou_{\delta}\mu, \xi)$ and then use estimates from the previous section in order to control $\mathcal{W}(\dou_{\delta} \mu, \xi) - \mathcal{W}(\mu, \xi)$.
  However, since the Ornstein--Uhlenbeck semi-group is only regularizing from $\dom \ent$ to $\dom \fish$, and that we have a priori no information on the regularity of geodesics, we implement this strategy in two steps.
  First, we use the entropic cost $\mathcal{J}_{\varepsilon}$ for which we know that minimizing curves are in the domain of the Fisher information, in order to derive a weaker version of \cref{e:evi} for $\mathcal{J}_{\varepsilon}$, and for $\mathcal{W}$ passing to the limit.
  Second, we can use this weak \cref{e:evi} in order to deduce that $\dom \ent$ is geodesically convex, thus gaining some regularity of geodesics.
  This regularity is sufficient in order to reimplement the above strategy but directly at the level of $\mathcal{W}$ rather than $\mathcal{J}_{\varepsilon}$.
  Since $\mathcal{W}$ has more structure than $\mathcal{J}_{\varepsilon}$ we can deduce \cref{e:evi}.

  \paragraph{Approximation of minimizers via the Ornstein--Uhlenbeck semi-group}
  Let $\varepsilon > 0$.
  By \cref{t:talagrand}, we get that $\mathcal{J}_{\varepsilon}(\mu, \xi) < \infty$.
  By \cref{t:minimizing-curve}, we can consider $(\bar{\mu}^{\varepsilon}, \bar{\nu}^{\varepsilon})$ solving the continuity equation and realizing $\mathcal{J}_{\varepsilon}(\mu, \xi)$.
  By the finiteness of $\Fish{ \mu_{t}^{\varepsilon} \given \pi }$ for almost every $t \in [0,1]$ and \cref{t:ce-continuous-representative}, we can write, for all $t \in [0,1]$, $\mu_{t}^{\varepsilon} = \rho_{t}^{\varepsilon} \pi$ for some probability density.
  By \cref{t:lagrangian-absolutely-continuous}, we can take $\nu_{t}^{\varepsilon} = w_{t}^{\varepsilon} (\pi \otimes m)$.
  Recall that by \cref{t:ou-solution}, we can use the Ornstein--Uhlenbeck to construct solutions to the continuity equation from a fixed initial measure.
  Here we use a similar strategy with an additional correction taking into account that $\bar{\mu}$ also depends on $t$.
  Namely, for all $\delta > 0$, we define
  \begin{align*}
      & \mu_{t}^{\varepsilon,\delta} = \mathsf{P}^{\star}_{t\delta} \mu_{t}^{\varepsilon} = \rho^{\varepsilon, \delta}_{t} \mu, \\
      & \nu_{t}^{\varepsilon,\delta} = \e^{-t\delta} \mathsf{P}_{t\delta}^{\star} \nu_{t}^{\varepsilon} - \delta \diff \mathsf{P}_{t\delta} \rho_{t}^{\varepsilon} (\pi \otimes m) = w^{\varepsilon,\delta}_{t} (\pi \otimes m).
  \end{align*}

  By construction, we have $\mu^{\varepsilon,\delta}_{0} = \mu^{\varepsilon}_{0} = \xi$ and $\mu^{\varepsilon,\delta}_{1} = \dou_{\delta} \mu$.
  Let us show that $(\bar{\mu}^{\varepsilon,\delta}, \bar{\nu}^{\varepsilon,\delta})$ solves the continuity equation.
  Indeed, let $F \in \mathscr{C}_{c}^{1}([0,1], \smooth)$.
  By definition of $\mathsf{L}$, we have
  \begin{equation}\label{e:time-derivative-Pt-Ft}
    \partial_{t} \ou_{t \delta} F_{t} = \delta \gen \ou_{t \delta} F_{t} + \ou_{t \delta} \partial_{t} F_{t}.
  \end{equation}
  By definition of $\mathsf{P}^{\star}$ and \cref{e:time-derivative-Pt-Ft},
  \begin{equation}\label{e:continuity-equation-delta-epsilon-1}
    \begin{split}
      \int_{0}^{1} \int \partial_{t} F_{t} \dd \dou_{t\delta} \mu^{\varepsilon}_{t} \dd t &= \int_{0}^{1} \int \mathsf{P}_{\delta t} \partial_{t} F_{t} \dd \mu_{t}^{\varepsilon} \dd t
        \\&= \int_{0}^{1} \int (\partial_{t} \ou_{t\delta}F_{t} - \delta \gen \ou_{t\delta}F_{t}) \dd \mu^{\varepsilon}_{t} \dd t.
    \end{split}
  \end{equation}
  On the one hand, since, by definition, $(\bar{\mu}^{\varepsilon}, \bar{\nu}^{\varepsilon})$ solves the continuity equation, and since $\diff \ou_{t \delta} = \e^{-t \delta} \ou_{t \delta} \diff$, we have that:
  \begin{equation}\label{e:continuity-equation-delta-epsilon-2}
    \int_{0}^{1} \int \partial_{t} \mathsf{P}_{t\delta}F_{t} \mathrm{d} \mu_{t}^{\varepsilon} \mathrm{d} t = - \int_{0}^{1} \int \diff \ou_{t \delta} F_{t} \dd \nu^{\varepsilon}_{t} \dd t = - \int_{0}^{1} \int \mathrm{e}^{-t\delta}F_{t} \mathrm{d} \nu^{\varepsilon}_{t} \mathrm{d} t.
  \end{equation}
  On the other hand, since $\mathsf{L}$ and $\mathsf{P}$ commute, and by integration by part between $\mathsf{L}$ and $\mathsf{D}$ provided by the Mecke formula
  \begin{equation}\label{e:continuity-equation-delta-epsilon-3}
    - \delta \int_{0}^{1} \int \mathsf{L} \mathsf{P}_{t\delta} F_{t} \mathrm{d} \mu_{t}^{\varepsilon} \mathrm{d} t =  \delta \int_{0}^{1} \int \diff  F_{t} \diff \ou_{t \delta} \rho^{\varepsilon}_{t} \dd (\pi \otimes m) \dd t.
  \end{equation}
   combining \cref{e:continuity-equation-delta-epsilon-1,e:continuity-equation-delta-epsilon-2,e:continuity-equation-delta-epsilon-3}, we find that
  \begin{equation*}
    \int_{0}^{1} \int \partial_{t} F_{t} \dd \mu^{\varepsilon,\delta}_{t} \dd t = - \int_{0}^{1} \int \diff F_{t} \dd \nu^{\varepsilon, \delta}_{t} \dd t.
  \end{equation*}
  That is to say that $(\bar{\mu}^{\varepsilon,\delta}, \bar{\nu}^{\varepsilon,\delta})$ solves the continuity equation.
  \paragraph{Expansion of the Lagrangian along the approximation}
  By the Cauchy--Schwarz inequality,
  \begin{equation*}
    \begin{split}
      \int_{0}^{1} \int \abs{\diff \log \rho^{\varepsilon,\delta}_{t} w^{\varepsilon,\delta}_{t}} \dd (\pi \otimes m) \dd t &= \int_{0}^{1} \int \abs*{ \mathsf{D} \log \rho^{\varepsilon,\delta}_{t} \mathsf{D} \rho^{\varepsilon,\delta}_{t}}^{1/2} \abs*{\frac{\mathsf{D} \log \rho^{\varepsilon,\delta}_{t}}{\mathsf{D} \rho^{\varepsilon,\delta}_{t}} w^{\varepsilon,\delta}_{t}}^{1/2} \mathrm{d} (\pi \otimes m) \mathrm{d} t
                                                                                                                          \\& \leq \paren*{\int_{0}^{1} \Fish{ \mu^{\varepsilon,\delta}_{t} \given \pi} \dd t \int_{0}^{1} \mathcal{L}(\mu^{\varepsilon,\delta}_{t}, \nu^{\varepsilon,\delta}_{t}) \dd t}^{\frac{1}{2}}.
    \end{split}
  \end{equation*}
  Using that $\paren{a+b}^{2} \le 2a^{2} + 2b^{2}$, we get
  \begin{equation*}
    \begin{split}
    \int_{0}^{1} \mathcal{L}(\mu^{\varepsilon,\delta}_{t}, \nu^{\varepsilon,\delta}_{t}) \mathrm{d} t \leq \int_{0}^{1} \e^{-2\delta t} \mathcal{L}(\dou_{t \delta} \mu^{\varepsilon}_{t}, \dou_{t \delta} \nu^{\varepsilon}_{t}) \dd t + \delta^{2} \int_{0}^{1} \int \frac{\abs{\diff \ou_{t \delta} \rho^{\varepsilon}_{t}}^{2}}{\theta(\ou_{t \delta} \rho^{\varepsilon}_{t} + \diff \ou_{t \delta} \rho^{\varepsilon}_{t}, \ou_{t \delta} \rho^{\varepsilon}_{t})} \dd (\pi \otimes m) \dd t.
    \end{split}
  \end{equation*}
  By \cref{t:contraction-lagrangian}, the first term is not larger than $\mathcal{A}(\bar{\mu}^{\varepsilon})$ which is finite by construction.
  The second term is, by definition, $\delta^{2} \int_{0}^{1} \Fish{ \dou_{t\delta} \mu_{t}^{\varepsilon} \given \pi} \dd t$.
  By the contractivity of the Fisher information along the Ornstein--Uhlenbeck semi-group and the assumption on $\mu^{\varepsilon}$, we have that
  \begin{equation*}
    \int_{0}^{1} \Fish{ \mu^{\varepsilon,\delta}_{t} \given \pi } \mathrm{d} t \leq \int_{0}^{1} \Fish{ \mu^{\varepsilon}_{t} \given \pi } \mathrm{d} t < \infty.
  \end{equation*}
  Thus, we have established that
  \begin{equation}\label{e:finiteness-w-d-log-rho}
    \int_{0}^{1} \int \abs{\diff \log \rho^{\varepsilon,\delta}_{t} w^{\varepsilon,\delta}_{t}} \dd (\pi \otimes m) \dd t < \infty.
  \end{equation}

  By definition, we have that
  \begin{equation*}
    w_{t}^{\varepsilon, \delta} = \e^{-t\delta} \mathsf{P}_{t\delta} w_{t}^{\varepsilon} - \delta \diff \rho^{\varepsilon,\delta}_{t}.
  \end{equation*}
  Using that $\paren{a-b}^{2} = a^{2} - 2(a-b)b - b^{2}$, we find that
  \begin{equation*}
    \abs{w_{t}^{\varepsilon,\delta}}^{2} = \e^{-2t\delta} \abs{\mathsf{P}_{t\delta}w^{\varepsilon}_{t}}^{2} - 2 \delta w_{t}^{\varepsilon,\delta} \diff \rho^{\varepsilon,\delta}_{t} - \delta^{2} \abs{\diff \rho^{\varepsilon,\delta}_{t}}^{2}.
  \end{equation*}
  Thus, for $t \in [0,1]$, expanding the square in this way in the definition of $\mathcal{L}$, we get
  \begin{equation}\label{e:expansion-lagrangian}
    \mathcal{L}(\mu^{\varepsilon,\delta}_{t}, \nu^{\varepsilon,\delta}_{t}) = \e^{-2t\delta} \mathcal{L}(\mathsf{P}_{t\delta}^{\star} \mu^{\varepsilon}_{t}, \mathsf{P}_{t\delta}^{\star}\nu^{\varepsilon}_{t}) - 2 \delta \int w_{t}^{\varepsilon, \delta} \diff \log \rho^{\varepsilon,\delta}_{t} \dd (\pi \otimes m) - \delta^{2} \Fish{ \mu^{\varepsilon,\delta}_{t} \given \pi },
  \end{equation}
  the first quantity is finite by \cref{t:contraction-lagrangian}, the second term is finite by \cref{e:finiteness-w-d-log-rho}, and the last term is finite by assumption.
  Using that $\mathcal{I} \geq 0$ and the contraction estimate \cref{t:contraction-lagrangian} for the Lagrangian yields:
  \begin{equation}\label{e:difference-lagrangian}
    \mathcal{L}(\mu^{\varepsilon,\delta}_{t}, \nu^{\varepsilon,\delta}_{t}) - \mathcal{L}(\mu^{\varepsilon}_{t}, \nu^{\varepsilon}_{t}) \leq \paren*{\e^{-2t\delta} -1} \mathcal{L}(\mu^{\varepsilon}_{t}, \nu^{\varepsilon}_{t}) - 2 \delta \int w^{\varepsilon,\delta}_{t} \diff \log \rho^{\varepsilon,\delta}_{t} \dd (\pi \otimes  m).
  \end{equation}

  \paragraph{The Ornstein--Uhlenbeck semi-group is an $EVI(0)$-gradient flow}
  By \cref{t:ce-entropy-production}, we find that
  \begin{equation}\label{e:entropy-production-epsilon}
    \frac{\dd}{\dd t} \Ent{ \mu_{t}^{\varepsilon,\delta} \given \pi } = \int \diff \log \rho^{\varepsilon,\delta}_{t} w^{\varepsilon,\delta}_{t} \dd (\pi \otimes m).
  \end{equation}
  Since $(\bar{\mu}^{\varepsilon}, \bar{\nu}^{\varepsilon})$ is a minimizer for $\mathcal{J}_{\varepsilon}(\mu, \xi)$, and since $(\bar{\mu}^{\varepsilon,\delta}, \bar{\nu}^{\varepsilon,\delta})$ is admissible for $\mathcal{J}_{\varepsilon}(\mathsf{P}^{\star}_{\delta}\mu, \xi)$
  \begin{equation}\label{e:difference-entropic-cost}
    \mathcal{J}_{\varepsilon}(\dou_{\delta}\mu, \xi) - \mathcal{J}_{\varepsilon}(\mu, \xi) \leq \int_{0}^{1} \paren*{\mathcal{L}(\mu^{\varepsilon,\delta}_{t}, \nu^{\varepsilon,\delta}_{t}) - \mathcal{L}(\mu^{\varepsilon}_{t}, \nu^{\varepsilon}_{t})} \mathrm{d} t +  \varepsilon \int_{0}^{1} \paren*{ \Fish{ \mu^{\varepsilon,\delta}_{t} \given \pi } - \Fish{ \mu^{\varepsilon}_{t} \given \pi } } \mathrm{d} t.
  \end{equation}
  The second term in the right-hand side is non-positive by the contractivity of the Fisher information along the Ornstein--Uhlenbeck semi-group.
  Since $\mathcal{L} \geq 0$ we can discard the first term in the right-hand side of \cref{e:difference-lagrangian}, and by \cref{t:ce-entropy-production}, this gives
  \begin{equation*}
    \mathcal{L}(\mu^{\varepsilon,\delta}_{t}, \nu^{\varepsilon,\delta}_{t}) - \mathcal{L}(\mu^{\varepsilon}_{t}, \nu^{\varepsilon}_{t}) \leq -2 \delta \frac{\mathrm{d}}{\mathrm{d} t} \Ent{ \mu^{\varepsilon,\delta}_{t} \given \pi }.
  \end{equation*}
  Reporting in \cref{e:difference-entropic-cost} yields
  \begin{equation}\label{e:non-infinitesimal-evi}
  \mathcal{J}_{\varepsilon}(\dou_{\delta}\mu, \xi) - \mathcal{J}_{\varepsilon}(\mu, \xi) \leq - 2 \delta \paren*{ \Ent{ \mathsf{P}^{\star}_{\delta} \mu \given \pi } - \Ent{ \xi \given \pi } }.
  \end{equation}
  In \cref{e:non-infinitesimal-evi}, we first let $\varepsilon \to 0$ and invoke \cref{t:gamma-convergence}, and then divide by $2 \delta$ and take $\limsup_{\delta \to 0^{+}}$.
  This yields
  \begin{equation*}
    \Ent{ \mu \given \pi } + \frac{\dd^{+}}{\dd s} \restriction_{s=0} \frac{1}{2} \mathcal{W}^{2}(\dou_{s}\mu, \xi) \leq \Ent{ \xi \given \pi }.
  \end{equation*}
  Using the semi-group property of $\dou$ this yields that $\dou$ is an $EVI(0)$-gradient flow of $\ent$.
  In particular, by \cite[Thm.\ 2.1]{DaneriSavare}, we have that $\dom \ent$ is geodesically convex.

  \paragraph{The Ornstein--Uhlenbeck semi-group is an $EVI(1)$-gradient flow}
  Now we repeat the argument above working directly with $\mathcal{W}^{2}$.
  By \cref{t:minimizing-curve}, take $(\bar{\mu}, \bar{\nu})$ realizing $\mathcal{W}^{2}(\mu, \xi)$.
  By the geodesic convexity of $\dom \ent$, we find that, for all $t \in [0,1]$, $\mu_{t} = \rho_{t} \pi$, and thus $\nu_{t} = w_{t} \pi \otimes m$ by \cref{t:lagrangian-absolutely-continuous}.
  Construct $(\bar{\mu}^{\delta}, \bar{\nu}^{\delta})$ as above.
  By \cref{t:wcurl-time-reparametrization} and the Cauchy--Schwarz inequality, we have that
  \begin{equation*}
    \mathcal{W}(\dou_{\delta} \mu, \xi) \leq \sqrt{ \int_{0}^{1} \e^{-2\delta t} \dd t \int_{0}^{1} \e^{2 \delta t} \mathcal{L}(\mu_{t}^{\delta}, \nu_{t}^{\delta}) \dd t }.
  \end{equation*}
  For all $t \in [0,1]$, $\mu_{t} \in \dom \ent$, thus $\mu^{\delta}_{t} \in \dom \fish$ by \cref{e:regularizing-ou}.
  Actually, by~\cref{e:production-entropy-ou}, we find that \cref{e:ce-bound-fish-action} in \cref{t:ce-entropy-production} is satisfied.
  In particular, we obtain an expression similar to \cref{e:expansion-lagrangian} for $(\bar{\mu}^{\delta}, \bar{\nu}^{\delta})$.
  Since $\fish \geq 0$, and using \cref{t:contraction-lagrangian}, we get that
  \begin{equation}\label{e:bound-w-delta}
    \begin{split}
      \mathcal{W}^{2}(\dou_{\delta} \mu, \xi) & \leq a(\delta) \bracket*{ \int_{0}^{1} \mathcal{L}(\dou_{t\delta} \mu_{t}, \dou_{t\delta} \nu_{t}) \dd t - 2 \delta \int_{0}^{1} \e^{2t\delta} \frac{\dd}{\dd t} \Ent{ \mu_{t}^{\delta} \given \pi } \dd t  }
                                            \\& \leq a(\delta) \bracket*{ \int_{0}^{1} \mathcal{L}(\mu_{t}, \nu_{t}) \dd t - 2 \delta \int_{0}^{1} \e^{2t\delta} \frac{\dd}{\dd t} \Ent{ \mu_{t}^{\delta} \given \pi } \dd t  },
    \end{split}
  \end{equation}
  where
  \begin{equation*}
    a(\delta) = \int_{0}^{1} \e^{-2t\delta} \dd t = \frac{1 - \e^{-2\delta}}{2\delta}.
  \end{equation*}
  By integration by parts, we find that
  \begin{equation}\label{e:ipp-entropy-delta}
    \int_{0}^{1} \e^{2t\delta} \frac{\dd}{\dd t} \Ent{ \mu_{t}^{\delta} \given \pi } \dd t = \e^{2\delta} \Ent{ \dou_{\delta}\mu \given \pi } - \Ent{ \xi \given \pi } - 2 \delta \int_{0}^{1} \e^{2t\delta} \Ent{ \mu_{t}^{\delta} \given \pi } \dd t.
  \end{equation}
  Substituting \cref{e:ipp-entropy-delta} in \cref{e:bound-w-delta}, and using that $(\bar{\mu}, \bar{\nu})$ is a minimizer for $\mathcal{W}^{2}(\mu, \xi)$, we get
  \begin{equation*}
    \begin{split}
      \mathcal{W}^{2}(\dou_{\delta}\mu, \xi) - \mathcal{W}^{2}(\mu, \xi) & \leq  (a(\delta) - 1)  \int_{0}^{1} \mathcal{L}(\mu_{t}, \nu_{t}) \dd t \\ 
                                                                         &+ 2 \delta a(\delta) \bracket*{  \Ent{ \xi \given \pi } - \e^{2\delta} \Ent{ \dou_{\delta}\mu \given \pi } } \\
                                                                         &+ 4 a(\delta) \delta^{2} \int_{0}^{1} \e^{2t\delta} \Ent{ \mu_{t}^{\delta} \given \pi } \dd t.
    \end{split} 
  \end{equation*}
  Dividing by $\delta$ and taking $\limsup_{\delta \to 0^{+}}$, and using the lower semi-continuity of $\mathcal{H}$ and that $\mathcal{H}$ decreases along $\mathsf{P}$ yields
  \begin{equation*}
    \frac{\dd^{+}}{\dd s} \restriction_{s=0} \mathcal{W}^{2}(\dou_{s}\mu, \xi) \leq -\mathcal{W}^{2}(\mu, \xi) + 2 (\Ent{ \xi \given \pi } - \Ent{ \mu \given \pi }),
  \end{equation*}
  which is exactly \cref{e:evi} for $s = 0$.
  This yields \cref{e:evi} for all $s$ by the semi-group property of $\dou$.
\end{proof}

We now draw two standard conclusions from the above Evolution Variation Inequality.

\begin{theorem}[{\cite[Thm.\ 2.1]{DaneriSavare}}]\label{t:entropy-geodesically-convex}
  The relative entropy is $1$-geodesically convex.
  Namely, let $\mu_{0}$ and $\mu_{1} \in \dom \ent$.
  Take $\set{ \mu_{t} : t \in [0,1] }$ a geodesic joining $\mu_{0}$ to $\mu_{1}$.
  Then,
  \begin{equation*}
    \Ent{\mu_{t} \given \pi} \leq (1-t) \Ent{\mu_{0} \given \pi} + t \Ent{\pi_{1} \given \pi} - \frac{t(1-t)}{2} \mathcal{W}^{2}(\mu_{0}, \mu_{1}), \qquad t \in [0,1].
  \end{equation*}
\end{theorem}

The \emph{descending slope} of $\ent$ at $\mu \in \dom \ent$ plays the role of the length of the gradient in our non-smooth setting: 
\begin{equation*}
  \abs{D^{-} \ent}(\mu) \coloneq \limsup_{\xi \to \mu} \frac{\paren{\Ent{ \mu \given \pi } - \Ent{ \xi \given \pi}}_{+}}{\mathcal{W}(\mu, \xi)}.
\end{equation*}

\begin{theorem}[{\cite[Prop.\ 4.6]{AGUser}}]\label{t:ede}
  The Ornstein--Uhlenbeck semi-group is a gradient flow of the entropy in the following sense:
  \begin{equation*}
     \abs{D^{-} \ent}(\dou_{t} \mu) = - \frac{\dd}{\dd t} \Ent{ \dou_{t} \mu \given \pi } = \Fish{ \dou_{t} \mu \given \pi }.
  \end{equation*}
\end{theorem}

We also have the following Poisson equivalent of the celebrated HWI inequality.
\begin{theorem}\label{t:hwi}
  Let $\mu \in \mathscr{P}_{1}(\config)$.
  Then:
  \begin{equation*}
    \Ent{ \mu \given \pi } \leq \mathcal{W}(\mu, \pi) \sqrt{\Fish{ \mu \given \pi }} - \frac{1}{2} \mathcal{W}^{2}(\mu, \pi).
  \end{equation*}
\end{theorem}

\begin{proof}
  The proof is identical to \cite[Thm.\ 7.3]{ErbarMaas}.
  The equivalent of \cite[Prop.\ 4.1]{ErbarMaas} in our setting is \cref{t:wcurl-derivative-fisher}.
\end{proof}

\section{Appendix}
\begin{lemma}\label{t:DualLLoc}
Let $(E,\mathscr{F},m)$ be a $\sigma$-finite measure space, and $\mathfrak{B}\subset \mathscr{F}$ be a family of measurable sets such that
\begin{enumerate}[$(a)$]
\item\label{i:l:Appendix:3} there exists an $m$-negligible set $N$ and a countable nested exhaustion $(B_n)_{n\in\nat}\subset \mathfrak{B}$ of $E\setminus N$ additionally such that for every $B\in\mathfrak{B}$ there exists $n\in\nat$ so that $B\subset B_n$.
\end{enumerate}
For $p\in [1,\infty)$ let
\[
L^p_{\mathrm{loc}}(E)= \set{f\in L^0(E): \norm{f\, 1_B}_{L^p}<\infty , B\in\mathfrak{B}}
\]
be endowed with the topology induced by the family of semi-norms
\[
\norm{f}_{p,B}= \norm{f \, 1_B}_{L^p}
\]
Then, $L^p_{\mathrm{loc}}(E)$ is a Fr\'echet space.
Further let $q$ be the Hölder conjugate exponent to $p$. Then, $T\in L^p_{\mathrm{loc}}(E)^*$ if and only if there exists $B\in\mathfrak{B}$ and $g_B\in L^q(E)$ with $g_B\equiv 0$ on $E\setminus B$ and such that
\[
T(f)=\int_E g_B f \dd m, \qquad f\in L^p_{\mathrm{loc}}(E).
\]
\begin{proof}
It is clear that $L^p_{\mathrm{loc}}(E)=L^p_{\mathrm{loc}}(E\setminus N)$, thus we may and will assume with no loss of generality that $N=\varnothing$.
By~\ref{i:l:Appendix:3} and monotonicity of the semi-norms $B\mapsto \norm{\emparg}_{p,B}$, the topology of $L^p_{\mathrm{loc}}(E)$ is induced by the countable family of semi-norms $\norm{\emparg}_{p,B_n}$ with $(B_n)_n$ as in~\ref{i:l:Appendix:3}; thus $L^p_{\mathrm{loc}}(E)$ is a Fréchet space.

Now, let $T\in L^p_{\mathrm{loc}}(E)^*$. By continuity of $T$ there exist $k\in\mathbb{N}$, constants $a_1,\dotsc,a_k>0$, and sets $B_1,\dotsc, B_k\in\mathfrak{B}$ so that $\abs{T(f)}\leq \sum_{i=1}^k a_i \norm{f}_{p,B_i}$ for all $f\in L^p_{\mathrm{loc}}(E)$.
Setting $a= \max_{i\leq k} a_i$, again by~\ref{i:l:Appendix:3} there exists $B\in\mathfrak{B}$ so that
\[
\abs{T(f)}\leq a k \norm{f}_{p,B} , \qquad f \in L^p_{\mathrm{loc}}(E).
\]
Consider the map $1_B\colon L^p_{\mathrm{loc}}(E)\to L^p(B)$.
By the above inequality, $\ker 1_B\subset \ker T$, hence $T=T_B\circ 1_B$ factors over some $T_B\in L^p(B)^*$.
Since $(B,m_B)$ is $\sigma$-finite, $T_B$ is represented by some function $g\in L^q(B)$ in the standard way.
Letting $g_B$ denote the extension by $0$ of $g\in L^q(B)$ to $E$, we have therefore that
\[
T(f)=T_B (1_B f)=\int_B g 1_B f \dd m_B= \int_E g_B f \dd m, \qquad f\in L^p_{\mathrm{loc}}(E).
\]
The reverse implication is straightforward.
\end{proof}
\end{lemma}

\begin{remark}
We note that the previous Lemma applies to every metric measure space $(E,d,m)$ when $\mathfrak{B}=\mathfrak{B}_0(E)$ and $m$ is finite on $\mathfrak{B}_0(E)$.
\end{remark}

\printbibliography[
heading=bibintoc,
title={Bibliography}
]

\end{document}